\documentclass[review]{elsarticle}
\usepackage{srcltx}
\usepackage{float}
\usepackage{pgf,tikz}
\usepackage{subcaption}
\usetikzlibrary{positioning}
\usetikzlibrary{arrows,patterns}
\usetikzlibrary{decorations.pathreplacing}
\usepackage{eurosym}
\usepackage{mathtools}
\usepackage{amsmath}
\usepackage{amsfonts}
\usepackage{amssymb}
\usepackage{amsthm}
\usepackage{graphicx}
\usepackage{mathrsfs}
\usepackage{xcolor}
\usepackage{exscale}
\usepackage{latexsym}
\usepackage{esvect}
\usepackage[font=small,skip=-10pt]{caption}
\renewcommand{\arraystretch}{2.2}
\newcommand\st{1.1}
\usepackage[overload]{empheq}
\definecolor{airforceblue}{rgb}{0.36, 0.54, 0.66}
\definecolor{auburn}{rgb}{0.43, 0.21, 0.1}
\usepackage[colorlinks,plainpages=true,pdfpagelabels,hypertexnames=true,colorlinks=true,pdfstartview=FitV,linkcolor=blue,citecolor=red,urlcolor=black,backref=page]{hyperref}
\PassOptionsToPackage{unicode}{hyperref}
\PassOptionsToPackage{naturalnames}{hyperref}
\definecolor{alizarin}{rgb}{0.82, 0.1, 0.26}
\usepackage{cancel}
\usepackage{enumerate}
\usepackage[shortlabels]{enumitem}
\usepackage{bookmark}
\usepackage{wasysym}
\usepackage{esint}
\usepackage[titletoc,toc,page]{appendix}
\usepackage[ddmmyyyy]{datetime}
\numberwithin{equation}{section}
\everymath{\displaystyle}
\usepackage[margin=2.2cm]{geometry}

\usepackage[capitalize,nameinlink]{cleveref}
\crefname{section}{Section}{Sections}
\crefname{subsection}{Subsection}{Subsections}
\crefname{condition}{Condition}{Conditions}
\crefname{hypothesis}{Hypothesis}{Hypothesis}
\crefname{assumption}{Assumption}{Assumptions}
\crefname{lemma}{Lemma}{Lemmas}
\crefname{claim}{Claim}{Claims}
\crefname{remark}{Remark}{Remarks}
\crefname{figure}{Figure}{Figures}

\newcommand{\tlcref}[1]{\textup{\labelcref{#1}}}

\crefformat{equation}{\textup{#2(#1)#3}}
\crefrangeformat{equation}{\textup{#3(#1)#4--#5(#2)#6}}
\crefmultiformat{equation}{\textup{#2(#1)#3}}{ and \textup{#2(#1)#3}}
{, \textup{#2(#1)#3}}{, and \textup{#2(#1)#3}}
\crefrangemultiformat{equation}{\textup{#3(#1)#4--#5(#2)#6}}%
{ and \textup{#3(#1)#4--#5(#2)#6}}{, \textup{#3(#1)#4--#5(#2)#6}}%
{, and \textup{#3(#1)#4--#5(#2)#6}}

\Crefformat{equation}{#2Equation~\textup{(#1)}#3}
\Crefrangeformat{equation}{Equations~\textup{#3(#1)#4--#5(#2)#6}}
\Crefmultiformat{equation}{Equations~\textup{#2(#1)#3}}{ and \textup{#2(#1)#3}}
{, \textup{#2(#1)#3}}{, and \textup{#2(#1)#3}}
\Crefrangemultiformat{equation}{Equations~\textup{#3(#1)#4--#5(#2)#6}}%
{ and \textup{#3(#1)#4--#5(#2)#6}}{, \textup{#3(#1)#4--#5(#2)#6}}%
{, and \textup{#3(#1)#4--#5(#2)#6}}

\crefdefaultlabelformat{#2\textup{#1}#3}
\newtheorem{theorem}{Theorem}[section]
\newtheorem{lemma}[theorem]{Lemma}
\newtheorem{corollary}[theorem]{Corollary}
\newtheorem{claim}[theorem]{Claim}
\newtheorem{proposition}[theorem]{Proposition}

\newtheorem{definition}[theorem]{Definition}
\newtheorem{remark}[theorem]{Remark}        
\newtheorem{assumption}[theorem]{Assumption}  
\numberwithin{equation}{section}

\def\YYint#1#2#3{{\setbox0=\hbox{$#1{#2#3}{\iint}$}
		\vcenter{\hbox{$#2#3$}}\kern-.50\wd0}}

\def\XXint#1#2#3{{\setbox0=\hbox{$#1{#2#3}{\int}$}
		\vcenter{\hbox{$#2#3$}}\kern-.50\wd0}}

\makeatletter
\def\namedlabel#1#2{\begingroup
	\def\@currentlabel{#2}%
	\label{#1}\endgroup
}
\makeatother
\makeatletter
\newcommand{\rmh}[1]{\mathpalette{\raisem@th{#1}}}
\newcommand{\raisem@th}[3]{\hspace*{-1pt}\raisebox{#1}{$#2#3$}}
\makeatother

\newcommand{\lsb}[2]{#1_{\rmh{-3pt}{#2}}}

\newcommand{\redref}[2]{\texorpdfstring{\protect\hyperlink{#1}{\textcolor{black}{(}\textcolor{red}{#2}\textcolor{black}{)}}}{}}
\newcommand{\redlabel}[2]{\hypertarget{#1}{\textcolor{black}{(}\textcolor{red}{#2}\textcolor{black}{)}}}

\newcommand{\descref}[2]{\hyperref[#1]{\textup{\textcolor{black}{(}\textcolor{blue}{\bf #2}\textcolor{black}{)}}}}

\newcommand{\descitemnormal}[2]{\item[#1]\label{#2}}
\newcommand{\descrefnormal}[2]{\hyperref[#1]{\textup{\textcolor{blue}{\bf #2}}}}

\makeatletter
\g@addto@macro\normalsize{%
	\setlength\abovedisplayskip{3pt}
	\setlength\belowdisplayskip{3pt}
	\setlength\abovedisplayshortskip{1pt}
	\setlength\belowdisplayshortskip{3pt}
}
\makeatother
\makeatletter
\def\ps@pprintTitle{%
	\let\@oddhead\@empty
	\let\@evenhead\@empty
	\def\@oddfoot{}%
	\let\@evenfoot\@oddfoot}
\makeatother
\usepackage{setspace}

\newcommand{\mcq}{\mathcal{Q}}

\newcommand{\mbcq}{{\mathcal{\bf Q}}_{\tilde{R}}}

\newcommand\RR{\mathbb{R}}

\newcommand\NN{\mathbb{N}}
\newcommand{\al}{\alpha}
\newcommand{\be}{\beta}

\newcommand{\de}{\delta}
\newcommand{\ve}{\varepsilon}

\newcommand{\tht}{\theta}

\newcommand{\om}{\omega}

\newcommand{\Om}{\Omega}

\newcommand{\La}{\Lambda}

\DeclareMathOperator{\dv}{div}
\DeclareMathOperator{\spt}{spt}

\DeclareMathOperator{\loc}{loc}

\DeclareMathOperator*{\esssup}{ess\,sup}
\DeclareMathOperator*{\essinf}{ess\,inf}
\DeclareMathOperator*{\tail}{{\textup{Tail}_\infty}}
\newcommand{\tailp}[1][p-1]{\textup{Tail}_\infty^{#1}}
\DeclareMathOperator*{\essosc}{ess\,osc}

\DeclareMathOperator*{\I}{\mathbf{I}}
\DeclareMathOperator*{\II}{\mathbf{II}}
\DeclareMathOperator*{\III}{\mathbf{III}}

\DeclareMathOperator*{\J}{\mathbf{J}}

\DeclareMathOperator*{\mC}{\mathbf{C}}
\allowdisplaybreaks

\newcommand{\htt}{\mathfrak{\mathring{t}}}
\newcommand{\mft}{\mathfrak{\tilde{t}}}

\newcommand{\mreta}{\mathring\eta}

\newcommand{\norm}[1]{\left|\hspace{-0.2mm}\left| #1 \right|\hspace{-0.2mm}\right|}
\newcommand{\abs}[1]{\left| #1\right|}

\newcommand{\lbr}[1][(]{\left#1}
\newcommand{\rbr}[1][)]{\right#1}
\newcommand{\overlabel}[2]{\overset{\text{\cref{#1}}}{#2}}
\newcommand{\overred}[3]{\overset{\redlabel{#1}{#2}}{#3}}

\newcommand{\txt}[1]{\qquad \text{#1} \quad}

\newcommand{\pa}{\partial}

\definecolor{Plum}{HTML}{89b02e}

\definecolor{Violet}{HTML}{58429B}

\definecolor{OliveGreen}{HTML}{0d8795}

\usepackage[titletoc]{appendix}

\DeclareMathOperator{\bsmu}{\boldsymbol\mu}
\DeclareMathOperator{\bsom}{\boldsymbol\omega}
\DeclareMathOperator{\bsy}{\boldsymbol  Y}
\DeclareMathOperator{\bsc}{\boldsymbol  C}
\newcommand{\omt}{ \Om \times (-T,T)}

\newcommand{\data}[1]{\{n,p,s,\Lambda#1\}}
\newcommand{\datanb}[1]{n,p,s,\Lambda#1}

\makeatletter
\def\ps@pprintTitle{%
	\let\@oddhead\@empty
	\let\@evenhead\@empty
	\def\@oddfoot{}%
	\let\@evenfoot\@oddfoot}
\makeatother
\newenvironment{notationlist}
{\begin{enumerate}[labelindent=*,
		leftmargin=*,
		label=(\bf{N}{\arabic*})
		]}
	{\end{enumerate}}

\begin{document}

\begin{frontmatter}
\title{Local H\"older regularity for bounded, signed solutions to nonlocal Trudinger equations}
\author[myaddress]{Karthik Adimurthi\tnoteref{thanksfirstauthor}}
\ead{karthikaditi@gmail.com and kadimurthi@tifrbng.res.in}
\tnotetext[thanksfirstauthor]{Supported by the Department of Atomic Energy,  Government of India, under
	project no.  12-R\&D-TFR-5.01-0520}

\address[myaddress]{Tata Institute of Fundamental Research, Centre for Applicable Mathematics,Bangalore, Karnataka, 560065, India}

\begin{abstract}
We prove local H\"older regularity for bounded and  sign-changing weak solutions to nonlocal Trudinger equations of the form
\[
(|u|^{p-2}u)_t + \text{P.V.} \int_{\RR^n} \frac{|u(x,t) - u(y,t)|^{p-2}(u(x,t)-u(y,t))}{|x-y|^{n+sp}} = 0,
\]
 in the  range $1< p<\infty$ and $s \in (0,1)$.  One of the main difficulties in extending the  local theory to the  nonlocal Trudinger equation is that when $0 \ll u \ll \infty$ locally, a crucial change of variable is unavailable in the nonlocal case due to the presence of the Tail term.  We adapt several  new ideas developed in the past few years to prove the required H\"older regularity.  
\end{abstract}

\end{frontmatter}
\tableofcontents

\section{Introduction}\label{sec1}

In this article, we prove local H\"older regularity for bounded, sign-changing solutions of nonlocal parabolic equations whose prototype structure is  of the form
\begin{equation}\label{maineq}
	\partial_t (|u|^{p-2}u) + \text{P.V.}\int_{\RR^n} |u(x,t)-u(y,t)|^{p-2}(u(x,t)-u(y,t))K(x,y,t)\,dy=0,
\end{equation} with $p\in (1,\infty)$ and $s \in (0,1)$. Moreover, for some universally fixed constant  $\La \geq 1$ and for almost every $x,y \in \RR^n$, we take $K:\RR^n\times\RR^n\times \RR \to [0,\infty)$ to be a symmetric measurable function satisfying
\begin{equation}\label{boundsonKernel}
	\frac{(1-s)}{\Lambda|x-y|^{n+sp}}\leq K(x,y,t)\leq \frac{(1-s)\Lambda}{|x-y|^{n+sp}}. 
\end{equation}

The main theorem we prove in this paper is the following.

\begin{theorem}\label{holderparabolic}
	Let $p\in(1,\infty)$, $s\in(0,1)$ and let $u\in L^p(I;W^{s,p}_{\text{loc}}(\Om))\cap L^\infty(I;L^2_{\text{loc}}(\Om))\cap L^\infty(I;L^{p-1}_{sp}(\RR^n))$ be any bounded, sign-changing weak solution to \cref{maineq}. Then $u$ is locally H\"older continuous in $\Om_T$, i.e., there exist constants $j_o > 1$, $\mathbf{C}_o>1$ and $\alpha\in (0,1)$ depending only on the data, such that the following holds:
	\[
	|u(0,0) - u(x,t)| \leq C_{\data{}} \frac{L}{R^{\alpha}} \lbr \max\{|x - 0|, |t - 0|^{\frac{1}{sp}}\}\rbr^{\alpha},
	\]
	for any $(x,t) \in B_{\frac12 R}(0)\times (-(\tfrac12 R)^{sp},0)$ and 
	\[
	L:= 2 \mathbf{C}_o^{\frac{sp}{p-1}j_o} \|u\|_{L^{\infty}(B_{8R}\times (-(8R)^{sp},0))} + \tail(|u|,8R,0,(-(8R)^{sp},0)).
	\]
	Here $R$ is any fixed radius and we assume $B_{8R}\times (-(8R)^{sp},0)\subset \Omega_T$ and $(0,0) \in \Omega_T$ is any fixed point.
\end{theorem}

\subsection{A brief history of the problem} 

Much of the early work on regularity of fractional elliptic equations in the case $p=2$ was carried out by Silvestre \cite{silvestreHolderEstimatesSolutions2006}, Caffarelli and Vasseur \cite{caffarelliDriftDiffusionEquations2010}, Caffarelli, Chan, Vasseur \cite{caffarelliRegularityTheoryParabolic2011} and Bass-Kassmann  \cite{bassHarnackInequalitiesNonlocal2005,bassHolderContinuityHarmonic2005, kassmannPrioriEstimatesIntegrodifferential2009}. An early formulation of the fractional $p$-Laplace operator was done by Ishii and Nakamura \cite{ishiiClassIntegralEquations2010} and existence of viscosity solutions was established. DiCastro, Kuusi and Palatucci extended the De Giorgi-Nash-Moser framework to study the regularity of the fractional $p$-Laplace equation in \cite{dicastroLocalBehaviorFractional2016}. The subsequent work of Cozzi \cite{cozziRegularityResultsHarnack2017} covered a stable (in the limit $s\to 1$) proof of H\"older regularity by defining a novel fractional De Giorgi class. An alternate proof of H\"older regularity based on the ideas of \cite{tilli} was given in \cite{adimurthiOlderRegularityFractional2022}. Explicit exponents for H\"older regularity with no coefficients were found in \cite{brascoHigherHolderRegularity2018,brascoContinuitySolutionsNonlinear2021}. 

 Regarding parabolic nonlocal equations of the form 
\begin{equation}\label{nonlocalp}
	\partial_t u + \text{P.V.}\int_{\RR^n} \frac{|u(x,t)-u(y,t)|^{p-2}(u(x,t)-u(y,t))}{|x-y|^{n+sp}}\,dy=0,
\end{equation}
bounded, weak solutions were shown to be H\"older continuous in \cite{adimurthiLocalHolderRegularity2022,liaoHolderRegularityParabolic2022}. 
Furthermore,  H\"older regularity of bounded, weak solutions for doubly nonlinear equations of the form
\begin{equation}\label{localtrudinger}
\partial_{t}(|u|^{p-2}u) - \Delta_p u = 0,
\end{equation}
was proved in \cite{bogeleinHolderRegularitySigned2021} (see also \cite{bogeleinOlderRegularitySigned2021,liaoHolderRegularitySigned2021}). In the special case of non-negative solutions, analogous H\"older regularity results were proved in the important paper \cite{MR2957656,MR3029403} which forms the framework for all future developments in understanding the regularity for \cref{localtrudinger}. 

 \emph{We note that \cref{nonlocalp} is translation invariant but not  scale invariant, whereas \cref{maineq} is scale invariant but not translation invariant. It is well known that failure of either translation invariance or scale invariance leads to fundamental difficulties when trying to obtain H\"older regularity. These difficulties are amplified further when dealing with nonlocal counterparts of the analogous local equations.}

Our approach to proving H\"older regularity for \cref{maineq} combines the  techniques from \cite{bogeleinHolderRegularitySigned2021} to handle the time terms and \cite{adimurthiLocalHolderRegularity2022,liaoHolderRegularityParabolic2022} to handle the nonlocal term. The main difficulty in studying \cref{maineq} was that the regularity theory for nonlocal $p$-parabolic equations was unknown until recently which played a crucial role when studying \cref{maineq}. As a consequence, the only known regularity results for \cref{maineq} in existing literature were local semi-continuity, local boundedness  and reverse H\"older inequality for globally bounded, strictly positive weak solutions proved in \cite{banerjeeLocalPropertiesSubsolution2021,banerjeeLowerSemicontinuityPointwise2021} and H\"older regularity was stated as an open problem.

\subsection{On historical development of intrinsic scaling}

The method of intrinsic scaling was developed by E.DiBenedetto when $p \geq 2$ in \cite{dibenedettoLocalBehaviourSolutions1986} to prove H\"older regularity for degenerate quasilinear parabolic equations of the form
\[
u_t - \dv |\nabla u|^{p-2} \nabla u = 0.
\]  A technical requirement of the proof was a novel logarithmic estimate which aids in the expansion of positivity.  Subsequently, the proof of H\"older regularity for the singular case  $p < 2$ was given in \cite{ya-zheLocalBehaviorSolutions1988} by switching the scaling from time variable to the space variable. These results are collected in E.DiBenedetto's treatise \cite{dibenedettoDegenerateParabolicEquations1993}. After a gap of several years, E.DiBenedetto, U.Gianazza and V.Vespri \cite{dibenedettoHarnackEstimatesQuasilinear2008, dibenedettoHarnackInequalityDegenerate2012} were able to develop fundamentally new ideas involving an exponential change of variable in time to prove Harnack's inequality for the parabolic $p$-Laplace equations. Crucially, this proof relies on expansion of positivity estimates and does not involve logarithmic test functions. Then, a new proof of H\"older regularity was given with a more geometric flavour in \cite{gianazzaNewProofHolder2010}. This theory was extended to generalized parabolic $p$-Laplace equations with Orlicz growth in \cite{hwangHolderContinuityBounded2015,hwangHolderContinuityBounded2015a}. Subsequently, using the exponential change of variables, a new proof of H\"older regularity was given in \cite{liaoUnifiedApproachHolder2020}.

\subsection{Our approach to proving H\"older regularity to \texorpdfstring{\cref{maineq}}.}

   The general strategy to proving H\"older regularity is based on the framework developed in \cite{bogeleinHolderRegularitySigned2021}, though our proof requires additional difficulties to be overcome. In particular, we encounter two problems:
   \begin{enumerate}[(i)]
   	\item In the `close to zero' case, we can iterate the analogous version of \cref{Prop:1:1} to obtain pointwise information at the top of the reference domain. This is similar to the local version, though we need to also assume a further bound on the tail, see \cref{Section4}. 
   	\item In the `away from zero case' (see \cref{section6}), the local equation allows to perform a change of variable of the form given in \cref{defvw}. Unfortunately, in the nonlocal equation, such a change of variable is only available locally and cannot be performed to handle the tail term. In order to overcome this, we have to obtain a mixed  energy estimate that retains the tail in the unchanged variable (see \cref{lemma6.3}) and prove H\"older regularity for this class of estimates.
   	\item We need the theory of H\"older regularity of weak solutions of  
   	\begin{equation*}
   		\partial_t u + \text{P.V.}\int_{\RR^n} \frac{|u(x,t)-u(y,t)|^{p-2}(u(x,t)-u(y,t))}{|x-y|^{n+sp}}\,dy=0,
   	\end{equation*}
   	obtained in \cite{adimurthiLocalHolderRegularity2022,liaoHolderRegularityParabolic2022} in order to obtain H\"older regularity in the `away from zero' case. It is not clear how to adapt the ideas from \cite{adimurthiLocalHolderRegularity2022} since the change of variables in \cref{section6} is not available for the tail, although we expect this can be made to work with delicate modifications. In order to overcome this, we instead follow the ideas from \cite{liaoHolderRegularityParabolic2022} and prove H\"older regularity for the mixed type energy estimate obtained in \cref{lemma6.3}. 
   	\item The H\"older regularity and covering argument for the mixed type equation in \cref{lemma6.3} is quite delicate and requires several modifications of existing ideas.
   \end{enumerate}

\begin{remark}
	{In this paper, we assume that solutions to the nonlocal equation are locally bounded. }
\end{remark}

\subsection{Notations}
We begin by collecting the standard notation that will be used throughout the paper:
\begin{notationlist}
	\item\label{not1} We shall denote $n$ to be the space dimension and by $z=(x,t)$ to be  a point in $ \RR^n\times (0,T)$.  
	\item\label{not2} We shall alternately use $\dfrac{\partial f}{\partial t}$,$\partial_t f$,$f'$ to denote the time derivative of $f$.
	\item\label{not3} Let $\Omega$ be an open bounded domain in $\mathbb{R}^n$ with boundary $\partial \Omega$ and for $0<T\leq\infty$,  let $\Omega_T\coloneqq \Omega\times (0,T)$. 
	\item\label{not4} We shall use the notation
	\begin{equation*}
		\begin{array}{ll}
			B_{\varrho}(x_0)=\{x\in\RR^n:|x-x_0|<\varrho\}, &
			\overline{B}_{\varrho}(x_0)=\{x\in\RR^n:|x-x_0|\leq\varrho\},\\
			I_{\tht}(t_0)=\{t\in\RR:t_0-\tht<t<t_0\},
			&Q_{\varrho,\tht}(z_0)=B_{\varrho}(x_0)\times I_\tht(t_0).
		\end{array}
	\end{equation*} 
\item\label{not5} We shall also use the notation $Q_R^\theta(x_0,t_0) := B_R(x_0) \times (t_0 - \theta R^{sp}, t_0)$.  
	\item\label{not7} Integration with respect to either space or time only will be denoted by a single integral $\int$ whereas integration on $\Om\times\Om$ or $\RR^n\times\RR^n$ will be denoted by a double integral $\iint$. 
	\item\label{not8} We will  use $\iiint dx\,dy\,dt$ to denote integral over $\RR^n \times \RR^n \times (0,T)$. More specifically, we will use the notation $\iiint_{\mcq}$ and $\iiint_{I \times B}$ to denote the integral over $\iiint_{I\times B \times B }dx\,dy\,dt$ where $\mcq = B\times I$ with the order of integration made precise by the order of the differentials.
	\item\label{not9} The notation $a \lesssim b$ is shorthand for $a\leq C b$ where $C$ is a constant which only depends on $\datanb{}$ and could change from line to line.
	\item\label{not10} Given a domain $Q=B\times(0,T)$, we denote it's parabolic boundary by $\pa_p Q = (\pa B \times (0,T)) \cup (B \times \{t=0\})$.
	\item\label{not11} For any fixed $t,k\in\RR$ and set $\Om\subset\RR^n$, we denote $A_{\pm}(k,t) := \{x\in \Om: (u-k)_{\pm}(\cdot,t) > 0\}$; for any ball $B_r$ we write $A_{\pm}(k,t) \cap (B_{r}\times I) =: A_{\pm}(k,t,r)$. 
	\item\label{not12} For any ball $B \subset \RR^n$, we denote $\mathcal{C}_B:=(B^c\times B^c)^c=\left(B\times B\right) \cup \left( B\times(\RR^n\setminus B)\right) \cup \left((\RR^n\setminus B)\times B \right)$.
	\item\label{not13} We shall denote a constant to depend on data if it depends on $\data{}$. Many a times, we will ignore writing this dependence explicitly for the sake or brevity.
\end{notationlist}

\subsection{Function spaces}
Let $1<p<\infty$, we denote by $p'=p/(p-1)$ the conjugate exponent of $p$. Let $\Om$ be an open subset of $\RR^n$, we  define the {\it Sobolev-Slobodeki\u i} space, which is the fractional analogue of Sobolev spaces as follows:
\begin{equation*}
	W^{s,p}(\Om):=\left\{ \psi\in L^p(\Omega): [\psi]_{W^{s,p}(\Om)}<\infty \right\},\qquad \text{for} \  s\in (0,1),
\end{equation*} where the seminorm $[\cdot]_{W^{s,p}(\Om)}$ is defined by 
\begin{equation*}
	[\psi]_{W^{s,p}(\Om)}:=\left( \iint_{\Om\times\Om} \frac{|\psi(x)-\psi(y)|^p}{|x-y|^{n+sp}}\,dx\,dy \right)^{\frac 1p}.
\end{equation*}
The space when endowed with the norm $\norm{\psi}_{W^{s,p}(\Om)}=\norm{\psi}_{L^p(\Om)}+[\psi]_{W^{s,p}(\Om)}$ becomes a Banach space. The space $W^{s,p}_0(\Om)$ is the subspace of $W^{s,p}(\RR^n)$ consisting of functions that vanish outside $\Om$. We will use the notation $W^{s,p}_{(u_0)}(\Om)$ to denote the space of functions in $W^{s,p}(\RR^n)$ such that $u-u_0\in W^{s,p}_0(\Om)$.

Let $I$ be an interval and let $V$ be a separable, reflexive Banach space, endowed with a norm $\norm{\cdot}_V$. We denote by $V^*$ to be its topological dual space. Let $v$ be a mapping such that for a.e. $t\in I$, $v(t)\in V$. If the function $t\mapsto \norm{v(t)}_V$ is measurable on $I$, then $v$ is said to belong to the Banach space $L^p(I;V)$ provided $\int_I\norm{v(t)}_V^p\,dt<\infty$. It is well known that the dual space $L^p(I;V)^*$ can be characterized as $L^{p'}(I;V^*)$.

Since the boundedness result requires some finiteness condition on the nonlocal tails, we define the tail space for some $m >0$ and $s >0$ as follows:
\begin{equation*}
	L^m_{s}(\RR^n):=\left\{ v\in L^m_{\text{loc}}(\RR^n):\int_{\RR^n}\frac{|v(x)|^m}{1+|x|^{n+s}}\,dx<+\infty \right\}.
\end{equation*}
Then a nonlocal tail is defined by 
\begin{equation*}
	\text{Tail}_{m,s,\infty}(v;x_0,R,I):=\text{Tail}_\infty(v;x_0,R,t_0-\tht,t_0):=\sup_{t\in (t_0-\tht, t_0)}\left( R^{sm}\int_{\RR^n\setminus B_R(x_0)} \frac{|v(x,t)|^{m-1}}{|x-x_0|^{n+sm}}\,dx \right)^{\frac{1}{m-1}},
\end{equation*} where $(x_0,t_0)\in \RR^n\times (-T,T)$ and the interval $I=(t_0-\tht,t_0)\subseteq (-T,T)$. From this definition, it follows that for any $v\in L^\infty(-T,T;L^{m-1}_{sm}(\RR^n))$, there holds $\text{Tail}_{m,s,\infty}(v;x_0,R,I)<\infty$. 

\subsection{Definitions}

For $k, l\in\RR$ we define 
\begin{equation}\label{defgpm}
	\mathfrak g_\pm (l,k):=\pm (p-1)\int_{k}^{l}|s|^{p-2}(s-k)_\pm\,ds.
\end{equation}
Note that $\mathfrak g_\pm (l,k)\ge 0$, then we have the following lemma from \cite[Lemma 2.2]{bogeleinHolderRegularitySigned2021}:
\begin{lemma}\label{lem:g}
	There exists a constant $\bsc =\bsc (p)$ such that,
	for all $l,k\in\RR$, the following inequality holds true:
	\begin{equation*}
		\tfrac1{\bsc } \lbr|l| + |k|\rbr^{p-2}(l-k)_\pm^2
		\le
		\mathfrak g_\pm (l,k)
		\le
		\bsc  \lbr|l| + |k|\rbr^{p-2}(l-k)_\pm^2
	\end{equation*}
\end{lemma}
Now, we are ready to state the definition of a weak sub(super)-solution.

\begin{definition}
	A function $u\in L^p_{\loc}(I;W^{s,p}_{\text{loc}}(\Om))\cap L^\infty_{\loc}(I;L^p_{\text{loc}}(\Om))\cap L^\infty_{\loc}(I;L^{p-1}_{sp}(\RR^n))$ is said to be a local weak sub(super)-solution to  if for any closed interval $[t_1,t_2]\subset I$ and a compact set $B\subseteq \Om$, the following holds:
	\begin{multline*}
		\int_{B} (|u|^{p-2}u)(x,t_2)\phi(x,t_2)\,dx - \int_{B} (|u|^{p-2}u)(x,t_1)\phi(x,t_1)\,dx - \int_{t_1}^{t_2}\int_{B} (|u|^{p-2}u)(x,t)\partial_t\phi(x,t)\,dx\,dt\\
		+\iiint_{\mathcal{C}_B}\,K(x,y,t)|u(x,t)-u(y,t)|^{p-2}(u(x,t)-u(y,t))(\phi(x,t)-\phi(y,t))\,dy\,dx\,dt
		\leq (\geq)0 ,
	\end{multline*} for all $\phi\in L^p_{\loc}(I,W^{s,p}_{\loc}(\Omega))\cap W^{1,p}_{\loc}(I,L^p(\Om))$ and the spatial support of $\phi$ is contained in $\sigma B$ for some $\sigma \in (0,1)$.
\end{definition}

\subsection{Auxiliary Results}
We collect the following standard results which will be used in the course of the paper. We begin with the Sobolev-type inequality~\cite[Lemma 2.3]{dingLocalBoundednessHolder2021}.

\begin{theorem}\label{fracpoin}
	Let $t_2>t_1>0$ and suppose $s\in(0,1)$ and $1\leq p<\infty$. Then for any $f\in L^p(t_1,t_2;W^{s,p}(B_r))\cap L^\infty(t_1,t_2;L^2(B_r))$, we have
	\begin{equation*}
		\begin{array}{rc@{}l}
			\int_{t_1}^{t_2}\fint_{B_r}|f(x,t)|^{p\left(1+\frac{2s}{N}\right)}\,dx\,dt
			& \apprle_{n,s,p} &  \left(r^{sp}\int_{t_1}^{t_2}\int_{B_r}\fint_{B_r}\frac{|f(x,t)-f(y,t)|^p}{|x-y|^{n+sp}}\,dx\,dy\,dt+\int_{t_1}^{t_2}\fint_{B_r}|f(x,t)|^p\,dx\,dt\right) \\
			&&\quad \times\left(\sup_{t_1<t<t_2}\fint_{B_r}|f(x,t)|^2\,dx\right)^{\frac{sp}{N}}.
		\end{array}
	\end{equation*}  
\end{theorem}


%
Let us recall the following simple algebraic lemma:
\begin{lemma}\label{alg_lem}
	Let $c_1, c_2 \in (0,\infty)$ and supposed $c_1 \leq \al,\be \leq c_2$ be two numbers with $\alpha \geq \beta$. Then for any $1\leq  q< \infty$, we have 
	\[
	q\frac{ c_1^q}{c_2} (\al-\be) \leq \al^q - \be^q \leq q \frac{c_2^{q}}{c_1} (\al-\be).
	\]
	In the case $0 < q < 1$, we instead have
	\[
	q\frac{ c_1}{c_2^{2-q}} (\al-\be) \leq \al^q - \be^q \leq q \frac{c_2}{c_1^{2-q}} (\al-\be).
	\]
\end{lemma}
\begin{proof}
Let $\be \in [c_1,c_2]$ be fixed and denote $\al = \be+x$ for $x \in [0,c_2-\be]$. Define the functions 
\[
\text{Case}\,\, q \geq 1:\begin{cases}
	f_1(x) = (\be+x)^q - \be^q - q\tfrac{ c_1^q}{c_2}x, \\
	f_2(x) = (\be+x)^q - \be^q - q \tfrac{c_2^{q}}{c_1}x,
\end{cases}
\text{Case}\,\, 0<q<1:\begin{cases}
	h_1(x) = (\be+x)^q - \be^q - q\tfrac{ c_1}{c_2^{2-q}}x, \\
	h_2(x) = (\be+x)^q - \be^q -q \tfrac{c_2}{c_1^{2-q}}x.
\end{cases}
\]
 Then, we see that for $x \in [0,c_2-\be]$, there holds $f_1(0) = f_2(0) = h_1(0)= h_2(0)= 0$ along with  $f_1'(x) \geq 0$, $f_2'(x) \leq 0$, $h_1'(x) \geq 0$ and  $h_2'(x) \leq 0$ which proves the lemma.
\end{proof}

Finally, we recall the following well known lemma concerning the geometric convergence of sequence of numbers (see \cite[Lemma 4.1 from Section I]{dibenedettoDegenerateParabolicEquations1993} for the details): 
\begin{lemma}\label{geo_con}
	Let $\{Y_n\}$, $n=0,1,2,\ldots$, be a sequence of positive number, satisfying the recursive inequalities 
	\[ Y_{n+1} \leq C b^n Y_{n}^{1+\alpha},\]
	where $C > 1$, $b>1$, and $\alpha > 0$ are given numbers. If 
	\[ Y_0 \leq  C^{-\frac{1}{\alpha}}b^{-\frac{1}{\alpha^2}},\]
	then $\{Y_n\}$ converges to zero as $n\to \infty$. 
\end{lemma}
 Let us recall the following algebraic lemma from \cite[Lemma 4.1]{cozziRegularityResultsHarnack2017}: 
\begin{lemma}\label{pineq1}
	Let $p\geq 1$ and $a,b\geq 0$, then for any $\tht\in [0,1]$,  the following holds:
	\begin{equation*}
		(a+b)^p-a^p\geq \tht p a^{p-1}b + (1-\tht)b^p.
	\end{equation*}
\end{lemma}

Let us recall the following algebraic lemma from \cite[Lemma 4.3]{cozziRegularityResultsHarnack2017}:
\begin{lemma}\label{pineq3}
	Let $p\geq 1$ and $a\geq b\geq 0$, then for any $\ve>0$, the following holds:
	\begin{equation*}
		a^p-b^p\leq \ve a^p+\left(\frac{p-1}{\ve}\right)^{p-1}(a-b)^p.
	\end{equation*}
\end{lemma}

%
%
%

\section{Preliminary Estimates}
In this section, we recall some important estimates. The first one is a standard energy estimate,   the proof of which follows by combining \cite[Proposition 3.1]{bogeleinHolderRegularitySigned2021} along with \cite[Theorem 3.1]{adimurthiLocalHolderRegularity2022}.
\begin{proposition}\label{Prop:energy}
	Let $u$ be a  local weak sub(super)-solution to in $\Om_T$ and $z_o = (x_o,t_o)$ be a fixed point.
	There exists a constant $\bsc  >0$ depending only on the data such that
 	for all cylinders $\mcq_{R,S}=B_R(x_o)\times (t_o-S,t_o)\Subset E_T$,
 	every $k\in\RR$, and every non-negative, piecewise smooth cut-off function
 	$\zeta$ vanishing on $\partial B_{R}(x_o)\times (t_o-S,t_o)$,  there holds
\begin{multline*}
	\esssup_{t_o-S<t<t_o}\int_{B_R(x_o)\times\{t\}}	
	\zeta^p\mathfrak g_\pm (u,k)\,dx +
	\iint_{\mcq_{R,S}(z_o)} (u-k)_{\pm}(x,t)\zeta^p(x,t)\int_{B_R(x_o)}\frac{(u-k)_{\mp}^{p-1}(y,t)}{|x-y|^{n+sp}}\,dy\,dx\,dt 
	\\
	 +\int_{t_o-S}^{t_o}\iint_{B_R(x_o)\times B_R(x_o)}|(u-k)_{\pm}(x,t)\zeta(x,t)-(u-k)_{\pm}(y,t)\zeta(y,t)|^p\,d\mu\,dt\\ 
	 \begin{array}{rcl}
	&\leq&
	\boldsymbol{C} \int_{t_o-S}^{t_o}\iint_{B_R(x_o)\times B_R(x_o)} \frac{\max\{(u-k)_{\pm}(x,t),(u-k)_{\pm}(y,t)\}^{p}|\zeta(x,t)-\zeta(y,t)|^p}{|x-y|^{n+sp}}\,dx\,dy\,dt
	\\
	&&+
	\iint_{\mcq_{R,S}(z_o)}\mathfrak g_\pm (u,k)|\partial_t\zeta^p| \,dx\,dt
	+\int_{B_R(x_o)\times \{t_o-S\}} \zeta^p \mathfrak g_\pm (u,k)\,dx 
	\\
	&&+ \,\boldsymbol{C}\lbr  \underset{\stackrel{t \in (t_o-S,t_o)}{x\in \spt \zeta}}{\esssup}\,\int_{\RR^n \setminus B_R(x_o)}\frac{(u-k)_{\pm}^{p-1}(y,t)}{|x-y|^{n+sp}}\,dy\rbr\iint_{(t_o-S,t_o)\times B_R(x_o)} (u-k)_{\pm}(x,t)\zeta^p(x,t)\,dx\,dt,
	\end{array}
\end{multline*}
where we recall the definition of $\mathfrak{g}_{\pm}$ from  \cref{defgpm}.
\end{proposition}

\subsection{Shrinking Lemma}

One of the main difficulties we face when dealing with regularity issues for nonlocal equations is the lack of a corresponding isoperimetric inequality for $W^{s,p}$ functions. Indeed, since such functions can have jumps, a generic isoperimetric inequality seems out of reach at this time (see \cite{cozziFractionalGiorgiClasses2019, adimurthiOlderRegularityFractional2022} ) One way around this issue is to note that since we are working with solutions of an equation, which we expect to be continuous and hence have no jumps, we could try and cook up an isoperimetric inequality for solutions. Such a strategy turns out to be feasible due to the presence of the ``good term'' or the isoperimetric term in the Caccioppoli inequality. The following lemma can be found in \cite[Lemma 3.3]{adimurthiLocalHolderRegularity2022}.

\begin{lemma}\label{lem:isop}
	Let $k<l<m$ be arbitrary levels and $A \geq 1$. Then,
	\[
	(l-k)(m-l)^{p-1}\abs{[u>m]\cap B_{\varrho}}\abs{[u<k]\cap B_{\varrho}} \leq 
	C\varrho^{n+sp}\int_{B_{\varrho}} (u-l)_{-}(x)\int_{B_{A\varrho}}\frac{(u-l)_{+}^{p-1}(y)}{|x-y|^{n+sp}}\,dy\,dx,
	\]
	where $C = C(n,s,p,A)>0$. 
\end{lemma}

We can apply \cref{lem:isop} to get the following shrinking lemma whose proof can be found in \cite[Lemma 3.4]{adimurthiLocalHolderRegularity2022}.
\begin{lemma}\label{lem:shrinking}
	Let $u$ be a given function and suppose that for some level $m$, some constant $\nu \in (0,1)$ and all time levels $\tau$ in some interval $I$, we have
	\[
	|[u(\cdot,\tau)>m]\cap B_{\varrho}| \geq \nu|B_{\varrho}|,
	\]
	and we can arrange that for some $A \geq 1$,  the following is also satisfied:
	\begin{equation*}
		\iint_{I\times B_{\varrho}} (u-l)_{-}(x,t)\int_{B_{A\varrho}}\frac{(u-l)_{+}^{p-1}(y,t)}{|x-y|^{n+sp}}\,dy\,dx\,dt \leq {\mC}_1\frac{l^p}{\varrho^{sp}}|Q|,
	\end{equation*}
	where 
	\[
	l = \frac{m}{2^{j}}, \qquad j\geq 1 \qquad \text{ and } \qquad Q := B_{\varrho} \times I,
	\] 
	then the following conclusion holds:
	\[
	\left|\left[u<\frac{m}{2^{j+1}}\right]\cap Q\right| \leq \left(\frac{C}{2^{j}-1}\right)^{p-1}|Q|,
	\]
	where $C = ({\mC}_1,A,n,\nu) >0$.
\end{lemma}
\subsection{Tail Estimates}\label{sec:tail}
In this section, we want to outline how we estimate the tail term and we shall refer to this section whenever we make a similar calculation.

For any level $k = \bsmu^{-}+ M$ with $M > 0$, we want to estimate
\[
\underset{\stackrel{t \in I;}{x\in \spt \zeta}}{\esssup}\int_{\RR^n \setminus B_{\varrho}(y_o)}\frac{(u-k)_{-}^{p-1}(y,t)}{|x-y|^{n+sp}}\,dy,
\]
where $I$ is some time interval and $\zeta$ is a cut-off function supported in $B_{\vartheta\varrho}$ for some $\vartheta \in (0,1)$. In such case, we have
\[
|y-y_o| \leq |x-y|\left(1+\frac{|x-y_o|}{|x-y|}\right)\leq  |x-y|\left(1+\frac{\vartheta}{(1-\vartheta)}\right),
\]
using which, we get
\[
\underset{\stackrel{t \in I;}{ x\in \spt \zeta}}{\esssup}\int_{\RR^n \setminus B_{\varrho}(y_o)}\frac{(u-k)_{-}^{p-1}(y,t)}{|x-y|^{n+sp}}\,dy 
\leq \frac{1}{(1-\vartheta)^{n+sp}}\underset{t \in I}{\esssup}\int_{\RR^n \setminus B_{\varrho}(y_o)}\frac{(u-k)_{-}^{p-1}(y,t)}{|y-y_o|^{n+sp}}\,dy. 
\]
In the local case, we could always estimate $(u-k)_{-} \leq k$ because we take $u \geq 0$ locally. However, in the $\tail$ term,  we are on the complement of a cube and so unless we make a global boundedness assumption (for e.g. $u \geq 0$ in full space), the best we can do is
\[
(u-k)_{-} \leq u_{-} + k,
\]
which leads us to the next estimate
\[
\underset{t \in I}{\esssup}\int_{\RR^n \setminus B_{\varrho}(y_o)}\frac{(u-k)_{-}^{p-1}(y,t)}{|y-y_0|^{n+sp}}\,dy  \leq C(p)\frac{M^{p-1}}{\varrho^{sp}} +  C(p)\,\underset{t \in I}{\esssup}\int_{\RR^n \setminus B_{\varrho}(y_o)}\frac{(u -\bsmu^-)_-^{p-1}(y,t)}{|y-y_o|^{n+sp}}\,dy.
\]
Putting together the above estimates yield
\[
\underset{\stackrel{t \in I;}{ x\in \spt \zeta}}{\esssup}\int_{\RR^n \setminus B_{\varrho}(y_o)}\frac{(u-k)_{-}^{p-1}(y,t)}{|x-y|^{n+sp}}\,dy  \leq \frac{C(p)}{\varrho^{sp}}\left[M^{p-1}+\tailp((u -\bsmu^-)_-;y_o,\varrho,I)\right].
\]
Finally we usually want an estimate of the form
\[
\underset{\stackrel{t \in I;}{ x\in \spt \zeta}}{\esssup}\int_{\RR^n \setminus B_{\varrho}(y_o)}\frac{(u-k)_{-}^{p-1}(y,t)}{|x-y|^{n+sp}}\,dy \leq C\frac{M^{p-1}}{\varrho^{sp}},
\]
and so we impose the condition that the following is satisfied:
\[
\tailp((u -\bsmu^-)_-;y_0,\varrho,I) \leq M^{p-1}.
\]
This is the origin of the various Tail alternatives. 

\begin{remark}
	We have an analogous estimate regarding $k = \boldsymbol{\mu}^+-M$ which corresponds to subsolutions.
\end{remark}

\section{Expansion of Positivity}
For a compact ball $B_R\subset\RR^n$ of radius $R$ and
 a cylinder $\mathbf{Q} := B_R\times(T_1,T_2]\subset \Om_T$, 
we introduce numbers $\bsmu^{\pm}$ and $\bsom$ satisfying
\begin{equation*}
	\bsmu^+\geq  \esssup_{\mathbf{Q}}u, \qquad \bsmu^-\leq \essinf_{\mathbf{Q}} u \qquad \text{ and } 
	\qquad \bsom \geq \bsmu^+ - \bsmu^-.	
\end{equation*}
We also assume $(x_o,t_o)\in \mcq$, such that the forward
cylinder 
\begin{equation}\label{Eq:1:5}
	B_{8\varrho}(x_o)\times\left(t_o-(8\varrho)^{sp},t_o+(8\varrho)^{sp}\right]\subset\mathbf{Q}.
\end{equation}
Next we state our main proposition of this section.

\begin{proposition}\label{Prop:1:1}
	Let $u$ be a locally bounded, local, weak sub(super)-solution to \cref{maineq} in $\omt$. With notation as in \cref{Eq:1:5}, suppose there exists constants  $M>0$ and $\alpha \in(0,1)$ such that the following is satisfied:
	\begin{equation*}
		\left|\left\{\pm\lbr\bsmu^{\pm}-u(\cdot, t_o)\rbr\geq M\right\}\cap B_\varrho(x_o)\right|
		\geq
		\alpha \big|B_\varrho\big|.
	\end{equation*}
Then there exist constants $\xi\in(0,1)$, $\delta\in(0,1)$ and $\eta\in(0,1)$ depending only on the data and $\alpha$,
such that, if the following is additionally satisfied 
\begin{equation}\label{Prop:1:eq}
 \tail((u-\mu^{\pm})_{\pm};x_0,\varrho,(t_o,t_o+\delta\varrho^{sp}]) \leq \eta M \quad \text{and} \quad \abs{\bsmu^{\pm}}\leq \xi M,
\end{equation}
then we have the following conclusion
\begin{equation*}
	\pm\lbr\bsmu^{\pm}-u\rbr\geq\eta M
	\quad
	\mbox{a.e.~in $B_{2\varrho}(x_o)\times\left(t_o+\tfrac{\de}{2}\varrho^{sp},t_o+\delta\varrho^{sp}\right],$}
\end{equation*}
where 
\[	\xi	:= \left\{
	{\def\arraystretch{1} \begin{array}{ll}
	2\eta, & \mbox{if $p>2$,} \\
	8, & \mbox{if $1<p\le 2$.}
	\end{array}}
	\right.
\]
\end{proposition}
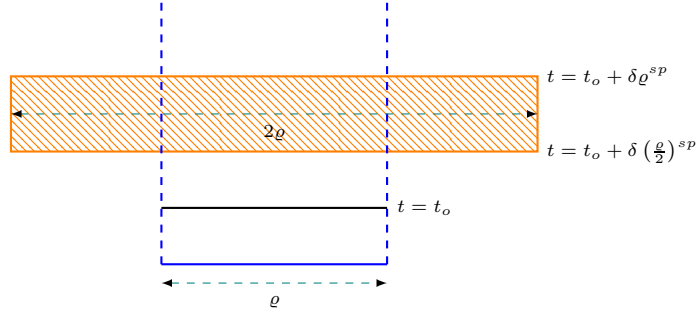
\begin{figure}[ht]\label{figureProp1.1}
	\begin{center}
		\begin{tikzpicture}[remember picture,scale=0.5,>=latex]
			\coordinate  (O) at (0,0);
			\draw[thick, draw=orange,pattern=north west lines, pattern color=orange, opacity=0.3] (-7,-2) rectangle (7,0);
			
			\draw[thick, draw=blue] (-3,-5) -- (3,-5);
			\draw[thick, draw=blue, dashed] (-3,-5) -- (-3,2);
			\draw[thick, draw=blue, dashed] (3,-5) -- (3,2);
			
			
			\draw[draw=black,thick] (-3,-3.5) -- (3,-3.5);
			\draw[draw=teal, dashed, <->] (-3,-5.5) -- (3,-5.5);
			\draw[draw=teal, dashed, <->] (-7,-1) -- (7,-1);
			\node  at (0,-6) {\scriptsize $\varrho$};
			\node  at (0,-1.5) {\scriptsize $2\varrho$};
			\node  at (4,-3.5) {\scriptsize $t=t_o$};
			\node  [anchor=west] at (7,-2) {\scriptsize $t=t_o+ \de \lbr\tfrac{\varrho}{2}\rbr^{sp}$};
			\node  [anchor=west] at (7,0) {\scriptsize $t=t_o+ \de \varrho^{sp}$};
%
%
		\end{tikzpicture}
	\end{center}
	\caption{Measure to pointwise bound in \cref{Prop:1:1}}
\end{figure}

\begin{remark}\label{rmk:4.2}
	Let $B_R(x_o) \times (t_o, t_o+S)$ be some cylinder and let $\nu^\pm$ be two numbers satisfying
	\[
	\nu^+\geq  \esssup_{B_R(x_o) \times (t_o, t_o+S)}u \qquad \text{ and }
	\qquad 
	\nu^-\leq \essinf_{B_R(x_o) \times (t_o, t_o+S)} u,
	\]
	then we see that $(u-\nu^{\pm})_{\pm} = 0$ in $B_R(x_o)\times (t_o,t_o+S)$. Thus, for any $8\varrho \leq R$, we have
	\[
	\tail((u-\nu^{\pm})_{\pm};x_o;\varrho;(t_o, t_o + S)) \leq \left(\frac{\varrho}{R}\right)^{\frac{sp}{p-1}}\tail((u-\nu^{\pm})_{\pm};x_o;R;(t_0, t_o + S)),
	\]
	so the $\tail$ alternative from \cref{Prop:1:eq} in \cref{Prop:1:1} can be rewritten as
	\[
	\left(\frac{\varrho}{R}\right)^{\frac{sp}{p-1}}\tail((u-\bsmu^{\pm})_{\pm};x_o;R;(t_o, t_o + \delta \varrho^{sp})) \leq \eta M.
	\]
\end{remark}

\subsection{De Giorgi Lemma}

Here we prove a De Giorgi-type Lemma on cylinders of the form $\mcq_\varrho^\theta(z_o) =B_\varrho (x_o)\times (t_o-\theta\varrho^{sp}, t_o]$. In the application, $\theta$ will be a universal constant depending only on data and hence we will keep track of the dependence of the constants on $\theta$ only qualitatively.

\begin{lemma}\label{Lm:3:3}
 Let $u$ be a locally bounded, local sub(super)-solution to \cref{maineq} in $\omt$ and
$\mcq_\varrho^\theta(z_o)=B_\varrho (x_o)\times (t_o-\theta\varrho^{sp}, t_o]$  be given for some $\tht \in (0,\infty)$. There exists a constant 
$\nu = \nu(\datanb{,\tht})\in (0,1)$
, such that if
\begin{equation*}
	\left|\left\{
	\pm\left(\bsmu^{\pm}-u\right)\le M\right\}\cap \mcq_{\varrho}^\theta(z_o)\right|
	\le
	\nu|\mcq_{\varrho}^\theta|,
\end{equation*}
holds along with the assumptions 
\begin{equation}\label{Lm:3:3:hypothesis}
    \tail((u-\bsmu^{\pm})_{\pm};x_0,\varrho,(t_o-\theta\varrho^{sp}, t_o]) \leq M \qquad \text{and} \qquad |\bsmu^{\pm}|\leq 8M,
\end{equation}
then the following conclusion holds:
\begin{equation*}
	\pm\lbr\bsmu^{\pm}-u\rbr\ge\tfrac{1}2M
	\quad
	\mbox{ on }\quad \mcq_{\frac{\varrho}{2}}^\theta(z_o) = B_{\frac{\varrho}2} (x_o)\times \left(t_o-\theta\lbr \tfrac{\varrho}{2}\rbr^{sp}, t_o\right].
\end{equation*}

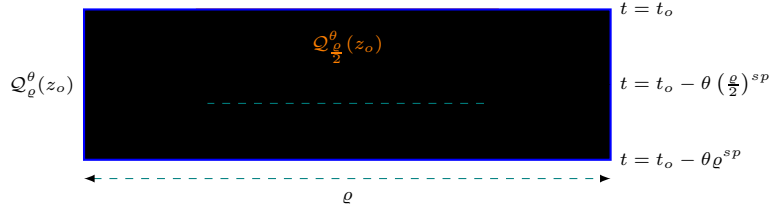
\begin{figure}[ht]\label{figureLm3.3}
	\begin{center}
		\begin{tikzpicture}[remember picture,scale=0.5,>=latex]
			\coordinate  (O) at (0,0);
			\draw[thick, draw=orange,pattern=north west lines, pattern color=orange, opacity=0.3] (-4,-2) rectangle (4,0);
			
			\draw[thick, draw=blue,fill=black, opacity=0.1] (-7,-4) rectangle (7,0);
			
			
			\draw[draw=teal, dashed, <->] (-7,-4.5) -- (7,-4.5);
			\draw[draw=teal, dashed, <->] (-4,-2.5) -- (4,-2.5);
			\draw[draw=black, dashed, |->] (4.2,-2) -- (7,-2);
			\node  at (0,-5) {\scriptsize $\varrho$};
			\node  at (0,-3) {\scriptsize $\tfrac{\varrho}{2}$};
			\node  [anchor=west] at (7,-4) {\scriptsize $t=t_o-\theta \varrho^{sp}$};
			\node  [anchor=west] at (7,0) {\scriptsize $t=t_o$};
			\node  [anchor=west] at (7,-2) {\scriptsize $t=t_o-\theta \lbr \tfrac{\varrho}{2}\rbr^{sp}$};
			\node [anchor=east] at (-7,-2) {\scriptsize$\mcq_{\varrho}^\theta(z_o)$};
			\node  at (0,-1) {\textcolor{orange}{\scriptsize$\mcq_{\frac{\varrho}{2}}^\theta(z_o)$}};
			%
			%
		\end{tikzpicture}
	\end{center}
	\caption{De Giorgi lemma}
\end{figure}
\end{lemma}

\begin{proof}
We prove the case for only the supersolutions because the case for subsolutions is analogous. Without loss of generality, we assume that $(x_o,t_o)=(0,0)$ and for $j=0,1,\ldots$, define
\begin{equation}\label{choices:B_n}
	{\def\arraystretch{\st}\begin{array}{cccc}
	 k_j:=\bsmu^-+\tfrac{M}2+\tfrac{M}{2^{j+1}},& \tilde{k}_j:=\tfrac{k_j+k_{j+1}}2,&
	 \varrho_j:=\tfrac{\varrho}2+\tfrac{\varrho}{2^{j+1}},
	&\tilde{\varrho}_j:=\tfrac{\varrho_j+\varrho_{j+1}}2,\\
	 B_j:=B_{\varrho_j},& \widetilde{B}_j:=B_{\tilde{\varrho}_j},&
	 \mcq_j:=\mcq_{\varrho_j}^\theta,&
	\widetilde{\mcq}_j:=\mcq_{\tilde\varrho_j}^\theta.
	\end{array}}
\end{equation}
Furthermore, we also define 
\begin{equation}\label{choices:B_nn}
			\hat{\varrho}_j:=\tfrac{3\varrho_j+\varrho_{j+1}}4, \quad 
			\bar{\varrho}_j:=\tfrac{\varrho_j+3\varrho_{j+1}}4,\quad 
			\hat{\mcq}_j:=\mcq_{\hat\varrho_j}^\theta, \quad
			\bar{\mcq}_j:=\mcq_{\bar\varrho_j}^\theta.
\end{equation}
We now consider a cut-off functions $\bar{\zeta_j}$ and $\zeta_j$ such that
\begin{equation}\label{cutoff_size}
	\begin{array}{c}
		\bar{\zeta_j} \equiv 1 \text{ on } B_{j+1}, \quad \bar{\zeta_j} \in C_c^{\infty}(B_{\bar{\varrho_j}}), \quad  |\nabla\bar{\zeta_j}| \apprle \frac{1}{\bar{\varrho_j} - \varrho_{j+1}} \approx \frac{2^j}{\varrho} \quad \text{and} \quad  |\pa_t\bar{\zeta_n}| \apprle \frac{1}{\theta(\bar\varrho_j^{sp} - \varrho_{j+1}^{sp})} \approx \frac{2^{jsp}}{\theta\varrho^{sp}}, \\
		{\zeta_j} \equiv 1 \text{ on } B_{\tilde{\varrho_j}}, \quad {\zeta_j} \in C_c^{\infty}(B_{\hat{\varrho_j}}), \quad  |\nabla{\zeta_j}| \apprle \frac{1}{\hat{\varrho_n} - \tilde{\varrho_{j}}}\approx \frac{2^{j}}{\varrho} \quad \text{and} \quad  |\pa_t{\zeta_j}| \apprle \frac{1}{\theta(\hat\varrho_j^{sp} - \tilde{\varrho_{j}}^{sp})}\approx \frac{2^{jsp}}{\theta\varrho^{sp}},
	\end{array}
\end{equation}
We note that by \cref{lem:g},  we have
\begin{equation*}
	\mathfrak g_-(u,k)
	\le
	\bsc  \lbr|u|+|k|\rbr^{p-2}(u-k)_-^2
	\le
	\bsc  \lbr|u|+|k|\rbr^{p-1}(u-k)_-.
\end{equation*}
Next, we make the following observations 
\begin{itemize}
	\item For $\tilde k<k$, there holds $(u-k)_-\ge (u-\tilde k)_-$, 
	\item On $A_j = \left\{u<k_j\right\}\cap \mcq_j$, we have $\bsmu^-\le u\le k_j\le \bsmu^-+M$ ,
	\item On $A_j = \left\{u<k_j\right\}\cap \mcq_j$, we have $|u|+|k_j|\leq 18 M$  since $|\bsmu^-|\leq 8M$ from \cref{Lm:3:3:hypothesis},
	\item On the set $\{u<{k}_{j+1}\}$, we have $|u|+|\tilde{k}_j|\geq \tilde{k}_j-u\geq \tilde{k}_j- {k}_{j+1}=\tfrac{M}{2^{j+3}}$.
\end{itemize}
%
%
Now, we  apply the energy estimates from \cref{Prop:energy} over  $\mcq_j $ to $ (u-\tilde{k}_j)_{-}$ and $\zeta_j$ to get
\begin{multline}\label{Eq:3.6en}
    \underset{-\theta\varrho_j^{sp} < t < 0}{\esssup}\int_{B_j}(|u|+|\tilde{k}_j|)^{p-2}(u-\tilde{k}_j)_{-}^2\zeta_j^p(x,t)\,dx  
    \\+ {\iiint_{\mcq_j}}\frac{|(u-\tilde{k}_j)_{-}(x,t)\zeta_j(x,t)-(u-\tilde{k}_j)_{-}\zeta_j(y,t)|^p}{|x-y|^{n+sp}}\,dx \,dy\,dt
    \\  \begin{array}{rcl} &\apprle_{\data{}}&  \frac{2^{jp}}{\varrho^{sp}}M^p |A_j| +\frac{2^{jp}}{\tht\varrho^{sp}}M^p |A_j| 
    \\&&+\lbr \underset{\stackrel{-\theta\varrho_j^{sp} < t < 0;}{ x\in \spt \zeta_j}}{\esssup}\int_{\RR^n \setminus B_j}\frac{(u-\tilde{k}_j)_{-}^{p-1}(y,t)}{|x-y|^{n+sp}}\,dy\rbr\lbr \int_{-\theta\varrho_j^{sp}}^{0}\int_{B_j} (u-k_j)_{-}(x,t)\zeta(x,t)\,dx\,dt\rbr.
    \end{array}
\end{multline}

We estimate the Tail as below following what was described in \cref{sec:tail} to get
\begin{equation*}
{\def\arraystretch{1}	\begin{array}{rcl}
\underset{\stackrel{-\theta\varrho_j^{sp} < t < 0;}{ x\in \spt \zeta_j}}{\esssup}\int_{\RR^n \setminus B_j}\frac{(u-\tilde{k}_j)_{-}^{p-1}(y,t)}{|x-y|^{n+sp}}\,dy & \apprle & 2^{jsp}\frac{M^{p-1}}{\varrho^{sp}} +  \frac{2^{jsp}}{\varrho^{sp}}\tailp((u-\bsmu^{-})_{-};0;\varrho;(-\theta\varrho^{sp}, 0])
\\
&\overset{\cref{Lm:3:3:hypothesis}}{\apprle} & 2^{jsp}\frac{M^{p-1}}{\varrho^{sp}}.
\end{array}}
\end{equation*}
Combining the previous two estimates along with the observations, we have
\begin{equation}\label{Eq:sample}
	\frac{1}{18^2}\frac{M^{p-2}}{2^{p(j+3)}} \underset{-\theta\tilde\varrho_j^{sp} < t < 0}{\esssup}
	\int_{\widetilde{B}_j} (u-\tilde{k}_j)_-^2\,dx
	+\iiint_{\widetilde{\mcq}_j}\frac{|(u-\tilde{k}_j)_{-}(x,t)-(u-\tilde{k}_j)_{-}|^p}{|x-y|^{n+sp}}\,dx \,dy\,dt
	{\leq}
	\bsc  \frac{2^{jsp}}{\varrho^{sp}}M^{p}|A_j|,
\end{equation} where we recall that $A_j = \left\{u<k_j\right\}\cap \mcq_j$.

From Young's inequality, we have
\begin{multline}\label{eq3.9}
	|(u-\tilde{k}_{j})_- \bar{\zeta_j}(x,t) - (u-\tilde{k}_{j})_- \bar{\zeta_j}(y,t)|^p \leq c |(u-\tilde{k}_{j})_- (x,t) - (u-\tilde{k}_{j})_- (y,t)|^p\bar{\zeta_j}^p(x,t) \\
	+ c |(u-\tilde{k}_{j})_-(y,t)|^p |\bar{\zeta_j}(x,t) - \bar{\zeta_j}(y,t)|^p,
\end{multline}
from which we obtain the following sequence of estimates:
\begin{equation*}
	\begin{array}{rcl}
	\frac{M}{2^{j+3}}
	|A_{j+1}|
	&\overred{3.7a}{a}{\leq} &
	\iint_{\bar{\mcq}_{j+1}}( u-\tilde{k}_j)_-\bar{\zeta_j}\,dx\,dt \\
	&\overred{3.7b}{b}{\leq} &
	\left[\iint_{\widetilde{\mcq}_j}\left[( u-\tilde{k}_j)_-\bar{\zeta_j}\right]^{p\frac{n+2s}{n}}
	\,dx\,dt\right]^{\frac{n}{p(n+2s)}}|A_j|^{1-\frac{n}{p(n+2s)}}\\
	&\overred{3.7c}{c}{\leq} &\bsc 
	\left(\tilde{\rho}^{sp}\iiint_{\widetilde{\mcq}_j}\frac{|(u-\tilde{k}_j)_{-}(x,t)\bar{\zeta_j}(x,t)-(u-\tilde{k}_j)_{-}\bar{\zeta_j}(y,t)|^p}{|x-y|^{n+sp}}\,dx \,dy\,dt\right)^{\frac{n}{p(n+2s)}} 
    \\&&\qquad\times \left(\underset{-\theta\tilde\varrho_j^{sp} < t < 0}{\esssup}\int_{\widetilde{B}_j}\left[(u-\tilde{k}_j)_{-}\bar\zeta_j(x,t)\right]^2\,dx\right)^{\frac{s}{n+2s}}|A_j|^{1-\frac{n}{p(n+2s)}}\\
	&\overred{3.7d}{d}{\leq}&
	\bsc  
	\lbr {2^{jsp}}M^p\rbr^{\frac{n}{p(n+2s)}}
	\lbr\frac{2^{p(j+3)+jsp}}{\varrho^{n+sp}}M^2\rbr^{\frac{s}{n+2s}}
	|A_j|^{1+\frac{s}{n+2s}} \\
	&\overred{3.7e}{e}{\leq} &
	\bsc 
	\frac{b_0^j}{\varrho^\frac{s(n+sp)}{n+2s}}
	M |A_j|^{1+\frac{s}{n+2s}},
	\end{array}
\end{equation*}
where to obtain \redref{3.7a}{a}, we made use of the observations and enlarged the domain of integration with $\tilde{\zeta}_j$ as defined in \cref{cutoff_size}; to obtain \redref{3.7b}{b}, we applied H\"older's inequality; to obtain \redref{3.7c}{c}, we applied \cref{fracpoin}; to obtain \redref{3.7d}{d}, we made use of \cref{Eq:3.6en}, \cref{Eq:sample} and \cref{eq3.9}  with $\bsc = \bsc_{\data{,\theta}}$ and finally we collected all the terms to obtain \redref{3.7e}{e}, where $b_0=b_0(\datanb{})\geq 1$ is a universal constant. Setting
$\bsy_j=|A_j|/|\mcq_j|$, we get
\begin{equation*}
	\bsy_{j+1}
	\le
	\bsc  \boldsymbol b^j \bsy_j^{1+\frac{s}{n+2s}},
\end{equation*}
for  constants $\bsc  = \bsc_{\data{,\theta}} \geq 1$ and $\boldsymbol b=b_{\data{}}\geq  1$ depending only on the data. The conclusion now follows from \cref{geo_con}.
\end{proof}

\begin{remark}\label{remark_degiorgi}
	The conclusion of \cref{Lm:3:3} can be changed to hold on the following cylinders 
	\begin{equation*}
		\pm\lbr\bsmu^{\pm}-u\rbr\ge\tfrac{1}2M
		\quad
		\mbox{ on }\quad \mcq_{\frac{\varrho}{2}}^\theta(z_o) = B_{\frac{\varrho}2} (x_o)\times \left(t_o-\tfrac{\theta}{2} \varrho^{sp}, t_o\right].
	\end{equation*}

The proof is a simple modification by considering $a_j := \frac12 + \frac{1}{2^{j+1}}$, $\hat{a_j} := \frac{3a_j+a_{j+1}}{4}$, $\bar{a_j} := \frac{a_j + 3a_{j+1}}{4}$ and $\tilde{a_j} := \frac{a_j + a_{j+1}}{2}$. Then we take cylinders of the form
\[
\hat{\mcq}_j= B_{\hat\varrho_j} \times (t_0 -\theta \varrho^{sp}\hat{a_j}, t_0) \qquad 
\bar{\mcq}_j=B_{\bar\varrho_j} \times (t_0 -\theta \varrho^{sp}\bar{a_j}, t_0) \txt{and} \tilde{\mcq}_j= B_{\tilde\varrho_j} \times (t_0 -\theta \varrho^{sp}\tilde{a_j}, t_0),
\]
where $\hat\varrho_j, \bar\varrho_j$ and $\tilde\varrho_j$ are as defined in \cref{choices:B_n} and \cref{choices:B_nn}. 

With these choices, we see that \cref{cutoff_size} becomes
\[
 |\pa_t\bar{\zeta_n}| \apprle \frac{1}{\theta\varrho^{sp}(\bar a_j - a_{j+1})} \approx \frac{2^{j}}{\theta\varrho^{sp}} \txt{and}  |\pa_t{\zeta_n}| \apprle \frac{1}{\theta\varrho^{sp}(\hat a_j - \tilde a_{j+1})} \approx \frac{2^{j}}{\theta\varrho^{sp}},
\]
and the rest of the proof follows verbatim as in \cref{Lm:3:3} with a slightly different $\nu = \nu(\datanb{,\tht})\in (0,1)$.
\end{remark}

\subsection{Expansion of positivity in time}\label{sec:exptime}
We prove the following expansion of positivity in time lemma.
\begin{lemma}\label{Lm:3:1}
	Let $M>0$ and $\alpha \in(0,1)$ be given,  then, there exist constants $\delta \in (0,1)$ and $\epsilon \in (0,1)$,
	depending only on the data and $\alpha$, such that whenever $u$ is a locally bounded, local, 
	weak sub(super)-solution in $\omt$ satisfying
	\begin{equation*}
	\left|\left\{
		\pm\lbr\bsmu^{\pm}-u(\cdot, t_o-\delta\varrho^{sp})\rbr\geq M
		\right\}\cap B_{\varrho}(x_o)\right|
		\geq\alpha \left|B_{\varrho}\right|,
	\end{equation*}
	along with the bounds  
	\begin{equation}\label{sec:exptime:hyp}
	\tail((u-\bsmu^{\pm})_{\pm};x_0,\varrho,(t_o-\delta\varrho^{sp},t_o]) \leq  M \qquad \text{and} \quad |\bsmu^{\pm}|\leq 8M,  
	\end{equation}
	then the following conclusion follows:
	\begin{equation}\label{Eq:3:1}
	\left|\left\{
	\pm\left(\bsmu^{\pm}-u(\cdot, t)\right)\geq \epsilon M\right\} \cap B_{\varrho}(x_o)\right|
	\geq\frac{\alpha}2 |B_\varrho|
	\quad \text{ for all } \,  t\in(t_o-\delta\varrho^{sp},t_o].
\end{equation}
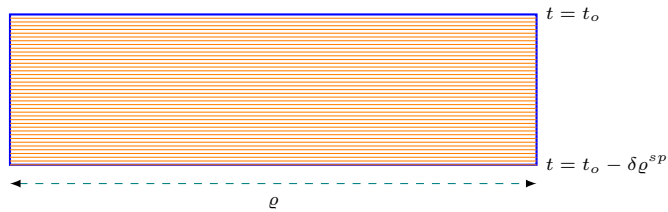
\begin{figure}[h]\label{figureLm3.1}
	\begin{center}
		\begin{tikzpicture}[remember picture,scale=0.5,>=latex]
			\coordinate  (O) at (0,0);
			
			\draw[thick, draw=blue] (-7,-4) rectangle (7,0);
			\foreach \y in{-4,-3.90,...,0}
			\draw[color=orange, opacity=0.3] (-7,\y) -- (7,\y);
			
			
			
			\draw[draw=teal, dashed, <->] (-7,-4.5) -- (7,-4.5);
			\node  at (0,-5) {\scriptsize $\varrho$};
			\node  [anchor=west] at (7,-4) {\scriptsize $t=t_o-\de \varrho^{sp}$};
			\node  [anchor=west] at (7,0) {\scriptsize $t=t_o$};
			%
			%
		\end{tikzpicture}
	\end{center}
	\caption{Expansion of positivity in time}
\end{figure}
\end{lemma}
\begin{proof} 
We prove the case of supersolutions only because the case for subsolutions is analogous. Without loss of generality,  assume that $(x_o,t_o)=(0,0)$. For $k>0$ and $t>0$,  set
\[
A_{k,\varrho}(t) = \left\{u(\cdot,t)  - \bsmu^-< k\right\} \cap B_{\varrho}.
\]
By hypothesis, we have 
\[
|A_{M,\varrho}(-\delta\varrho^{sp})|\leq (1-\alpha)|B_{\varrho}|.
\]
We consider the energy estimate  from \cref{Prop:energy} for $(u-k)_{-}$ with $k=\bsmu^-+M$ over the cylinder $B_{\varrho}\times(-\delta\varrho^{sp},0]$ where $\delta>0$ will be chosen later in the proof. Note that $(u-k)_{-} \leq M$ in $B_{\varrho}$. For $\tilde\sigma \in (0,\tfrac18]$ to be chosen later, we take a cut-off function $\zeta = \zeta(x)\geq 0$,  such that it is supported in $B_{\varrho(1-\frac{\tilde\sigma}{2})}$ with $\zeta \equiv 1$ on $B_{(1-\tilde\sigma)\varrho}$ and $|\nabla \zeta| \apprle \tfrac{1}{\tilde\sigma\varrho}$ and make use of  \cref{Prop:energy} to get
\begin{multline}\label{est4.6}
    \int_{B_\varrho\times\{\mft\}}\int_{u}^k |s|^{p-2}(s-k)_-\,ds \zeta^p\,dx 
     \leq  \frac{C M^p}{\tilde\sigma^p \varrho^{sp}}|\de\varrho^{sp}+\mft||B_{\varrho}|
    + \int_{B_\varrho\times\{-\de\varrho^{sp}\}}\int_{u}^k     |s|^{p-2}(s-k)_-\,ds\,\zeta^p\,dx
    \\
     + C\underset{\stackrel{t \in (-\delta\varrho^{sp},0]}{ x\in \spt \zeta}}{\esssup}\int_{B_{\varrho}^c}\frac{(u-k)_{-}^{p-1}(y,t)}{|x-y|^{n+sp}}\,dy\int_{-\delta\varrho^{sp}}^0\int_{B_{\varrho}} (u-k)_{-}(x,t)\zeta(x,t)\,dx\,dt,
\end{multline}
where $\mft \in (-\de \varrho^{sp},0]$ is any time level.  Furthermore, using $u\geq\bsmu^-$, we get
\begin{equation*}
	\begin{array}{rcl}
	\int_{B_\varrho\times\{-\de\varrho^{sp}\}}\int_{u}^k |s|^{p-2}(s-k)_-\,ds\zeta^p\,dx
	 &\leq &  
	|A_{M,\varrho}(-\de\varrho^{sp})|\int_{\bsmu^-}^{k} |s|^{p-2}(s-k)_-\,ds\\
	&\leq& 
	(1-\alpha)|B_\varrho|\int_{\bsmu^-}^{k} |s|^{p-2}(s-k)_-\,ds.
	\end{array}
\end{equation*}
For the $\tail$ term, we follow what was described in \cref{sec:tail} to get
\begin{equation*}
	\begin{array}{rcl}
\underset{\stackrel{t \in (-\delta\varrho^{sp},0]}{ x\in \spt \zeta}}{\esssup}\int_{\RR^n \setminus B_{\varrho}}\frac{(u-k)_{-}^{p-1}(y,t)}{|x-y|^{n+sp}}\,dy &\apprle_{p,n}& \tilde\sigma^{-(n+sp)}\frac{M^{p-1}}{\varrho^{sp}} +  \frac{\tilde\sigma^{-(n+sp)}}{\varrho^{sp}}\tailp((u-\bsmu^{-})_{-};0,\varrho,(-\delta\varrho^{sp},0])
\\
&\overset{\cref{sec:exptime:hyp}}{\apprle_{p,n}} & \tilde\sigma^{-(n+sp)}\frac{M^{p-1}}{\varrho^{sp}}.
\end{array}
\end{equation*}

Setting $k_\epsilon:=\epsilon M$ for some $ \epsilon \in (0,\tfrac12)$ to be chosen later, we estimate the term on the left hand side of \cref{est4.6} from below to get
\begin{equation*}
	\int_{B_\varrho\times\{\mft\}}\int_{u}^k |s|^{p-2}(s-k)_-\,ds \,\zeta^p\,dx
	\geq 
	\big|A_{k_\epsilon,(1-\tilde\sigma)\varrho}(\mft)\big| \int_{\bsmu^-+k_\epsilon}^k |s|^{p-2}(s-k)_-\,ds.
\end{equation*}
 Making use of the bound 
$\tfrac{1}{2}M\leq(1-\epsilon)M=k-(\bsmu^-+k_\epsilon)\leq |\bsmu^-+k_\epsilon|+|k|\leq 2(|\bsmu^-|+ M)\overlabel{sec:exptime:hyp}{\leq} 18 M$, we further get
\begin{equation}\label{prop_1}
	\def\arraystretch{\st}
	\begin{array}{rcl}
	\int_{\bsmu^-+k_\epsilon}^k|s|^{p-2}(s-k)_-\,ds
	&=&
	\tfrac{1}{p-1}\, \mathfrak g_-(\bsmu^-+k_\epsilon,k)
	\overlabel{lem:g}{\geq} 
	\tfrac{1}{\bsc  (p)}
	\lbr|\bsmu^-+k_\epsilon|+|k|\rbr^{p-2} (k-\bsmu^--k_\epsilon)^2\\
	&\geq &
	\tfrac{1}{\bsc  (p)} M^p.
	\end{array}
\end{equation}
Next, we note that
\begin{equation*}
	\abs{A_{k_\epsilon,\varrho}(\mft)} \leq 
	\abs{A_{k_\epsilon,(1-\tilde\sigma)\varrho}(\mft)}+n\tilde\sigma |B_\varrho|.
\end{equation*}
Putting together the above estimates gives
\begin{equation*}
	|A_{k_\epsilon,\varrho}(\mft)|
	\leq 
	\frac{\int_{\bsmu^-}^k |s|^{p-2}(s-k)_-\,ds }{\int_{\bsmu^-+k_\epsilon}^k |s|^{p-2}(s-k)_-\,ds }
	(1-\alpha)|B_\varrho|
	+
	\frac{\bsc  \delta}{\tilde\sigma^p}|B_\varrho| +n\tilde\sigma|B_\varrho|,
\end{equation*}
for a universal constant $\bsc  > 0$. Proceeding analogously as in \cite[Proof of Lemma 4.1]{bogeleinHolderRegularitySigned2021}, we write
\[
\frac{\int_{\bsmu^-}^k |s|^{p-2}(s-k)_{-} \,ds }{\int_{\bsmu^-+k_\epsilon}^k |s|^{p-2}(s-k)_{-} \,ds } = 1 + \frac{\int_{\bsmu^-}^{\bsmu^-+k_\epsilon} |s|^{p-2}(s-k)_{-} \,ds}{\int_{\bsmu^-+k_\epsilon}^k  |s|^{p-2}(s-k)_{-} \,ds} = 1 + I_{\epsilon}.
\]
Using $|\bsmu^-|\leq 8M$, $|\bsmu^-+k_\epsilon|\leq 9M$ and applying \cref{lem:g},  we get
\[
	\int_{\bsmu^-}^{\bsmu^-+k_\epsilon} |\tau|^{p-2}(\tau-k)_-\,d\tau
	\leq 
	M\int_{\bsmu^-}^{\bsmu^-+k_\epsilon}|\tau|^{p-2}\,d\tau
	=
	M |s|^{p-2} s\Big|_{\bsmu^-}^{\bsmu^-+k_\epsilon}
	\leq
	\bsc  (p)M^p\epsilon.
\]
Making use of  \cref{prop_1},  we obtain
\begin{equation*}
	I_\epsilon
	\leq
	\bsc  (p) \epsilon.
\end{equation*}
We now choose $\epsilon\in (0,1)$ small enough such that
\begin{equation*}
	(1-\alpha)(1+\bsc \epsilon)\leq 1-\tfrac{3}{4}\alpha,
\end{equation*}
and make the choice $\tilde\sigma=\tfrac{\alpha}{8n}$.
Finally, we choose $\delta\in (0,1)$ small enough so that $\tfrac{\bsc\delta}{\tilde\sigma^p}\leq\tfrac{\alpha}{8}$, which gives the required conclusion.
\end{proof}

\subsection{A Shrinking Lemma}

\begin{lemma}\label{Lm:3:2}
With $\epsilon$ and $\de$ as obtained in \cref{Lm:3:1}, we assume that \cref{Eq:3:1} from \cref{Lm:3:1} holds. Let
$\mcq=\mcq_{\varrho}^{\de}(z_o) = B_{\varrho}(x_o)\times\left(t_o-\delta\varrho^{sp},t_o\right]$ be the corresponding cylinder and let $\widehat{\mcq}=B_{4\varrho}(x_o)\times\left(t_o-\delta\varrho^{sp},t_o\right]\Subset E_T$. Furthermore, for some integer $j \in \NN$, assume the following is satisfied:
\begin{equation}\label{Lm:3:2:hyp}
\tail((u-\bsmu^{\pm})_{\pm};x_0,8\varrho,(t_o-\delta\varrho^{sp},t_o]) \leq \frac{\epsilon M}{2^j} \qquad \text{and} \qquad |\bsmu^{\pm}|\leq \sigma M,
\end{equation}
where $\sigma$ is chosen to be
\[
\sigma := \left\{ {\def\arraystretch{\st}\begin{array}{ll} 8 & \text{if} \,\,\, 1<p<2, \\
	\frac{\epsilon}{2^{j}} & \text{if} \,\,\, p >2,
	\end{array}}\right.
\]
then there exists $\bsc_o >0$ depending only on data and $\alpha$ (recall $\alpha$ is from \cref{Lm:3:1}),  we have 
\begin{equation*}
	\left|\left\{
	\pm\lbr\bsmu^{\pm}-u\rbr\le\frac{\epsilon M}{2^{j+1}}\right\}\cap \widehat{ \mcq}\right|
	\leq\left(\frac{\bsc_o}{2^{j}-1}\right)^{p-1}|\widehat{ \mcq}|.
\end{equation*}
\end{lemma}
\begin{proof}
We prove the case of super-solutions only because the case for sub-solutions is analogous. Without loss of generality,  assume that $(x_o,t_o)=(0,0)$. We write the energy estimates over the cylinder $B_{8\varrho}\times (-\delta\varrho^{sp},0]$ for the functions $(u-k_j)_-$ with the choice of levels 
\[
k_j = \bsmu^{-} + \frac{\epsilon M}{2^j} \qquad \text{for given } \ j \geq 1.
\] 
We choose a test function $\zeta = \zeta(x)$ such that $\zeta\equiv 1$ on $B_{4\varrho}$, it is supported in $B_{6\varrho}$ and
\[
|\nabla \zeta|\leq \frac{1}{2\varrho}.
\] 
Using $(u-k_j)_{-} \leq \tfrac{\epsilon M}{2^{j}}$ locally (see \cref{Eq:1:5}) and  by our choice of the test function, applying \cref{Prop:energy}, we get
\begin{multline}\label{Lm:3:2:energy}
\int_{-\delta\varrho^{sp}}^0\int_{B_{4\varrho}}|(u-k_j)_-(x,t)|\int_{B_{4\varrho}}\frac{|(u-k_j)_+(y,t)|}{|x-y|^{n+sp}}\,dx\,dy\,dt
\\
\apprle_{\data{}}\frac{1}{\varrho^{sp}}\left(\frac{\epsilon M}{2^j}\right)^{p}|\widehat{\mcq}|
     +\frac{\epsilon M}{2^j}|\widehat{\mcq}|\underset{\stackrel{t \in (-\delta\varrho^{sp},0]}{ x\in \spt \zeta}}{\esssup}\int_{B_{8\varrho}^c}\frac{(u-k_j)_{-}(y,t)}{|x-y|^{n+sp}}\,dy
          + \int_{B_{8\varrho}\times\{-\de \varrho^{sp}\}} \mathfrak g_- (u,k_j) \,dx.
\end{multline}
 We recall the Tail estimates from  \cref{sec:tail} to get
\begin{equation*}
	\begin{array}{rcl}
\underset{\stackrel{t \in (-\delta\varrho^{sp},0]}{ x\in \spt \zeta}}{\esssup}\int_{B_{8\varrho}^c}\frac{(u-k_j)_{-}^{p-1}(y,t)}{|x-y|^{n+sp}}\,dy &\apprle_{n,p}& \frac{1}{\varrho^{sp}}\left(\frac{\epsilon M}{2^j}\right)^{p-1} +  \frac{1}{\varrho^{sp}}\tailp((u-\bsmu^{-})_{-};0;8\varrho;(-\delta\varrho^{sp},0])
\\
&\overlabel{Lm:3:2:hyp}{\apprle_{n,p}}& \frac{1}{\varrho^{sp}}\left(\frac{\epsilon M}{2^j}\right)^{p-1}.
\end{array}
\end{equation*}
 Proceeding analogously as \cite[Proof of Lemma 4.2]{bogeleinHolderRegularitySigned2021},  we make use of \cref{lem:g} to  get 
\begin{equation*}
	\mathfrak g_- (u,k_j)
	\le
	\bsc  \lbr|u|+|k_j|\rbr^{p-2}(u-k_j)_-^2.
\end{equation*}
\begin{description}
	\item[\color{airforceblue}{Case $p \geq 2$:}] For $p\geq 2$, we use $(u-k_j)_-\le |u|+|k_j|$, $u\geq \bsmu^-$ and $|\bsmu^{-}|\leq\epsilon M 2^{-j}$ to get
	\begin{equation*}
		\mathfrak g_- (u,k_j)
		\leq
		\bsc_p  \lbr|u|+|k_j|\rbr^p \chi_{\{u\leq k_j\}}
		\leq
		\bsc_p \left(\frac{\epsilon M}{2^j}\right)^{p}.
	\end{equation*}
	\item[\color{airforceblue}{Case $p \leq 2$:}] For $1<p<2$, we use $(u-k_j)_-\leq |u|+|k_j|$ and $u\ge \bsmu^-$ to get (note that in this case, we need no assumption on $\bsmu^-$)
	\begin{equation*}
		\mathfrak g_- (u,k_j)
		\leq
		\bsc_p  (u-k_j)_-^p
		\leq
		\bsc_p \left(\frac{\epsilon M}{2^j}\right)^{p}.
	\end{equation*}
\end{description}

 In particular,  for all $1<p<\infty$, we get
\begin{equation*}
	\int_{B_{8\varrho}\times\{-\de \varrho^{sp}\}}\mathfrak g_- (u,k_j)
	\leq
	 \frac{\bsc_p }{\delta\varrho^{sp}}
	 \left(\frac{\epsilon M}{2^j}\right)^{p} |\widehat{ \mcq}|.
\end{equation*}
Putting all the above estimates together and recalling that $\delta \in (0,1)$, we get
\[
\int_{-\delta\varrho^{sp}}^0\int_{B_{4\varrho}}|(u-k_j)_-(x,t)|\int_{B_{4\varrho}}\frac{|(u-k_j)_+(y,t)|}{|x-y|^{n+sp}}\,dx\,dy\,dt \leq \frac{\bsc }{\delta\varrho^{sp}}
	 \left(\frac{\epsilon M}{2^j}\right)^{p} |\widehat{ \mcq}|
\]
We now invoke \cref{lem:shrinking} (in which we replace $u$ with $u - \bsmu^-$) with $l =  \epsilon M 2^{-j}$ and $m = \epsilon M$ to conclude that
\[
|\{u-\bsmu^{-} < \epsilon M2^{-(j+1)}\}\cap\widehat{ \mcq})| \leq \left(\frac{\bsc}{2^{j}-1}\right)^{p-1}|\widehat{ \mcq}|,
\]
holds for a constant $\bsc>0$ depending only on $\alpha$ and data. Note that $\bsc$ depends on $\delta = \delta(n,p,s)$.  
\end{proof}

\subsection{Proof of \texorpdfstring{\cref{Prop:1:1}}.}
We prove the proposition for supersolutions because the case for subsolutions is analogous. Without loss of generality, we assume $(x_0,t_0) = (0,0)$ and prove the proposition in several steps:
\begin{description}
	\item[\color{auburn}{Step 1:}] With $\alpha$ be as given in \cref{Prop:1:1}, we let  $\delta=\de(\datanb{,\alpha}),\epsilon=\epsilon(\datanb{,\alpha})\in(0,1)$ be as obtained in \cref{Lm:3:1}. Let $\bsc_o = \bsc_o(\datanb{,\alpha}) \geq 1$ be the  constant from  \cref{Lm:3:2}  and  $\nu=\nu(\datanb{,\tht})\in(0,1)$ be the constant from \cref{Lm:3:3} and \cref{remark_degiorgi}, where we choose $\theta=\delta$ which gives $\nu = \nu(\datanb{,\alpha})$.
	\item[\color{auburn}{Step 2:}] Next, we choose an integer $j_{*} \geq 1$ in such a way that 
	\[
	\left(\frac{\boldsymbol{C}_0}{2^{j_*}-1}\right)^{p-1} \leq \nu,
	\]
	then, $j_* = j_*(\datanb{,\alpha})$.
	\item[\color{auburn}{Step 3:}] We take 
	\begin{equation*}
		\xi := \left\{ {\def\arraystretch{\st}\begin{array}{ll} 8 & \text{if} \ 1 < p < 2,\\
				\epsilon 2^{-j_*} & \text{if} \ p >2,\end{array}} \right.
	\end{equation*}
	and assume that $|\bsmu^{-}|\le \xi M$.
	\item[\color{auburn}{Step 4:}]   With these choices, applying \cref{Lm:3:1} and \cref{Lm:3:2}, we get
	\begin{equation*}
		\bigg|\bigg\{
		u\leq \bsmu^-+\frac{\epsilon M}{2^{j_*+1}}\bigg\}\cap \widehat{\mcq}\bigg|
		\leq
		\nu|\widehat{\mcq}|,\quad
	\end{equation*}
	where $\widehat{ \mcq}=B_{4\varrho}\times(-\delta\varrho^{sp},0]$
	provided
	\[
	\tail((u-\bsmu^{-})_{-};0,8\varrho,(-\delta\varrho^{sp},0]) \leq \frac{\epsilon M}{2^{j_*}}.
	\]
	\item[\color{auburn}{Step 5:}] Finally, we apply \cref{Lm:3:3} along with \cref{remark_degiorgi} and $M$ replaced by $\frac{\epsilon M}{2^{j_*+1}}$ to get 
	\begin{equation*}
		u\geq \bsmu^- + \frac{\epsilon M}{2^{j_*+2}}
		\quad
		\mbox{a.e.~in } B_{2\varrho}\times \left(-\tfrac{\de}{2}\varrho^{sp},0\right].
	\end{equation*}
	\item[\color{auburn}{Step 6:}] Note that the $\tail$ alternative required to apply \cref{Lm:3:3} is already satisfied due to the preceding condition on the $\tail$. We set
	$\eta=\frac{\epsilon}{2^{j_*+2}}$ and 
	the conclusion follows.
	\item[\color{auburn}{Step 7:}] The above proof is given for backwards in time cylinders but a simple change in the starting time level gives the desired proof.
\end{description}
\section{Iteration of \texorpdfstring{\cref{Prop:1:1}}.}
\label{Section4}


\begin{lemma}\label{iteratedprop}
	Let $u$ be a locally bounded, local, weak sub(super)-solution to \cref{maineq} in $\omt$. With notation as in \cref{Eq:1:5}, suppose there exists constants  $M>0$ and $\alpha \in(0,1)$ such that the following is satisfied:
	\begin{equation}\label{initialize1}
		\left|\left\{\pm\lbr\bsmu^{\pm}-u(\cdot, t_o)\rbr\geq M\right\}\cap B_\varrho(x_o)\right|
		\geq
		\alpha \big|B_\varrho\big|.
	\end{equation}
	Given any arbitrary $A >1$, there exists $\bar{\eta} = \bar{\eta}(n,p,s,\alpha,A) \in (0,1)$ such that if the following is satisfied
	\begin{equation*}
	\tail((u-\mu^{\pm})_{\pm};x_0,\varrho,(t_o,t_o+A\varrho^{sp}]) \leq \bar\eta M \quad \text{and} \quad \abs{\bsmu^{\pm}}\leq \xi M,
	\end{equation*}
	then we have the following conclusion
	\begin{equation*}
		\pm\lbr\bsmu^{\pm}-u\rbr\geq\bar\eta M
		\quad
		\mbox{a.e.~in $B_{2\varrho}(x_o)\times\left(t_0+\tfrac{\de}{2}\varrho^{sp},t_0+A\rho^{sp}\right] .$}
	\end{equation*}
	Here we have taken $\xi$ (and $\eta$ both independent of $A$) to be the same constant as in \cref{Prop:1:1}:
	\[	\xi	:= \left\{
	{\def\arraystretch{1} \begin{array}{ll}
			2\eta, & \mbox{if\,\,   $p>2$,} \\
			8, & \mbox{if\,\,   $1<p\le 2$.}
	\end{array}}
	\right.
	\]
\end{lemma}
\begin{proof}The proof follows by repeated application of \cref{Prop:1:1} and 
the proof will proceed according to the following steps:
	\begin{description}
		\item[Step 1:] We will apply \cref{Prop:1:1} with \cref{initialize1} and \cref{remark_degiorgi} (see \cref{fig5}) to get 
		\begin{equation}\label{first_iter_gen}
			\pm\lbr\bsmu^{\pm}-u\rbr\geq\eta M
			\quad
			\mbox{a.e.~in $B_{2\varrho}(x_o)\times\left(t_0+\tfrac{\de}{2}\varrho^{sp},t_0+\de\varrho^{sp}\right] .$}
		\end{equation}

		\item[Step 2:] Let us take any time level $\mft_1 \in \left(t_0+\tfrac{\de}{2}\varrho^{sp},t_0+\de{\rho}^{sp}\right]$ (see \cref{fig5}), then for $\eta$ as in \cref{Prop:1:1}, we have 
		\begin{equation}\label{eq3.33333}
			\left|\left\{\pm\lbr\bsmu^{\pm}-u(\cdot, \mft_1)\rbr\geq \eta M\right\}\cap B_\varrho(x_o)\right|
			=
			\big|B_\varrho\big|.
		\end{equation}
		\begin{figure}[ht]
		\begin{center}
			\begin{tikzpicture}[line cap=round,line join=round,>=latex,scale=0.5]
				\coordinate  (O) at (0,0);
				\draw[thick, draw=orange,pattern=north west lines, pattern color=orange, opacity=0.3] (-7,-2) rectangle (7,0);
				
				\draw[thick, draw=blue] (-3,-5) -- (3,-5);
				\draw[thick, draw=blue, dashed] (-3,-5) -- (-3,2);
				\draw[thick, draw=blue, dashed] (3,-5) -- (3,2);
				
				
				\draw[draw=black,thick] (-3,-3.5) -- (3,-3.5);
				\draw[draw=black,thick] (-3,1.5) -- (3,1.5);
				\draw[draw=teal, dashed, <->] (-3,-5.5) -- (3,-5.5);
				\draw[draw=teal, dashed, <->] (-7,-3) -- (7,-3);
				
				\foreach \y in{-2,-1.90,...,0}
				\draw[color=black, opacity=0.3] (-3,\y) -- (3,\y);
				
				\draw[decoration={brace,raise=5pt},decorate]
				(-3,-2) -- node[left=6pt] {\scriptsize $\mft_1$} (-3,0);
				\node  at (0,-6) {\scriptsize $\varrho$};
				\node  at (0,-2.5) {\scriptsize $2\varrho$};
				\node  at (4,-3.5) {\scriptsize $t=t_o$};
				\node  at (5,1.5) {\scriptsize $t=t_o+A\varrho^{sp}$};
				\node  [anchor=west] at (7,-2) {\scriptsize $t=t_o+\tfrac{\de}{2}\varrho^{sp}$};
				\node  [anchor=west] at (7,0) {\scriptsize $t=t_o+ \de \varrho^{sp}$};
			\end{tikzpicture}
		\end{center}
		\caption{\cref{Prop:1:1} gives the following pointwise information}
		\label{fig5}
	\end{figure}
	
		\item[Step 3:] We can now apply \cref{Lm:3:1} noting that \cref{eq3.33333} implies the first hypothesis of \cref{Lm:3:1} is satisfied with $\al =1$. Additionally, if the following is imposed to hold:
		\begin{equation}\label{corsec:exptime:hyp}
			\tail((u-\bsmu^{\pm})_{\pm};x_0,\varrho,(\mft_1,\mft_1+\de\varrho^{sp}]) \leq  \eta M \qquad \text{and} \quad |\bsmu^{\pm}|\leq 8M,  
		\end{equation}
		then all the hypothesis of \cref{Lm:3:1} are satisfied and  the following conclusion follows for some universal $\varepsilon = \varepsilon(\datanb) \in (0,1)$ and $\de = \de(\datanb)\in(0,1)$:
		\begin{equation}\label{corEq:3:1}
			\left|\left\{
			\pm\left(\bsmu^{\pm}-u(\cdot, t)\right)\geq \epsilon\eta M\right\} \cap B_{\varrho}(x_o)\right|
			\geq\frac{1}2 |B_\varrho|
			\quad \text{ for all } \,  t\in(\mft_1, \mft_1+\delta\varrho^{sp}].
		\end{equation}
		\item[Step 4:] 
		Our next step is to apply \cref{Lm:3:2} and for this, we need to be a bit more careful and keep track of all the constants. \emph{We show this calculations in the case of super solutions noting the same works also in the case of sub solutions}. In the proof of \cref{Lm:3:2}, we make the choice $k_j:=\bsmu^- + \frac{\varepsilon\eta M}{2^j}$ where $\eta$ is from \cref{corsec:exptime:hyp} and $\varepsilon$ as in \cref{corEq:3:1}. This gives, $u(\cdot,\mft_1) \geq \bsmu^-+\eta M \geq k_j$ and thus, we see that 
		\[
		\left.g_-(u,k_j)\right|_{t=\mft_1} \leq \left.\boldsymbol{C} (|u| + |k_j|)^{p-2}(u-k_j)_-^2\right|_{t=\mft_1} = 0.
		\]
		
		 As a consequence, in the proof of \cref{Lm:3:2}, we do not need to assume any bound on $\bsmu^-$ since the last term appearing on the right hand side of \cref{Lm:3:2:energy} is zero. Thus we obtain the estimate
		\[
		\int_{\mft_1}^{\mft_1+\delta\varrho^{sp}}\int_{B_{4\varrho}}|(u-k_j)_-(x,t)|\int_{B_{4\varrho}}\frac{|(u-k_j)_+(y,t)|}{|x-y|^{n+sp}}\,dx\,dy\,dt \leq \frac{\bsc}{\delta\varrho^{sp}}
		\left(\frac{\epsilon\eta M}{2^j}\right)^{p} |\widehat{ \mcq}_1|,
		\]
		where $\widehat{{\mcq}}_1=B_{4\varrho}(x_0)]\times (\mft_1,\mft_1+\de\varrho^{sp}]$.

		We now invoke \cref{lem:shrinking} (in which we replace $u$ with $u - \bsmu^-$) with $l =  \epsilon \eta M 2^{-j}$ and $m = \epsilon \eta M$ to conclude that the following
		\[
		|\{u-\bsmu^{-} < \epsilon\eta M2^{-(j+1)}\}\cap\widehat{ \mcq}_1| \leq \left(\frac{\bsc}{2^{j}-1}\right)^{p-1}|\widehat{ \mcq}_1|,
		\]
		holds for a constant $\bsc>0$ depending only on $\alpha$ and data.
		\descitemnormal{Step 5:}{step5} In this step, we apply \cref{Lm:3:3} with cut-off function depending only on the space variable $x$ along with \cref{remark_degiorgi} to find a $\widehat\nu_1 = \widehat\nu_1(\datanb{},\de) \in (0,1)$ such that if 
		\begin{equation*}
			\left|\left\{
			\pm\left(\bsmu^{\pm}-u\right)\le \epsilon \eta M2^{-(j+1)}\right\}\cap \widehat\mcq\right|
			\le
			\widehat\nu_1|\widehat\mcq|,
		\end{equation*}
		holds where $\widehat{{\mcq}}=B_{4\varrho}(x_0)]\times (\mft_1,\mft_1+\de\varrho^{sp}]$, along with the assumptions 
		\begin{equation*}
			\tail((u-\bsmu^{\pm})_{\pm};x_0,\varrho,(\mft_1,\mft_1+\de\varrho^{sp}]) \leq  \epsilon \eta M2^{-(j+1)} \qquad \text{and} \qquad |\bsmu^{\pm}|\leq 8M,
		\end{equation*}
		then the following conclusion follows:
		\begin{equation*}
			\pm\lbr\bsmu^{\pm}-u\rbr\ge\tfrac{1}2 \epsilon \eta M2^{-(j+1)}
			\quad
			\mbox{ on }\quad  B_{2\varrho}(x_o)\times \left(\mft_1+\tfrac{\de}{2}\varrho^{sp}, \mft_1+\de\varrho^{sp}\right].
		\end{equation*}
		Recall that $\mft_1 \in \left(t_0+\tfrac{\de}{2}\varrho^{sp},t_0+\de{\rho}^{sp}\right]$ as in \cref{fig5}. 
		\descitemnormal{Step 6:}{step6} In particular, by varying $\mft_1$ over $\left(t_0+\tfrac{\de}{2}\varrho^{sp},t_0+\de{\rho}^{sp}\right]$, we see that \descrefnormal{step5}{Step 5} gives 	
		\begin{equation}\label{second_iter_gen}
			\pm\lbr\bsmu^{\pm}-u\rbr\ge\eta \eta_1M
			\quad
			\mbox{ on }\quad  B_{2\varrho}(x_o)\times \left(t_o+\de\varrho^{sp}, t_o+2\de\varrho^{sp}\right],
		\end{equation}
		where we have set $\eta_1 = \tfrac{1}2 \epsilon 2^{-(j+1)}$.
		
		In particular, we can combine \cref{first_iter_gen} and \cref{second_iter_gen} (see \cref{fig6}) to get
		\begin{equation*}
			\pm\lbr\bsmu^{\pm}-u\rbr\ge\eta \eta_1M
			\quad
			\mbox{ on }\quad  B_{2\varrho}(x_o)\times \left(t_o+\tfrac{\de}{2}\varrho^{sp}, t_o+2\de\varrho^{sp}\right].
		\end{equation*}
			\begin{figure}[ht]				
				\begin{center}
					\begin{tikzpicture}[line cap=round,line join=round,>=latex,scale=0.5]
						\coordinate  (O) at (0,0);
						\draw[thick, draw=teal,pattern=north west lines, pattern color=teal, opacity=0.3] (-7,0) rectangle (7,3);
						
							\draw[thick, draw=orange,pattern=north west lines, pattern color=orange, opacity=0.3] (-7,-2) rectangle (7,0);
						\node [anchor=south] at (9.2,2.5) {\scriptsize \textcolor{teal}{$t=t_o + 2\de \varrho^{sp}$}};
						
						\draw[thick, draw=blue] (-3,-5) -- (3,-5);
						\draw[thick, draw=blue, dashed] (-3,-5) -- (-3,4);
						\draw[thick, draw=blue, dashed] (3,-5) -- (3,4);

						\draw[draw=black,thick] (-3,-3.5) -- (3,-3.5);
						\draw[draw=blue, dashed, <->] (-3,-5.5) -- (3,-5.5);
						\draw[draw=blue, dashed, <->] (-7,-3) -- (7,-3);
		
						\node  at (0,-6) {\scriptsize $\varrho$};
						\node  at (0,-2.5) {\scriptsize $2\varrho$};
						\node  at (4,-3.5) {\scriptsize $t=t_o$};
						\node  [anchor=south] at (9,-2.5) {\scriptsize \textcolor{orange}{$t=t_o+ \tfrac{\de}{2} \varrho^{sp}$}};
						\node  [anchor=south] at (9,-0.5) {\scriptsize \textcolor{orange}{$t=t_o+ \de \varrho^{sp}$}};

\node at (0,-1) {\scriptsize \textcolor{orange}{$\pm\lbr\bsmu^{\pm}-u\rbr\ge\eta M$}};

\node at (0,1.5) {\scriptsize \textcolor{teal}{$\pm\lbr\bsmu^{\pm}-u\rbr\ge\eta \eta_1M$}};
					\end{tikzpicture}
				\end{center}
				\caption{Application of \descrefnormal{step6}{Step 6}}
				\label{fig6}
			\end{figure}
	\end{description}

	Once we have this, the rest of the proof follows as in the proof of \cref{Prop:1:1}. 
	
\end{proof}

\section{Reduction of oscillation near zero: Degenerate Case}\label{section5}

In this case, we assume $p \geq 2$ and prove H\"older regularity. 

\subsection{Setting up the geometry}
Let us fix some reference cylinder $\mbcq \subset \omt$. Fix $(x_o,t_o)\in \omt$ and let ${\mcq}_o={\mcq}_{\varrho} = B_{\varrho}(x_0) \times (t_0-A\varrho^{sp},t_0] \subset \mbcq $ where $A \geq 1$ will be chosen later in terms of the data, be any cylinder centred at $(x_o,t_o)$. 
Assume, without loss of generality, that $(x_o,t_o) = (0,0)$ and let
\begin{equation}\label{defbsom}
\bsom := 2\esssup_{\mbcq} |u| +\tail(|u|; \mbcq), \quad \bsmu^+:=\esssup_{{\mcq_o}}u,
\txt{and}
\bsmu^-:=\essinf_{{\mcq_o}}u,
\end{equation}
then it is easy to see that 
\[
\essosc_{\mcq_o} u   = \bsmu^+ - \bsmu^-\leq \bsom.
\]

For some positive constant $\tau \in (0,\tfrac14)$ that will be made precise later on, we have two cases
\refstepcounter{equation}
\def\myname{\theequation}
\begin{align}[left=\empheqlbrace]
		&\mbox{when $u$ is \emph{near} zero:}  \bsmu^-\le\tau\bsom  \,\,\text{and} \,\,
			\bsmu^+\ge-\tau\bsom,\tag*{(\myname$_{a}$)}\label{Eq:Hp-main1}\\
		&\mbox{when $u$ is \emph{away} from zero:} \bsmu^- >\tau\bsom \,\,\text{or}\,\, \bsmu^+<-\tau\bsom.\tag*{(\myname$_{b}$)}\label{Eq:Hp-main2}
\end{align}
Furthermore, we will always assume the following is satisfied:
\begin{equation}\label{Eq:mu-pm-}
	\bsmu^+ -\bsmu^- >\tfrac12\bsom.
\end{equation} since in the other case, we trivially get the reduction of oscillation.


\subsection{Reduction of Oscillation near zero when \texorpdfstring{\tlcref{Eq:Hp-main1}}. holds}

\begin{description}
	\item[Lower bound for $\bsmu^-$:] In this case, we have
	\begin{equation}\label{deg5.3a}
		\bsmu^- = \bsmu^- - \bsmu^+ + \bsmu^+ \geq -\bsom + \bsmu^+ \geq -(1+\tau)\bsom.
	\end{equation}
	\item[Upper bound for $\bsmu^+$:] In this case, we have
	\begin{equation}\label{deg5.3b}
		\bsmu^+ = \bsmu^+ - \bsmu^- + \bsmu^- \leq \bsom + \bsmu^- \leq (1+\tau)\bsom.
	\end{equation}
\end{description}
In particular, \tlcref{Eq:Hp-main1} implies
\begin{equation}\label{boundmu}
	|\bsmu^{\pm}| \leq (1+\tau)\bsom.
\end{equation}

\begin{assumption}\label{assump1}
	Let us suppose that there exists $\mft \in (-(A-1)(c_o\varrho)^{sp},0]$ such that 
	\begin{equation*}
		\abs{\left\{u\le\bsmu^-+\tfrac{\bsom}{4} \right\}
		\cap 
		{\mcq}_{c_o\varrho}(0,\bar{t})}\le \nu|{\mcq}_{c_o\varrho}|,
	\end{equation*}
where $\nu$ is the constant determined in \cref{Lm:3:3} depending on $\data{}$ and the constant $c_o$ is to be determined according to \cref{degclaim5.2} and \cref{Lm:6:1}. Here we have taken ${\mcq}_{c_o\varrho}(0,\mft) = B_{c_o\varrho}(0) \times ( \mft-(c_o\varrho)^{sp},\mft]$.
\end{assumption}
	\begin{figure}[ht]
	\begin{center}
		\begin{tikzpicture}[line cap=round,line join=round,>=latex,scale=0.5]
			\coordinate  (O) at (0,0);
			
			
			\draw[thick, draw=orange, pattern color=orange,pattern=north east lines, opacity=0.3] (-1,2) rectangle (1,1);
			
			\draw[thick, draw=blue,pattern=north east lines, pattern color=blue, opacity=1] (-0.5,2) rectangle (0.5,1.5);
			
			\node  [anchor=west] at (3,2) {\scriptsize \textcolor{orange}{$t=\mft$}};
			\draw[dotted,->] (1,2) -- (3,2);
			\draw[dotted,->] (1,1) -- (3,1);
			\node  [anchor=west] at (3,1) {\scriptsize \textcolor{orange}{$t=\mft-(c_o\varrho)^{sp}$}};
			

			\node [anchor=east] at (-1,1.5) {\scriptsize{$\mcq_{c_o\varrho}$}};
			\node [anchor=east] at (-3,0) {{$\mcq_o$}};
			
			\draw[thick, draw=blue, dashed] (-3,-3.5) -- (-3,6);
			\draw[thick, draw=blue, dashed] (3,-3.5) -- (3,6);
			
			\draw[thick, draw=orange, dashed] (-1,-3.5) -- (-1,6);
			\draw[thick, draw=orange, dashed] (1,-3.5) -- (1,6);
			
			
			\draw[draw=black,thick] (-3,-3.5) -- (3,-3.5);
			
			\draw[draw=black,thick] (-3,5.5) -- (3,5.5);
			\node  [anchor=west] at (3,5.5) {\scriptsize \textcolor{black}{$t=t_o$}};
			
			\draw[draw=blue, dashed, <->] (-3,-4) -- (3,-4);
			\draw[draw=orange, dashed, <->] (-1,-3) -- (1,-3);
			
			%
			\node  at (0,-4.5) {\scriptsize $\varrho$};
			\node  at (0,-2.5) {\scriptsize $c_o\varrho$};
			\node [anchor=west] at (3,-3.5) {\scriptsize $t_o- A\varrho^{sp}$};
		\end{tikzpicture}
	\end{center}
	\caption{\cref{assump1}}
	\label{fig12}
\end{figure}


\subsubsection{Reduction of Oscillation when \texorpdfstring{$-\tau \bsom \leq \bsmu ^-\leq \tau \om$}. for subsolutions near its infimum}

Since $\tau \in (0,\tfrac14)$, we see that $|\bsmu^-| \leq (1+\tau)\om$ and hence the second condition  in \cref{Lm:3:3:hypothesis} is automatically satisfied with $M = \tfrac{\bsom}{4}$. 
 
 \begin{claim}\label{degclaim5.2}
 	There exists constant $c_o(\datanb{}) \in (0,1)$  such that $\tail((u-\bsmu^{-})_{-};0;c_o\varrho;(\mft - (c_o\varrho)^{sp},\mft] \leq \tfrac{\bsom}{4}$.
 	\end{claim}
 \begin{proof}[Proof of \cref{degclaim5.2}]
 From \cref{rmk:4.2} and \cref{defbsom}, we have
 \begin{equation}\label{deg5.5}
 \tail((u-\bsmu^{\pm})_{\pm};0;c_o\varrho;(\mft - (c_o\varrho)^{sp},\mft]) \leq c_o^{\frac{sp}{p-1}}\tail((u-\bsmu^{\pm})_{\pm};0;\varrho;( \mft - (c_o\varrho)^{sp},\mft]).
 \end{equation}
 Furthermore, we have the following sequence of estimates:
 \begin{equation}\label{deg5.6}
 	\begin{array}{rcl}
 		\tail((u-\bsmu^{-})_{-};0;\varrho;(\mft, \mft + \varrho^{sp}])^{p-1} & = & \esssup_{t \in(\mft, \mft + (c_o\varrho)^{sp}) }\varrho^{sp} \int_{\RR^n \setminus B_{\varrho}} \frac{(u-\bsmu^-)_-^{p-1}}{|x|^{n+sp}} \,dx\\
 		& \apprle  &  (\bsmu^-)^{p-1} + 
 		 \tail(u_{-};\mbcq)^{p-1} \\
 		 &&+ \esssup_{t \in(\mft - (c_o\varrho)^{sp},\mft] }\varrho^{sp} \int_{B_{\tilde{R}} \setminus B_{\varrho}} \frac{u_-^{p-1}}{|x|^{n+sp}} \,dx\\
 		 & \apprle & \bsom^{p-1},
 		\end{array}
 \end{equation}
where we have used $\bsmu^- \leq \bsom$ from \cref{boundmu}, $\tail(u_{-};\mbcq) \leq \bsom$ from \cref{defbsom}, $\tau \leq \tfrac14$ and $u_- \leq \bsom$ on $\mbcq$ holds due to $u_- \leq |\bsmu^-|$ and \cref{boundmu}. 

Combining \cref{deg5.5} and \cref{deg5.6}, we can now choose $c_o$ small enough such that the claim follows. 
 \end{proof}
 
 \begin{remark}
 	\cref{degclaim5.2} does not see the time interval and instead, we actually obtain a constant $c_o = c_o(\datanb{})$ small such that $\tail((u-\bsmu^{-})_{-};0;c_o\varrho;(-(A-	1)\varrho^{sp},0]) \leq \tfrac{\bsom}{4}$.
 \end{remark}


Thus both the conditions in \cref{Lm:3:3:hypothesis} are satisfied with $M = \tfrac{\bsom}{4}$ and we can apply \cref{Lm:3:3} to conclude (blue shaded region in \cref{fig12}) that
\begin{equation}\label{Eq:lower-bd}
	u\ge\bsmu^-+\tfrac{1}{8}\bsom
	\quad
	\mbox{a.e.~in }\,  \mcq_{\frac{c_o\varrho}{2}} = B_{\frac{c_o\varrho}2} \times (\mft- \lbr \tfrac{1}{2}c_o\varrho\rbr^{sp}, \mft].
\end{equation}

\begin{remark}
	Making use of \cref{remark_degiorgi}, instead of the time interval $(\mft- ( \tfrac{c_o\varrho}{2})^{sp}, \mft]$ in \cref{Eq:lower-bd}, we can also get $(\mft-  \tfrac{1}{2}(c_o\varrho)^{sp}, \mft]$, see \cref{remark_degiorgi} for this modification. 
\end{remark}
\begin{lemma}\label{degclaim5.3}
	There exist constants $\bar\eta = \bar\eta(\datanb{,A}) \in (0,1)$, $\eta=\eta(\datanb{})\in(0,1)$  and $c_o=c_o(\datanb{,A})$ such that if we take $\tau \leq 2\eta$ in \tlcref{Eq:Hp-main1} and  $|\bsmu^-| \leq \tau \bsom$, then we have 
	\[
	u \geq \bsmu^- + \bar\eta \bsom \txt{in} B_{\frac12 c_o\varrho} \times (t_o -\tfrac14 (c_o\varrho)^{sp},t_o],
	\]
	Here $\de = \de(\datanb{})$  is obtained by applying \cref{Prop:1:1} with $\al =1$. 
\end{lemma}
\begin{proof}[Proof of \cref{degclaim5.3}] Since \cref{Eq:lower-bd} holds (see also \cref{remark_degiorgi}), we see that the measure hypothesis in \cref{Prop:1:1} holds with $\al =1$ at the time level $t = \mft - \tfrac12\lbr c_o\varrho\rbr^{sp}$.

Let $\de = \de(\datanb{})$ (independent of $A$ at this instance) and $\eta = \eta(\datanb{})$ be as obtained in \cref{Prop:1:1} and $c_o = c_o(\datanb{})$ from \cref{degclaim5.2}, then we can further impose $\tfrac{\de}{2} \leq  \tfrac14$.
%
%
%
With this modification, we can now apply \cref{iteratedprop} and \cref{remark_degiorgi} to obtain an $\bar\eta = \bar\eta(\datanb{,A}) \in (0,1)$ such that we have
\[
u \geq \bsmu^- + \bar\eta \bsom \txt{in} B_{\frac12 c_o\varrho} \times (\mft -\tfrac12 (c_o\varrho)^{sp} + \tfrac{\de}{2}(c_o\varrho)^{sp},t_o],
\]
provided the $\tail$ alternative in the hypothesis of \cref{iteratedprop} holds. This can be ensured by now choosing $c_o=c_o(\datanb{,A})$ small and applying \cref{degclaim5.2}. Note that this is where $c_o$ additionally depends on $A$. 

Since $\mft$ is an arbitrary time level, we see that the following inclusion holds: 
\[
(t_o -\tfrac14 (c_o\varrho)^{sp},t_o] \subset (\mft -\tfrac12 (c_o\varrho)^{sp} + \tfrac{\de}{2}(c_o\varrho)^{sp},t_o],
\]
from which  the conclusion follows.
\end{proof}
\begin{remark}
Having obtained $\eta$ and hence $\tau$ according to \cref{degclaim5.3}, we are left to consider the case $\bsmu ^-\leq -\tau \bsom$ to complete the reduction of oscillation when \tlcref{Eq:Hp-main1} holds. In particular, this corresponds to the case $|\bsmu^-| \leq \tau \bsom$ does not hold in \cref{Prop:1:eq}.
\end{remark}
\subsubsection{Reduction of Oscillation when \texorpdfstring{$\bsmu ^-\leq -\tau \bsom$}. for subsolutions near its infimum}

Due to the restriction on $\bsmu^+$ in \tlcref{Eq:Hp-main1}, from \cref{deg5.3a},  we  have $\bsmu^-\geq-(1+\tau)\bsom > -2\bsom$.
Thus,  we proceed further with the assumptions
\begin{equation}\label{Eq:H1}
	\left\{
	{\def\arraystretch{1}\begin{array}{c}
		-2\bsom<\bsmu^-<-\tau\bsom,\\
		 u(\cdot, \mft-(\tfrac{c_o\varrho}2 )^{sp}]\ge\bsmu^-+\tfrac{1}8\bsom\;\;\mbox{a.e.~in}\,\, B_{c_o\frac{\varrho}{2}},
	\end{array}}
	\right.
\end{equation}
which holds due to \cref{Eq:lower-bd}.  

In the next lemma we establish that the pointwise information in \cref{Eq:H1}
propagates to the top of the cylinder using de Giorgi iteration. The novelty of the lemma is that we can use \cref{Eq:H1} to use cut-off functions independent of time to obtain a pointwise expansion of positivity type result.
\begin{lemma}\label{Lm:6:1}
	Suppose \cref{Eq:H1} and \cref{Eq:mu-pm-} holds,
	then there exists constants $\eta_1 = \eta_1(\datanb{,A,\tau})\in(0,1)$ and $c_o = c_o(\datanb{,\eta_1}) \in (0,1)$ such that the following holds
	\[
	u\ge\bsmu^-+\eta_1\bsom\quad\mbox{a.e. in}\,\, 
	B_{\frac14 c_o\varrho}\times(\mft-(\tfrac12 c_o \varrho)^{sp},t_o].
	\]
	provided
	\begin{equation}\label{degtailh}
	\tail((u-\bsmu^{-})_{-};0;\tfrac14 c_o\varrho;(\mft-(\tfrac12 c_o \varrho)^{sp},t_o]) \leq \eta_1  \bsom,
	\end{equation}
	As a result, we have a reduction of oscillation 
	\[
	\essosc_{\widehat{{\mcq}}_1}u\le(1-\eta_1)\bsom\quad\mbox{where $\widehat{{\mcq}}_1
		=B_{\frac14 c_o\varrho}\times(-(\tfrac12 c_o\varrho)^{sp},0]$.}
	\]
\end{lemma}
\begin{proof}
	Without loss of generality, we will take $t_o =0$. Assuming $\eta_1$ is already determined, we see that the calculations from \cref{degclaim5.2} gives the existence of  $c_o = c_o(\datanb{,\eta_1})\in(0,1)$ small enough such that \cref{degtailh} holds.  For $j=0,1,\ldots$, define
		\begin{equation*}
			\def\arraystretch{\st}\begin{array}{llll}
					k_j:=\bsmu^-+{\eta_1\bsom}+\tfrac{\eta_1\bsom}{2^{j}},& \tilde{k}_j:=\tfrac{k_j+k_{j+1}}2,&
					\varrho_j:=\tfrac{c_o\varrho}4+\tfrac{c_o\varrho}{4^{j+1}},
					&\tilde{\varrho}_j:=\tfrac{\varrho_j+\varrho_{j+1}}2,\\
					B_j:=B_{\varrho_j},& \widetilde{B}_j:=B_{\tilde{\varrho}_j},&
					\mcq_j:= B_j \times (\mft - (\varrho_j)^{sp}),0],&\\
					\widetilde{\mcq}_j:=\widetilde{B}_j \times (\mft - (\tilde\varrho_j)^{sp}),0].
			\end{array}
		\end{equation*} 
		Furthermore, we also define 
		\begin{equation*}
			\hat{\varrho}_j:=\tfrac{3\varrho_j+\varrho_{j+1}}4, \quad 
			\bar{\varrho}_j:=\tfrac{\varrho_j+3\varrho_{j+1}}4,\quad 
			\hat{\mcq}_j:=\mcq_{\hat\varrho_j}, \quad
			\bar{\mcq}_j:=\mcq_{\bar\varrho_j}.
		\end{equation*}
		We now consider a cut-off functions $\bar{\zeta_j}$ and $\zeta_j$ independent of time such that
		\begin{equation*}
			\begin{array}{c}
				\bar{\zeta_j} \equiv 1 \text{ on } B_{j+1}, \quad \bar{\zeta_j} \in C_c^{\infty}(B_{\bar{\varrho_j}}), \quad  |\nabla\bar{\zeta_j}| \apprle \frac{1}{\bar{\varrho_j} - \varrho_{j+1}} \approx \frac{2^j}{c_o\varrho}, \\
				{\zeta_j} \equiv 1 \text{ on } B_{\tilde{\varrho_j}}, \quad {\zeta_j} \in C_c^{\infty}(B_{\hat{\varrho_j}}), \quad  |\nabla{\zeta_j}| \apprle \frac{1}{\hat{\varrho_n} - \tilde{\varrho_{j}}}\approx \frac{2^{j}}{c_o\varrho},
			\end{array}
		\end{equation*}
	We note that by \cref{lem:g},  we have
	\begin{equation*}
		\mathfrak g_-(u,k)
		\le
		\bsc  \lbr|u|+|k|\rbr^{p-2}(u-k)_-^2.
	\end{equation*}

Next, we make the following observations 
\begin{itemize}
	\item For $\tilde k<k$, there holds $(u-k)_-\ge (u-\tilde k)_-$, 
	\item We trivially have $|u|+|k_j|\geq |k_j| = |\bsmu^- + \eta_1\bsom + \tfrac{\eta_1\bsom}{2^j}|$.   Since $-2\bsom\leq \bsmu^- \leq - \tau \bsom$ holds, we get
	\[
	-2\bsom \leq -2\bsom + \eta_1\bsom + \tfrac{\eta_1\bsom}{2^j} \overlabel{Eq:H1}{\leq} \underbrace{\bsmu^-  + \eta_1\bsom + \tfrac{\eta_1\bsom}{2^j}}_{= k_j} \leq \bsmu^- + 2\eta_1 \bsom \overlabel{Eq:H1}{\leq} -\tau \bsom  + 2\eta_1 \bsom \leq -\tfrac{\tau\bsom}{2},
	\]
	provided $2\eta_1 \leq \tfrac{\tau}{2}$. In particular, this implies that $|k_j| \geq \tfrac{\tau\bsom}{2}$.
	\item We have $(u-k_j)_-(\cdot,\mft-(\tfrac12 c_o \varrho)^{sp}) = 0$ provided $2\eta_1 \leq \tfrac18$, since 
	\[
	u(\cdot,\mft-(\tfrac12 c_o \varrho)^{sp}) \overset{\cref{Eq:H1}}{\geq} \bsmu^- + \tfrac{\bsom}{8} \geq \bsmu^- + 2\eta_1\bsom \geq \bsmu^-  + \eta_1\bsom + \tfrac{\eta_1\bsom}{2^j} = k_j.
	\] 
	\item Let us denote $A_j := \{ u < k_j\} \cap \mcq_j$.
\end{itemize}
%
		%
		%
		Now, we  apply the energy estimates from \cref{Prop:energy} over  $\mcq_j $ to $ (u-\tilde{k}_j)_{-}$ and $\zeta_j$ to get
		\begin{multline*}
		\bsom^{p-2}\hspace*{-0.5cm}	\underset{\mft-(\frac12 c_o \tilde\varrho)^{sp} < t < 0}{\esssup}\int_{B_j}(u-\tilde{k}_j)_{-}^2\zeta_j^p(x,t)\,dx  
			+ {\iiint_{\mcq_j}}\frac{|(u-\tilde{k}_j)_{-}(x,t)\zeta_j(x,t)-(u-\tilde{k}_j)_{-}\zeta_j(y,t)|^p}{|x-y|^{n+sp}}\,dx \,dy\,dt
			\\  \leq_{\data{}}  \frac{2^{jp}}{(c_o\varrho)^{sp}}(\eta_1 \bsom)^p |A_j| 
				+ \lbr \underset{\stackrel{\mft-(\frac12 c_o \varrho)^{sp} < t < 0;}{ x\in \spt \zeta_j}}{\esssup}\int_{\RR^n \setminus B_j}\frac{(u-\tilde{k}_j)_{-}^{p-1}(y,t)}{|x-y|^{n+sp}}\,dy\rbr   (\eta_1 \bsom) |A_j|.
		\end{multline*}
		
		We estimate the Tail as below following the what was describes in \cref{sec:tail} to get
		\begin{equation*}
			{\def\arraystretch{1}	\begin{array}{rcl}
					\underset{\stackrel{\mft-(\frac12 c_o \varrho)^{sp} \leq t \leq 0;}{ x\in \spt \zeta_j}}{\esssup}\int_{\RR^n \setminus B_j}\frac{(u-\tilde{k}_j)_{-}^{p-1}(y,t)}{|x-y|^{n+sp}}\,dy & \apprle & 2^{jsp}\frac{(\eta_1\bsom)^{p-1}}{(c_o\varrho)^{sp}} \\&&+  \frac{2^{jsp}}{(c_o\varrho)^{sp}}\tailp((u-\bsmu^{-})_{-};0;\tfrac14c_o\varrho;(\mft-(\tfrac12 c_o \varrho)^{sp}, 0])
					\\
					&\overset{\cref{degtailh}}{\apprle} & 2^{jsp}\frac{(\eta_1\bsom)^{p-1}}{(c_o\varrho)^{sp}}.
			\end{array}}
		\end{equation*}
		Combining the previous two estimates along with the observations, we have
		\begin{equation*}
				\bsom^{p-2}	\underset{\mft-(\frac12 c_o \tilde\varrho)^{sp} < t < 0}{\esssup}
				\int_{\widetilde{B}_j} (u-\tilde{k}_j)_-^2\,dx
				+\iiint_{\widetilde{\mcq}_j}\frac{|(u-\tilde{k}_j)_{-}(x,t)-(u-\tilde{k}_j)_{-}|^p}{|x-y|^{n+sp}}\,dx \,dy\,dt
				{\leq}
				\bsc  \frac{2^{jsp}}{(c_o\varrho)^{sp}}(\eta_1\bsom)^{p}|A_j|.
		\end{equation*}
		From Young's inequality, we have
		\begin{multline*}
			|(u-\tilde{k}_{j})_- \bar{\zeta_j}(x,t) - (u-\tilde{k}_{j})_- \bar{\zeta_j}(y,t)|^p \leq c |(u-\tilde{k}_{j})_- (x,t) - (u-\tilde{k}_{j})_- (y,t)|^p\bar{\zeta_j}^p(x,t) \\
			+ c |(u-\tilde{k}_{j})_-(y,t)|^p |\bar{\zeta_j}(x,t) - \bar{\zeta_j}(y,t)|^p,
		\end{multline*}
		from which we obtain the following sequence of estimates:
		\begin{equation*}
			\begin{array}{rcl}
				\frac{\eta_1\bsom}{2^{j+3}}
				|A_{j+1}|
				&\overred{3.7a}{a}{\leq} &
				\iint_{\bar{\mcq}_{j+1}}\lbr u-\tilde{k}_j\rbr_-\bar{\zeta_j}\,dx\,dt \\
				&\overred{3.7b}{b}{\leq} &
				\left[\iint_{\widetilde{\mcq}_j}\big[\lbr u-\tilde{k}_j\rbr_-\bar{\zeta_j}\right]^{p\frac{n+2s}{n}}
				\,dx\,dt\bigg]^{\frac{n}{p(n+2s)}}|A_j|^{1-\frac{n}{p(n+2s)}}\\
				&\overred{3.7c}{c}{\leq} &\bsc 
				\left(\iiint_{\widetilde{\mcq}_j}\frac{|(u-\tilde{k}_j)_{-}(x,t)\bar{\zeta_j}(x,t)-(u-\tilde{k}_j)_{-}\bar{\zeta_j}(y,t)|^p}{|x-y|^{n+sp}}\,dx \,dy\,dt\right)^{\frac{n}{p(n+2s)}} 
				\\&&\qquad\times \left(\underset{\mft-(\frac12 c_o \tilde\varrho)^{sp} < t < 0}{\esssup}\int_{\widetilde{B}_j}[(u-\tilde{k}_j)_{-}\bar\zeta_j(x,t)]^2\,dx\right)^{\frac{s}{n+2s}}|A_j|^{1-\frac{n}{p(n+2s)}}\\
				&\overred{3.7d}{d}{\leq}&
				\bsc  
				\lbr \frac{2^{jsp}}{(c_o\varrho)^{sp}}(\eta_1\bsom)^p\rbr^{\frac{n}{p(n+2s)}}
				\lbr\frac{2^{jsp}}{(c_o\varrho)^{sp}}\eta_1^p \bsom^2\rbr^{\frac{s}{n+2s}}
				|A_j|^{1+\frac{s}{n+2s}} \\
				&\overred{3.7e}{e}{\leq} &
				\bsc 
				\frac{b_0^j}{(c_o\varrho)^\frac{s(n+sp)}{n+2s}} \eta_1^{\frac{n+sp}{n+2s}}
				\bsom |A_j|^{1+\frac{s}{n+2s}},
			\end{array}
		\end{equation*}
		where to obtain \redref{3.7a}{a}, we made use of the observations and enlarged the domain of integration with $\tilde{\zeta}_j$ as defined in \cref{cutoff_size}; to obtain \redref{3.7b}{b}, we applied H\"older's inequality; to obtain \redref{3.7c}{c}, we applied \cref{fracpoin}; to obtain \redref{3.7d}{d}, we made use of \cref{Eq:3.6en}, \cref{Eq:sample} and \cref{eq3.9}  with $\bsc = \bsc_{\data{}}$ and finally we collected all the terms to obtain \redref{3.7e}{e}, where $b_0\geq 1$ is a universal constant. We see that 
		\[
		|\mcq_j| = (c_o\varrho)^n |\mft - (\tfrac12 c_o \varrho)^{sp}| \leq (c_o\varrho)^n 2^p A (\tfrac12 c_o \varrho)^{sp},
		\]
		where we have used the fact that $\mft \in (-(A-1)(c_o\varrho)^{sp},0]$ from \cref{assump1}. Setting
		$\bsy_j=|A_j|/|\mcq_j|$, we get
		\begin{equation*}
			\bsy_{j+1}
			\le
			\bsc  \boldsymbol b^j (\eta_1^{p-2}A)^{\frac{s}{n+2s}} \bsy_j^{1+\frac{s}{n+2s}},
		\end{equation*}
		for  constants $\bsc  = \bsc_{\data{}} \geq 1$ and $\boldsymbol b\geq  1$ depending only on the data. We see from \cref{geo_con} that if 
		\[
		\bsy_0 \leq \bsc^{-\frac{(n+2s)}{s}}\boldsymbol b^{-\frac{(n+2s)^2}{s^2}} \frac{\eta_1^{2-p}}{A} \quad \Longrightarrow \quad  \lim_{j\rightarrow \infty}\bsy_{j} \rightarrow 0.
		\]
		Since $p \geq 2$, we can choose $\eta_1 = \eta_1(\datanb{,A})$ small such that $\bsc^{-\frac{(n+2s)}{s}}\boldsymbol b^{-\frac{(n+2s)^2}{s^2}} \frac{\eta_1^{2-p}}{A} \geq 1$, which completes the proof. Note that $\eta_1$ is independent of $c_o \varrho$.
\end{proof}

\subsubsection{Reduction of Oscillation for supersolutions near its supremum}

In this subsection, we still assume \tlcref{Eq:Hp-main1} still holds, but \cref{assump1} isn't satisfied. In particular, we will consider the case
\begin{assumption}\label{assump2}
	Let us suppose that for any $\mft \in (-(A-1)(c_o\varrho)^{sp},0]$, there holds
	\begin{equation}\label{degEq:1st-alt-meas}
		\abs{\left\{u\le\bsmu^-+\tfrac{\bsom}{4} \right\}
			\cap 
			{\mcq}_{c_o\varrho}(0,\mft)}> \nu|{\mcq}_{c_o\varrho}|,
	\end{equation}
	where $\nu$ is the constant determined in \cref{Lm:3:3} depending on $\data{,A}$ and the constant $c_o = c_o(\datanb{})$ is to be further determined according to  \cref{deggclaim5.2} and the cylinder ${\mcq}_{c_o\varrho}(0,\mft) := B_{c_o\varrho}(0) \times ( \mft-(c_o\varrho)^{sp},\mft]$.
\end{assumption}
\subsubsection{Reduction of Oscillation when \texorpdfstring{$-\tau\bsom \leq \bsmu^+ \leq \tau \bsom$}. for supersolutions near its supremum}

\begin{claim} \label{claim3}
	For any  $\mft$ from \cref{assump2}, there exists some $\htt\in[
	\mft-(c_o\varrho)^{sp},\mft-\tfrac{1}2\nu(c_o\varrho)^{sp}]
	$
	with
	\begin{equation*}
		\left|\left\{u(\cdot, \htt)\le\bsmu^-+\tfrac{1}4\bsom\right\}\cap B_{c_o\varrho}\right|>\tfrac{1}2\nu |B_{c_o\varrho}|.
	\end{equation*}
\end{claim}
\begin{proof}[Proof of \cref{claim3}]
	Indeed, if the claim does not hold for any $\htt\in[
	\mft-(c_o\varrho)^{sp},\mft-\tfrac{1}2\nu(c_o\varrho)^{sp}]$, then 
	\begin{equation*}
		\begin{array}{rcl}
		\left|\left\{u\le\bsmu^-+\tfrac{1}4\bsom\right\}\cap {\mcq}_{c_o\varrho}(0,\mft)\right|
		&=&
		\int_{\mft-(c_o\varrho)^{sp}}^{\mft-\frac12\nu (c_o\varrho)^{sp}}\left|\left\{u(\cdot, t)
		\le\bsmu^-+\tfrac{1}4\bsom\right\}\cap B_{c_o\varrho}\right|\,dt\\
		&&
		+\int^{\mft}_{\mft-\frac12\nu(c_o\varrho)^{sp}}\left|\left\{u(\cdot, t)\le\bsmu^-+\tfrac{1}4\bsom\right\}\cap B_{c_o\varrho}\right|\,dt\\
		&\leq& \tfrac12\nu |B_{c_o\varrho}| \abs{(c_o\varrho)^{sp}-\tfrac12\nu(c_o\varrho)^{sp}} +\tfrac12\nu(c_o\varrho)^{sp}
		|B_{c_o\varrho} |
		\\&\leq  &\nu |{\mcq}_{c_o\varrho}|,
		\end{array}
	\end{equation*}
	contradicting \cref{degEq:1st-alt-meas}. This completes the proof of the claim. 
\end{proof}

\begin{remark}\label{remark5.10}
From \cref{Eq:mu-pm-}, we have $\bsmu^+-\tfrac14\bsom>\bsmu^-+\tfrac14\bsom$, from which we can rewrite the conclusion of \cref{claim3} to be
\begin{equation*}
	\left|\left\{u(\cdot,\htt)\le\bsmu^+-\tfrac{1}4\bsom\right\}\cap B_{c_o\varrho}\right|
	\ge\tfrac12\nu|B_{c_o\varrho}| \quad \text{for some}\quad  \htt\in[
	\mft-(c_o\varrho)^{sp},\mft-\tfrac{1}2\nu(c_o\varrho)^{sp}].
\end{equation*}
\end{remark}

\begin{claim}\label{deggclaim5.2}
	There exists constant $c_o = c_o(n,p,s)$ small such that $\tail((u-\bsmu^{+})_{+};0;c_o\varrho;(\mft - (c_o\varrho)^{sp},\mft]) \leq \tfrac{\bsom}{4}$.
\end{claim}
\begin{proof}[Proof of \cref{deggclaim5.2}]
	From \cref{rmk:4.2}, for any time interval $I \subseteq (-(A-1)(c_o\varrho)^{sp},0]$, we have
	\begin{equation}\label{degg5.5}
		\tail((u-\bsmu^{\pm})_{\pm};0;c_o\varrho;I) \leq c_o^{\frac{sp}{p-1}}\tail((u-\bsmu^{\pm})_{\pm};0;\varrho;I).
	\end{equation}
	Furthermore, we have the following sequence of estimates:
	\begin{equation}\label{degg5.6}
		\begin{array}{rcl}
			\tail((u-\bsmu^{+})_{+};0;\varrho;I)^{p-1} & = & \esssup_{t \in I }\varrho^{sp} \int_{\RR^n \setminus B_{\varrho}} \frac{(u-\bsmu^+)_+^{p-1}}{|x|^{n+sp}} \,dx\\
			& \apprle  &  (\bsmu^+)^{p-1} + 
			\tail(u_{+};\mbcq) \\
			&&+ \esssup_{t \in(-(A-1)(c_o\varrho)^{sp},0] }\varrho^{sp} \int_{B_{\tilde{R}} \setminus B_{\varrho}} \frac{u_+^{p-1}}{|x|^{n+sp}} \,dx\\
			& \apprle & \bsom^{p-1},
		\end{array}
	\end{equation}
	where to obtain the last inequality, we have used $\bsmu^+ \leq 2\bsom$ from \cref{boundmu}, $\tail(u_{+};\mbcq) \leq \bsom$ from \cref{defbsom} and $u_+ \leq \bsom$ on $\mbcq$ holds due to \cref{defbsom}. 
	
	Combining \cref{degg5.5} and \cref{degg5.6}, we can now choose $c_o$ small enough such that the claim follows. 
\end{proof}

Since $\tau \in (0,\tfrac14)$, we see that $|\bsmu^{\pm}| \leq (1+\tau)\om$ which follows from \cref{boundmu} and hence the second condition  in \cref{Lm:3:3:hypothesis} is automatically satisfied with $M = \tfrac{\bsom}{4}$. 

\begin{lemma}\label{deggclaim5.3}
	There exist a  constant $\eta_2 = \eta_2(\datanb{})$, $\bar \eta_2 = \bar \eta_2(\datanb{,A})$ and $c_o = c_o(\datanb{,A})$ such that if we take $\tau \leq \tfrac{\eta_2}{2}$ in \tlcref{Eq:Hp-main1} and  $\bsmu^+ \leq \tau \bsom$, then we have 
	\[
	u \leq \bsmu^+ - \bar\eta_2 \bsom \txt{in} B_{\frac12c_o\varrho} \times (- \tfrac{\nu}{4}(\tfrac{c_o\varrho}{2})^{sp}, 0]
	\]
	In particular, this implies a reduction of oscillation
	\begin{equation*}
		\essosc_{B_{\frac12c_o\varrho} \times (- \tfrac{\nu}{4}(\tfrac{c_o\varrho}{2})^{sp}, 0]}u\le(1-\bar\eta_2)\bsom.
	\end{equation*}
\end{lemma}
\begin{proof}
	From \cref{Eq:mu-pm-}, we see that $\bsmu^+ - \tfrac14 \bsom > \bsmu^-+\tfrac14 \bsom$, and hence at the time level $\htt$ as obtained in \cref{claim3}, we can rewrite the conclusion of \cref{claim3} to be
	\begin{equation*}
		\left|\left\{u(\cdot, \htt)\le\bsmu^+-\tfrac{1}4\bsom\right\}\cap B_{c_o\varrho}\right|\geq \tfrac{1}2\nu |B_{c_o\varrho}|.
	\end{equation*}
    
    With this, we can see that the hypothesis of \cref{Prop:1:1} and \cref{iteratedprop} is satisfied with $\al = \tfrac12 \nu$ and hence, there exists $\eta_2 = \eta_2(\datanb{})$, $\bar \eta_2 = \bar \eta_2(\datanb{,A})$ and $c_o = c_o(\datanb{,A})$ such that if the following is satisfied:
    \begin{equation}\label{eq5.22}
    	\tail((u-\mu^{+})_{+};x_0,c_o\varrho,(-(A-1)(c_o\varrho)^{sp},0]) \leq \bar\eta_2 \frac{\bsom}{4} \quad \text{and} \quad \abs{\bsmu^{\pm}}\leq 2 \eta_2 \frac{\bsom}{4},
    \end{equation}
then \cref{iteratedprop} implies
\[
u \leq \bsmu^+ - \bar \eta_2 \tfrac{\bsom}{4} \txt{on} B_{c_o\varrho} \times (\htt + \tfrac{\de}{2}(c_o\varrho)^{sp}, 0).
\]

From \cref{deggclaim5.2}, the choice of $c_o = c_o(\datanb{,\bar\eta_2})$ can be made such that the $\tail$ condition in \cref{eq5.22} is satisfied. Since we are considering the case $-\tau \bsom \leq \bsmu^+ \leq \tau \bsom$, if we choose $\tau \leq 2\eta_2$, then the second condition in \cref{eq5.22} is satisfied.
Since the time level $\htt$ and $\mft$ from \cref{claim3} are arbitrarily located in $(-(A-1)(c_o\varrho)^{sp},0]$, we see that the conclusion holds at least in the region
\[
u \leq \bsmu^+ - \bar \eta_2 \tfrac{\bsom}{4} \txt{on} B_{c_o\varrho} \times (\tfrac{\de-\nu}{2}(c_o\varrho)^{sp}, 0),
\]
provided we choose $\de \leq \tfrac{\nu}{4}$ which can be ensured analogously as in the proof of \cref{degclaim5.3}. This completes the proof of the lemma after replacing $\bar \eta_2$ in the proof with $\tfrac{\bar\eta_2}{4}$ which is in the statement of the lemma. 
%
%
%
\end{proof}


\subsubsection{Reduction of Oscillation for supersolutions when \texorpdfstring{$2\bsom \geq\bsmu^+ \geq \tau \bsom$}.  near its supremum}

 Estimate \cref{deg5.3b} shows that $\bsmu^+ \leq 2\bsom$ and hence we need to consider the last remaining case when \tlcref{Eq:Hp-main1} holds:
\begin{equation}\label{Eq:4:3a}
	\tau\bsom\le\bsmu^+\le2\bsom,
\end{equation}


First we prove an expansion of positivity type estimate:
\begin{lemma}\label{Lm:4:2}
	Suppose \cref{Eq:4:3a} and  \cref{assump2} are in force. Let $\mft$ and $\htt$ be as in \cref{claim3},  then there exists $c_o = c_o(n,p,s,\varepsilon)$ and  $\varepsilon = \varepsilon(\datanb{,\nu,\tau})\in(0,1)$, such that the following two conclusion holds:
\[	\begin{cases*}
		\left|\left\{ u(\cdot, t)\le\bsmu^+-\varepsilon\boldsymbol \om\right\}\cap B_{c_o\varrho}\right|
		\ge
		\tfrac14\nu |B_{c_o\varrho}|
		\txt{for all} t\in(\htt,\mft],\\
		c_o^{\frac{sp}{p-1}} \tail(( u-\bsmu^{+})_{+}; {\mbcq}) \leq   \varepsilon^{p-1}\bsom.
	\end{cases*}
	\]
\end{lemma}
\begin{proof}
	The tail estimate follows from \cref{deggclaim5.2} and so we are only left to prove the propagation of measure conclusion.
	
	Let us take $k = \bsmu^+ - \varepsilon\bsom$ for some $\varepsilon\in (0,\tfrac12\tau)$ to be eventually chosen. From the choice of $\varepsilon$ and $\bsmu^+ \geq \tau \bsom$, we see that $k \geq \tfrac12 \tau \bsom$. We consider the energy estimate from \cref{Prop:energy} for $(u-k)_+$ over the cylinder $(\htt,\mft]$. Furthermore, for $\tilde\sigma \in (0,\tfrac18]$ to be chosen later, we take a cut-off function $\zeta = \zeta(x) \geq 0$ such that it is supported in $B_{c_o\varrho(1-\frac12\tilde\sigma)}$ with $\zeta \equiv 1$ on $B_{c_o\varrho(1-\tilde\sigma)}$ and $|\nabla \zeta| \apprle\frac{1}{\tilde\sigma\varrho}$. Applying \cref{Prop:energy}, we get
	\begin{multline}\label{deg5.23}
	\int_{B_{c_o\varrho} \times \{\mft\}}\int_{k}^{u} a^{p-2}(a-k)_+ \,da\, \zeta^p \,dx \leq \int_{B_{c\varrho}\times \{\htt\}}\int_{k}^{u} a^{p-2}(a-k)_+ \,da\, \zeta^p \,dx   \\
	\begin{array}{l}
+\boldsymbol{C} \int_{\htt}^{\mft}\iint_{B_{c_o\varrho}\times B_{c_o\varrho}} \frac{\max\{(u-k)_{+}(x,t),(u-k)_{+}(y,t)\}^{p}|\zeta(x)-\zeta(y)|^p}{|x-y|^{n+sp}}\,dx\,dy\,dt\\
+ \boldsymbol{C}\lbr  \underset{\stackrel{t \in (\htt,\mft)}{x\in \spt \zeta}}{\esssup}\,\int_{\RR^n \setminus B_{c_o\varrho}}\frac{(u-k)_{+}^{p-1}(y,t)}{|x-y|^{n+sp}}\,dy\rbr\iint_{(\htt,\mft)\times B_{c_o\varrho}} (u-k)_{+}(x,t)\zeta(x)\,dx\,dt.
\end{array}
	\end{multline}

The first term on the right hand side of \cref{deg5.23} is estimated using \cref{assump2} and \cref{remark5.10} to get
\begin{equation*}
	\int_{B_{c\varrho}\times \{\htt\}}\int_{k}^{u} a^{p-2}(a-k)_+ \,da \zeta^p \,dx  \leq (1-\tfrac12 \nu) |B_{c_o\varrho}| \int_{k}^{\bsmu^+} a^{p-2}(a-k)_+ \,da
\end{equation*}

The second term on the right hand side of \cref{deg5.23} is estimated using $(u-k)_+ \leq \varepsilon\bsom$ to get 
\begin{equation*}
	\int_{\htt}^{\mft}\iint_{B_{c_o\varrho}\times B_{c_o\varrho}} \frac{\max\{(u-k)_{\pm}(x,t),(u-k)_{\pm}(y,t)\}^{p}|\zeta(x)-\zeta(y)|^p}{|x-y|^{n+sp}}\,dx\,dy\,dt \apprle \frac{|\mft -\htt| |B_{c_o\varrho}| \varepsilon^p \bsom^p}{\tilde\sigma^p (c_o\varrho)^{sp}}.
\end{equation*}

The third term on the right hand side of \cref{deg5.23} is estimated using the $\tail$ estimate from \cref{deggclaim5.2} to get
\begin{equation*}
	\lbr  \underset{\stackrel{t \in (\htt,\mft)}{x\in \spt \zeta}}{\esssup}\,\int_{\RR^n \setminus B_{c_o\varrho}}\frac{(u-k)_{+}^{p-1}(y,t)}{|x-y|^{n+sp}}\,dy\rbr\iint_{(\htt,\mft)\times B_{c_o\varrho}} (u-k)_{+}(x,t)\zeta(x)\,dx\,dt \apprle \frac{|\mft -\htt| |B_{c_o\varrho}| \varepsilon^p \bsom^p}{ \tilde\sigma^{sp}(c_o\varrho)^{sp}}
\end{equation*}

We now estimate the term appearing on the left hand side of \cref{deg5.23} as follows:
\begin{equation*}
	\int_{B_{c_o\varrho} \times \{\mft\}}\int_{k}^{u} a^{p-2}(a-k)_+ \,da\, \zeta^p \,dx \geq |\{u(\cdot,\mft) > k_{\tilde\varepsilon}\} \cap B_{c_o\varrho(1-\tilde\sigma)}| \int_k^{k_{\tilde\varepsilon}} a^{p-2}(a-k)_+\,da,
\end{equation*}
where $k_{\tilde\varepsilon} = \bsmu^+ - \varepsilon\tilde\varepsilon\bsom$ for some $\tilde\varepsilon\in(0,\tfrac12)$. Recalling the we are in the case $\tau \bsom \leq \bsmu^+ \leq 2\bsom$, we further get
\begin{equation*}
	\int_k^{k_{\tilde\varepsilon}} a^{p-2}(a-k)_+\,da \geq \lbr \tfrac{\tau}{2}\bsom\rbr^{p-2} \int_k^{k_{\tilde\varepsilon}} (a-k)_+ \,da =  \lbr\tfrac{\tau}{2}\bsom\rbr^{p-2} \frac{(\varepsilon\bsom)^2 (1-\tilde\varepsilon)^2}2,
\end{equation*}
where we have used the restriction $\varepsilon \leq \tfrac12 \tau$. As in the proof of \cref{Lm:3:1}, we have
\begin{multline}\label{deg5.29}
	|\{u(\cdot,\mft) >k_{\tilde\varepsilon}\}\cap B_{c_o\varrho(1-\tilde\sigma)}|
	\leq 
	\frac{\int_k^{\bsmu^+} |a|^{p-2}(a-k)_-\,da }{\int_k^{k_{\tilde\varepsilon}} |a|^{p-2}(a-k)_-\,da }
	(1-\tfrac12\nu)|B_{c_o\varrho}|
	+
	\bsc \frac{|\mft-\htt| |B_{c_o\varrho}| \varepsilon^{p-2}}{\tilde\sigma^p(c_o\varrho)^{sp}(1-\tilde\ve)^2\tau^{p-2}} \\
	+
	\bsc \frac{|\mft-\htt| |B_{c_o\varrho}|\varepsilon^{p-2} }{\tilde\sigma^{sp}(c_o\varrho)^{sp}(1-\tilde\ve)^2\tau^{p-2}}+n\tilde\sigma|B_{c_o\varrho}|.
\end{multline}
Furthermore, proceeding as in \cite[Lemma 6.2]{liaoHolderRegularitySigned2021}, we use $\tau \bsom \leq \bsmu^+ \leq 2\bsom$ and $k \geq \tfrac12 \tau \bsom$ to get
\begin{equation}\label{deg5.30}
	\frac{\int_k^{\bsmu^+} |a|^{p-2}(a-k)_-\,da }{\int_k^{k_{\tilde\varepsilon}} |a|^{p-2}(a-k)_-\,da }\leq 1 + \bsc_p \tilde\varepsilon\varepsilon.
\end{equation}
Combining \cref{deg5.29} and \cref{deg5.30} and noting that $|\mft-\htt| \leq (c_o\varrho)^{sp}$ from \cref{assump2}, we get
\begin{equation*}
	|\{u(\cdot,\mft) >k_{\tilde\varepsilon}\}\cap B_{c_o\varrho(1-\sigma)}|
	\leq 
	\lbr (1+\bsc_p \tilde\varepsilon)
	(1-\tfrac12\nu)
	+
	\bsc \frac{  \varepsilon^{p-2}}{\sigma^p(1-\tilde\ve)^2\tau^{p-2}} 
	+
	\bsc \frac{ \varepsilon^{p-2} }{\sigma^{sp}(1-\tilde\ve)^2\tau^{p-2}}+n\sigma\rbr|B_{c_o\varrho}|.
\end{equation*}
We first choose $\tilde\varepsilon$ small such that $(1+\bsc_p \tilde\varepsilon)
(1-\tfrac12\nu) \leq 1-\tfrac38 \nu$ and then choose $\sigma = \tfrac{\nu}{16n}$ followed by $\varepsilon$ small such that $\bsc \frac{  \varepsilon^{p-2}}{\sigma^p(1-\tilde\ve)^2\tau^{p-2}} 
+
\bsc \frac{ \varepsilon^{p-2} }{\sigma^{sp}(1-\tilde\ve)^2\tau^{p-2}} \leq \tfrac{\nu}{16}$ to complete the proof of the lemma. 
\end{proof}


\begin{definition}\label{degdef5.10}
	Let us fix some number $j_* \geq 2$ (which is eventually chosen at the end of this subsection to depend on $n,p,s,\varepsilon,\tau$ where $\varepsilon$ is from \cref{Lm:4:2}) and then choose $A = 2^{j_*(p-2)}+1$. With this choice, we  denote $\mcq_{c_o\varrho}^{j_*} := B_{c_o\varrho}\times (-2^{j_*(p-2)}(c_o\varrho)^{sp},0)$. 
\end{definition}

\begin{corollary}\label{degrmk5.10}
	From \cref{assump2}, \cref{Lm:4:2} and the choice of $A$ according to \cref{degdef5.10}, we see that the conclusion of \cref{Lm:4:2} holds for all time levels $t \in (-2^{j_*(p-2)}(c_o\varrho)^{sp},0)$. In particular, we have
	\[	\begin{cases*}
		\left|\left\{ u(\cdot, t)\le\bsmu^+-\varepsilon\boldsymbol \om\right\}\cap B_{c_o\varrho}\right|
		\ge
		\tfrac14\nu |B_{c_o\varrho}|
		\txt{for all} t\in(-2^{j_*(p-2)}(c_o\varrho)^{sp},0],\\
		c_o^{\frac{sp}{p-1}} \tail(( u-\bsmu^{+})_{+}; {\mbcq}) \leq   \varepsilon^{p-1}\bsom.
	\end{cases*}
	\]
\end{corollary}
We now prove a measure shrinking lemma:

\begin{lemma}\label{deglemma5.12}
	Suppose $\tau\bsom\le\bsmu^+\le2\bsom$ and \cref{degrmk5.10} holds, then for any positive integer $j_*$, we have 
	\[
	|\{ u\ge\bsmu^+-\tfrac{\epsilon\bsom}{2^{j_*+1}}\}\cap {\mcq}_{c_o\varrho}^{j_*}|\leq 
	\left(\tfrac{\bsc}{2^{j_*}-1}\right)^{p-1}|{\mcq}_{c_o\varrho}^{j_*}|\quad\mbox{where} \quad {\mcq}_{c_o\varrho}^{j_*} = B_{c_o\varrho}\times (-2^{j_*(p-2)}(c_o\varrho)^{sp},0],
	\]
	where $\boldsymbol{C} = \boldsymbol{C}(\datanb{,\nu,\tau,\ve})$ and $\ve = \ve(\datanb{,\nu,\tau})$ is from \cref{Lm:4:2}  provided the following is satisfied:
	\begin{equation}\label{deglm5.2tail}
	c_o^{\frac{sp}{p-1}} \tail(( u-\bsmu^{+})_{+}; \mbcq) \leq   \frac{\varepsilon\bsom}{2^{j_*+1}}.
	\end{equation}
\end{lemma}

\begin{proof}
	From \cref{deggclaim5.2}, we see that there exists $c_o = c_o(n,p,s,\varepsilon,j_*)$ such that \cref{deglm5.2tail} is always satisfied. Thus, we are now left to obtain the measure shrinking conclusion of the lemma. 
	
	Consider a time independent cut-off function $\zeta(x)$ satisfying $\zeta \equiv 1$ on $B_{c_o\varrho}$ and supported in $B_{\tfrac32c_o\varrho}$ and $|\nabla \zeta| \apprle \tfrac{1}{c_o\varrho}$. Taking $k_j = \bsmu^+ - \tfrac{\varepsilon\bsom}{2^{j}}$ for some $j \in \{1,2,\ldots,j_*-1\}$, we now apply \cref{Prop:energy} over the cylinder  $\mcq_{2c_o\varrho}^{j_*}:=B_{2c_o\varrho}\times (-2^{j_*(p-2)}(c_o\varrho)^{sp},0]$ (note the abuse of notation where $2$ is only taken in $B_{2c_o\varrho}$) to get
	\begin{multline}\label{deg5.33}
		\iint_{\mcq_{c_o\varrho}^{j_*}} (u-k_j)_{+}(x,t)\int_{B_{c_o\varrho}}\frac{(u-k_j)_{-}^{p-1}(y,t)}{|x-y|^{n+sp}}\,dy\,dx\,dt \\ 
		\begin{array}{rcl}
			&\leq&
			C \int_{-2^{j_*(p-2)}(c_o\varrho)^{sp}}^{0}\iint_{B_{2c_o\varrho}\times B_{2c_o\varrho}} \frac{\max\{(u-k_j)_{+}(x,t),(u-k_j)_{+}(y,t)\}^{p}|\zeta(x)-\zeta(y)|^p}{|x-y|^{n+sp}}\,dx\,dy\,dt
			\\
			&&
			+\int_{B_{2c_o\varrho}\times \{-2^{j_*(p-2)}(c_o\varrho)^{sp}\}} \zeta^p \mathfrak g_+ (u,k_j)\,dx 
			\\
			&&+ C\lbr  \underset{\stackrel{t \in (-2^{j_*(p-2)}(c_o\varrho)^{sp},0]}{x\in \spt \zeta}}{\esssup}\,\int_{\RR^n \setminus B_{2c_o\varrho}}\frac{(u-k_j)_{+}^{p-1}(y,t)}{|x-y|^{n+sp}}\,dy\rbr\iint_{ \mcq_{2c_o\varrho}^{j_*}} (u-k_j)_{+}(x,t)\zeta(x)\,dx\,dt\\
			&=:&  \I + \II +  \III.
		\end{array}
	\end{multline}
	We estimate each of the terms appearing on the right hand side of \cref{deg5.33} as follows:
	\begin{description}
		\item[\color{airforceblue}{Estimate for $\I$:}]	We see that $(u-k_j)_+ \leq \tfrac{\varepsilon\bsom}{2^j}$ and thus using $|\nabla \zeta| \apprle \tfrac{1}{c_o\varrho}$, we get
		\[
		\I \apprle  \frac{1}{(c_o\varrho)^{sp}}\lbr\frac{\varepsilon\bsom}{2^j}\rbr^p |{\mcq}_{c_o\varrho}^{j_*}|.
		\]
			\item[\color{airforceblue}{Estimate for $\II$:}] Making use of \cref{lem:g} along with the bound $\tau \bsom \leq \bsmu^+ \leq 2\bsom$, we get
			\[
			\begin{array}{rcl}
				\II& \apprle & \lbr\frac{\varepsilon\bsom}{2^j}\rbr^2 \int_{B_{2c_o\varrho}\times \{-2^{j_*(p-2)}(c_o\varrho)^{sp}\}} (|u| + |k_j|)^{p-2}\,dx\\
				& \overred{deg5.33a}{a}{\apprle}& \lbr\frac{\varepsilon\bsom}{2^j}\rbr^2 \bsom^{2-p} |B_{2c_o\varrho}|
				 =  \lbr\frac{\varepsilon\bsom}{2^j}\rbr^2 \bsom^{2-p} \frac{|{\mcq}_{c_o\varrho}^{j_*}|}{|2^{j_*(p-2)}(c_o\varrho)^{sp}|}\\
				 & \overred{deg5.33b}{b}{\apprle}& \frac{1}{(c_o\varrho)^{sp}}\lbr\frac{\varepsilon\bsom}{2^j}\rbr^p |{\mcq}_{c_o\varrho}^{j_*}|,
				\end{array}	
			\]
			where to obtain \redref{deg5.33a}{a}, we used $(u-k_j)_+ \leq \tfrac{\varepsilon\bsom}{2^j}$ and to obtain \redref{deg5.33b}{b}, we note that $j \leq j_* -1$ and $\varepsilon$ is fixed depending only on $\data{,\nu,\tau}$ according to \cref{Lm:4:2}.
		\item[\color{airforceblue}{Estimate for $\III$:}] In order to estimate this term, we make use of \cref{degg5.5} along with $(u-k_j)_+ \leq \tfrac{\varepsilon\bsom}{2^j}$ and the fact that $\spt(\zeta) \subset B_{\frac32c_o\varrho}$ to get
		\[
		\begin{array}{rcl}
			\III & \apprle & \frac{1}{(c_o\varrho)^{sp}}\tailp((u-k_j)_+,0,2c_o\varrho,(-2^{j_*(p-2)}(c_o\varrho)^{sp},0])\lbr\frac{\varepsilon\bsom}{2^j}\rbr |{\mcq}_{c_o\varrho}^{j_*}|\\
			& \apprle& \frac{1}{(c_o\varrho)^{sp}}|{\mcq}_{c_o\varrho}^{j_*}|\lbr[[]\lbr\frac{\varepsilon\bsom}{2^j}\rbr^p  + \tailp((u-\bsmu^+)_+,0,2c_o\varrho,(-2^{j_*(p-2)}(c_o\varrho)^{sp},0])\lbr\frac{\varepsilon\bsom}{2^j}\rbr\rbr[]]   \\
			& \overlabel{deggclaim5.2}{\apprle}& \lbr\frac{\varepsilon\bsom}{2^j}\rbr^p \frac{1}{(c_o\varrho)^{sp}}|{\mcq}_{c_o\varrho}^{j_*}| + \frac{1}{(c_o\varrho)^{sp}}c_o^{{sp}}\tailp((u-\bsmu^+)_+,\mbcq)\lbr\frac{\varepsilon\bsom}{2^j}\rbr |{\mcq}_{c_o\varrho}^{j_*}|\\
			& \overlabel{deglm5.2tail}{\apprle}&\lbr\frac{\varepsilon\bsom}{2^j}\rbr^p \frac{1}{(c_o\varrho)^{sp}}|{\mcq}_{c_o\varrho}^{j_*}|,
		\end{array}
		\]
		where to obtain the last inequality, we used the fact that $\tfrac1{2^{j_*}}\leq \tfrac1{2^{j}}$ for any $j \leq j_*-1$.
	\end{description}
Thus combining the estimate for $\I$, $\II$ and $\III$ into \cref{deg5.33}, we get
\begin{equation}\label{deg5.34}
\iint_{\mcq_{c_o\varrho}^{j_*}} (u-k_j)_{+}(x,t)\int_{B_{c_o\varrho}}\frac{(u-k_j)_{-}^{p-1}(y,t)}{|x-y|^{n+sp}}\,dy\,dx\,dt\apprle\lbr\frac{\varepsilon\bsom}{2^j}\rbr^p \frac{1}{(c_o\varrho)^{sp}}|{\mcq}_{c_o\varrho}^{j_*}|.
\end{equation}

We see that the hypothesis of \cref{lem:shrinking} is satisfied with $A = 1$, $l= k_j $ applied to \cref{deg5.34} and \cref{degrmk5.10}  over $\mcq_{c_o\varrho}^{j_*}$  and thus, \cref{lem:shrinking} gives the required conclusion.
\end{proof}
We now prove a de Giorgi type iteration lemma.
\begin{lemma}\label{Lm:5:3}
	Suppose that  \cref{assump2} holds and and $\tau\bsom\le\bsmu^+\le2\bsom$ is satisfied.
	Then, there exists a constant $\nu_1 = \nu_1(\datanb{,\varepsilon,\tau})\in(0,1)$ such that if for some $j_*\geq 3$, the following estimate is satisfied:
	\begin{equation*}
		\left|\left\{\bsmu^{+}-u\leq \tfrac{\varepsilon\bsom}{2^{j_*}}\right\} 
		\cap {\mcq}_{c_o\varrho}^{j_*}
		\right|
		\leq \nu_1|{\mcq}_{c_o\varrho}^{j_*}|,
	\end{equation*}
	then the following conclusion holds
	\[
	\bsmu^{+}-u\geq \frac{\varepsilon\bsom}{2^{j_*+1}}\quad\mbox{a.e., in } \quad {\mcq}_{\frac12c_o\varrho}^{j_*},
	\]
	provided we have
	\begin{equation}\label{Lm:5:3:tail}
	c_o^{\frac{sp}{p-1}} \tail(( u-\bsmu^{+})_{+}; \mbcq)  \leq \frac{\epsilon\bsom}{2^{j_*+1}}.
	\end{equation}
	Recall ${\mcq}_{c_o\varrho}^{j_*} = B_{c_o\varrho}\times (-2^{j_*(p-2)}(c_o\varrho)^{sp},0)$ and ${\mcq}_{\frac12c_o\varrho}^{j_*} = B_{\frac12c_o\varrho}\times (-2^{j_*(p-2)}(\tfrac{c_o\varrho}{2})^{sp},0)$.
\end{lemma}
\begin{proof}
As in \cref{deggclaim5.2}, we can choose $c_o = c_o(\datanb{,\varepsilon},j_*) \in (0,1)$ sufficiently small such that the tail estimate \cref{Lm:5:3:tail} is satisfied. 	Let us denote $M = \frac{\varepsilon\bsom}{2^{j_*}}$ and $\tht := 2^{j_*(p-2)}$, then as in the proof of \cref{Lm:6:1},  we define 
	\begin{equation*}
		{\def\arraystretch{\st}\begin{array}{llll}
				k_i:= \bsmu^+ - \frac{M}2 - \frac{M}{2^{i+1}},& \tilde{k}_i:=\tfrac{k_i+k_{i+1}}2,&
				\varrho_i:=\tfrac{c_o\varrho}2+\tfrac{c_o\varrho}{2^{i+1}},
				&\tilde{\varrho}_i:=\tfrac{\varrho_i+\varrho_{i+1}}2,\\
				B_i:=B_{\varrho_i},& \widetilde{B}_i:=B_{\tilde{\varrho}_i},&
				\mcq_i:= B_i \times ( - \tht(\varrho_i)^{sp}),0],&
				\widetilde{\mcq}_i:=\widetilde{B}_i \times ( - \tht(\tilde\varrho_i)^{sp}),0].
		\end{array}}
	\end{equation*}
Furthermore, we also define 
\begin{equation*}
	\hat{\varrho}_i:=\tfrac{3\varrho_i+\varrho_{i+1}}4, \quad 
	\bar{\varrho}_i:=\tfrac{\varrho_i+3\varrho_{i+1}}4,\quad 
	\hat{\mcq}_i:=B_{\hat{\varrho}_i} \times ( - \tht(\hat\varrho_i)^{sp}),0], \quad
	\bar{\mcq}_i:=B_{\bar{\varrho}_i} \times ( - \tht(\bar\varrho_i)^{sp}),0].
\end{equation*}
	We now consider a cut-off functions $\bar{\zeta_i}$ and $\zeta_i$ independent of time such that
\begin{equation}\label{degcutoff_size_i}
	\begin{array}{cccc}
		\bar{\zeta_i} \equiv 1 \text{ on } \mcq_{i+1},& \quad \bar{\zeta_i} =0 \,\,\text{on} \,\,\partial_p \bar{\mcq}_i , &\quad  |\nabla\bar{\zeta_i}| \apprle \frac{1}{\bar{\varrho_i} - \varrho_{i+1}} \approx \frac{2^i}{c_o\varrho}, &\quad |\pa_t\bar{\zeta_i} | \apprle \frac{2^{spi}}{\tht(c_o\varrho)^{sp}} \\
		{\zeta_i} \equiv 1 \text{ on } \mcq_{\tilde{\varrho_j}}, &\quad {\zeta_i}=0 \,\,\text{on} \,\, \pa_p\hat\mcq_i, &\quad  |\nabla{\zeta_i}| \apprle \frac{1}{\hat{\varrho_n} - \tilde{\varrho_{j}}}\approx \frac{2^{j}}{c_o\varrho},&\quad |\pa_t{\zeta_i} | \apprle \frac{2^{spi}}{\tht(c_o\varrho)^{sp}}.
	\end{array}
\end{equation}
We now make the following observations:
\begin{itemize}
	\item Since $\tau \bsom \leq \bsmu^+ \leq 2\bsom$ and $\varepsilon \in (0,\tfrac12\tau)$ as in \cref{Lm:4:2}, we see that $$|u|+|k_i| \geq |k_i| \geq \tau \bsom - \frac{\varepsilon\bsom}{2^{j_*+1}} - \frac{\varepsilon\bsom}{2^{j_*+i+1}} \overset{\text{\cref{Lm:4:2}}}{\geq} \tau \bsom - \frac{\tau\bsom}{2^{j_*+1}} - \frac{\tau\bsom}{2^{j_*+i+2}} \geq \tfrac14 \tau \bsom, $$
	provided $j_* \geq 4$. In particular, we have $\mathfrak g_+(u,\tilde{k}_i) \geq \bsom^{p-2} \lbr\tfrac{\tau}{4}\rbr^{p-2} (u-\tilde{k}_i)_+^2$.
	\item Let us denote $A_i = \{u > k_i\} \cap \mcq_i$.
	\item We trivially have $(u-k_i)_+ \leq M$ and $\mathfrak g(u,k_i)_+ \apprle \bsom^{p-2}M^2$ where we used $\bsmu^+ \leq 2\bsom$ and $M \leq \bsom$.
\end{itemize}
Now, we  apply the energy estimates from \cref{Prop:energy} over  $\mcq_i $ to $ (u-\tilde{k}_i)_{+}$ and $\zeta_i$ to get
\begin{multline}\label{degEq:3.6en_i}
	(\tau\bsom)^{p-2}	\underset{-\theta\varrho_i^{sp} < t < 0}{\esssup}\int_{B_i}(u-\tilde{k}_i)_{+}^2\zeta_i^p(x,t)\,dx  
	+ {\iiint_{\mcq_i}}\frac{|(u-\tilde{k}_i)_{+}(x,t)\zeta_i(x,t)-(u-\tilde{k}_i)_{+}\zeta_i(y,t)|^p}{|x-y|^{n+sp}}\,dx \,dy\,dt
	\\ 
	\begin{array}{rcl} &\leq_{\data{}} & \frac{2^{isp}}{(c_o\varrho)^{sp}}M^p |A_i| +
	\frac{2^{isp}}{\tht(c_o\varrho)^{sp}}M^2\bsom^{p-2} |A_i|\\
	&&+ \lbr \underset{\stackrel{-\theta\varrho_i^{sp} < t < 0;}{ x\in \spt \zeta_i}}{\esssup}\int_{\RR^n \setminus B_i}\frac{(u-\tilde{k}_i)_{+}^{p-1}(y,t)}{|x-y|^{n+sp}}\,dy\rbr   M |A_i|.
				\end{array}
\end{multline}
We estimate the tail term as in \cref{sec:tail} to get
\begin{equation}\label{degg5.38_i}
\underset{\stackrel{-\theta\varrho_i^{sp} < t < 0;}{ x\in \spt \zeta_i}}{\esssup}\int_{\RR^n \setminus B_i}\frac{(u-\tilde{k}_i)_{+}^{p-1}(y,t)}{|x-y|^{n+sp}}\,dy \apprle \frac{2^{isp}}{(c_o\varrho)^{sp}}\tail((u-\bsmu^+)_+;0;\varrho_i,(-\tht\varrho_i^{sp},0]) + \frac{2^{isp}}{(c_o\varrho)^{sp}}M^{p-1}.
\end{equation}
From the choice of $c_o = c_o(\datanb{,\varepsilon,j_*})$, we see that 
\begin{equation}\label{degg5.39_i}
c_o^{\frac{sp}{p-1}} \tail(( u-\bsmu^{+})_{+}; \mbcq)  \leq \frac{\epsilon\bsom}{2^{j_*+1}}
\Longrightarrow\tail((u-\bsmu^+)_+;0;\varrho_i,(-\tht\varrho_i^{sp},0]) \apprle M^{p-1}.
\end{equation}
Thus combining \cref{degg5.38_i} and \cref{degg5.39_i} into \cref{degEq:3.6en_i}, we get
\begin{multline}\label{degEq:5.40en_i}
	(\tau\bsom)^{p-2}	\underset{-\theta\varrho_i^{sp} < t < 0}{\esssup}\int_{B_i}(u-\tilde{k}_i)_{+}^2\zeta_i^p(x,t)\,dx  
	+ {\iiint_{\mcq_i}}\frac{|(u-\tilde{k}_i)_{+}(x,t)\zeta_i(x,t)-(u-\tilde{k}_i)_{+}\zeta_i(y,t)|^p}{|x-y|^{n+sp}}\,dx \,dy\,dt
	\\ 
	\begin{array}{rcl}
	&\apprle_{\data{,\tau}} & \frac{2^{isp}}{(c_o\varrho)^{sp}}M^p |A_i| +
		\frac{2^{isp}}{\tht(c_o\varrho)^{sp}}M^2\bsom^{p-2} |A_i| + \frac{2^{isp}}{(c_o\varrho)^{sp}}M^{p}|A_i|\\
		&{\overred{5.43a}{a}{\apprle}}_{\data{,\tau}}& \frac{2^{isp}}{(c_o\varrho)^{sp}}M^p |A_i| \lbr 1 + \varepsilon^{2-p}\rbr,
		\end{array}
\end{multline}
where to obtain \redref{5.43a}{a}, we substituted the expression for $\tht$ and $M$ into the second term.
From Young's inequality, we have
\begin{multline}\label{degeq3.9_i}
	|(u-\tilde{k}_{i})_+ \bar{\zeta_i}(x,t) - (u-\tilde{k}_{i})_+ \bar{\zeta_i}(y,t)|^p \leq c |(u-\tilde{k}_{i})_+ (x,t) - (u-\tilde{k}_{i})_+ (y,t)|^p\bar{\zeta_i}^p(x,t) \\
	+ c |(u-\tilde{k}_{i})_+(y,t)|^p |\bar{\zeta_i}(x,t) - \bar{\zeta_i}(y,t)|^p,
\end{multline}
from which we obtain the following sequence of estimates:
\begin{equation*}
	\begin{array}{rcl}
		\frac{M}{2^{i+3}}
		|A_{i+1}|
		&\overred{3.7ai}{a}{\leq} &
		\iint_{\bar{\mcq}_{i+1}}\lbr u-\tilde{k}_i\rbr_+\bar{\zeta_i}\,dx\,dt \\
		&\overred{3.7bi}{b}{\leq} &
		\left[\iint_{\widetilde{\mcq}_i}\big[\lbr u-\tilde{k}_i\rbr_+\bar{\zeta_i}\right]^{p\frac{n+2s}{n}}
		\,dx\,dt\bigg]^{\frac{n}{p(n+2s)}}|A_i|^{1-\frac{n}{p(n+2s)}}\\
		&\overred{3.7ci}{c}{\leq} &\bsc 
		\left(\iiint_{\widetilde{\mcq}_i}\frac{|(u-\tilde{k}_i)_{+}(x,t)\bar{\zeta_i}(x,t)-(u-\tilde{k}_i)_{+}\bar{\zeta_i}(y,t)|^p}{|x-y|^{n+sp}}\,dx \,dy\,dt\right)^{\frac{n}{p(n+2s)}} 
		\\&&\qquad\times \left(\underset{-\theta\varrho_i^{sp} < t < 0}{\esssup}\int_{\widetilde{B}_i}[(u-\tilde{k}_i)_{+}\bar\zeta_i(x,t)]^2\,dx\right)^{\frac{s}{n+2s}}|A_i|^{1-\frac{n}{p(n+2s)}}\\
		&\overred{3.7di}{d}{\leq}&
		\bsc  
		\lbr \frac{2^{isp}}{(c_o\varrho)^{sp}}M^p\lbr 1 + \varepsilon^{2-p}\rbr\rbr^{\frac{n}{p(n+2s)}}
		\lbr\frac{2^{isp}}{(c_o\varrho)^{sp}}M^p\lbr 1 + \varepsilon^{2-p}\rbr \tau^{2-p}\bsom^{2-p}\rbr^{\frac{s}{n+2s}}
		|A_i|^{1+\frac{s}{n+2s}} \\
		&\overred{3.7ei}{e}{\leq} &
		\bsc 
		\frac{b_0^i}{(c_o\varrho)^\frac{s(n+sp)}{n+2s}} M \lbr \frac{M^{p-2}}{\bsom^{p-2}}\rbr^{\frac{s}{n+2s}} \lbr 1 + \varepsilon^{2-p}\rbr^{\frac{n+sp}{p(n+2s)}}
		 \tau^{\frac{(2-p)s}{n+2s}}|A_i|^{1+\frac{s}{n+2s}},
	\end{array}
\end{equation*}
where to obtain \redref{3.7ai}{a}, we made use of the observations and enlarged the domain of integration with $\tilde{\zeta}_i$ as defined in \cref{degcutoff_size_i}; to obtain \redref{3.7bi}{b}, we applied H\"older's inequality; to obtain \redref{3.7ci}{c}, we applied \cref{fracpoin}; to obtain \redref{3.7di}{d}, we made use of \cref{degEq:5.40en_i} and \cref{degeq3.9_i}  with $\bsc = \bsc_{\data{}}$ and finally we collected all the terms to obtain \redref{3.7ei}{e}, where $b_0\geq 1$ is a universal constant.

Setting $Y_i = |A_i|/|\mcq_i|$, we get
\[
Y_{i+1} \leq \bsc b_o^i \lbr 1 + \varepsilon^{2-p}\rbr^{\frac{n+sp}{p(n+2s)}}\varepsilon^{\frac{s(p-2)}{n+2s}}\tau^{\frac{(2-p)s}{n+2s}}|Y_i|^{1+\frac{s}{n+2s}}.
\]
We can now apply \cref{geo_con} to obtain the required conclusion.
%
%
%
\end{proof}
Now we can conclude the required reduction of oscillation. First we collect the definition of constants to keep track of their dependence: Note that all constants depend additionally on $\La$ defined in \cref{boundsonKernel}, but we suppress writing this dependence explicitly.

\begin{itemize}
	\item We first choose $j_* = j_*(n,p,s,\varepsilon,\tau)$ such that $\lbr \frac{\bsc}{2^{j_*}-1}\rbr^{p-1}\leq \nu_1$ where $\bsc(\datanb{,\nu,\tau,\ve})$ is the constant in \cref{deglemma5.12}.
	\item We can further choose $j_*$ large enough such that $2^{j_*(p-2)} \geq 1$. 
	\item Once $j_*$ is chosen, this then fixes $A = 2^{j_*(p-2)}+1$ according to \cref{degdef5.10}. 
	\item With this choice of $j_*$, let the rest of the constants be as given in the above table. Then we get 
	\begin{equation}\label{sub5.2.6.a}
		\bsmu^+ - u \geq \frac{\varepsilon\bsom}{2^{j_*+1}} \txt{a.e. on}  \mcq_{\frac12c_o\varrho}^{j_*}
	\end{equation}
\end{itemize}

\subsection{Constants in degenerate case near zero}\label{subsection5.2.6}
\noindent\begin{minipage}{0.49\textwidth}
	\begin{center}
		\begin{tabular}{|c|c|c|}
			\hline
			Symbol & location & dependence\\
			\hline\hline
			$c_o$ & \cref{degclaim5.2,degclaim5.3} & $n,p,s,A$\\
			 & \cref{deggclaim5.3} & $n,p,s,\bar\eta_2$\\
			& \cref{Lm:6:1} & $n,p,s,\eta_1$ \\
			& \cref{deggclaim5.2} & $n,p,s$\\
			& \cref{Lm:4:2} & $n,p,s,\varepsilon$\\
			& \cref{deglemma5.12} & $n,p,s,\varepsilon,j_*$\\
			& \cref{Lm:5:3} & $n,p,s,\varepsilon,j_*$\\
			\hline
			$\eta$ & \cref{degclaim5.3} & $n,p,s,\alpha$\\\hline
$\de$ & \cref{degclaim5.3} & $n,p,s$\\ \hline
		\end{tabular}
	\end{center}
\end{minipage}
\begin{minipage}{0.49\textwidth}
	\begin{center}
		\begin{tabular}{|c|c|c|}
			\hline
			Symbol & location & dependence\\
			\hline\hline
			$ \bar{\eta}$ & \cref{degclaim5.3} & $n,p,s,\alpha,A$ \\\hline
			$\eta_1 $ & \cref{Lm:6:1} & $n,p,s,A$ and $\eta_1 \leq \frac{\tau}4$\\\hline
			$\eta_2$ & \cref{deggclaim5.3} & $n,p,s$\\\hline
			$\bar\eta_2$ & \cref{deggclaim5.3} & $n,p,s,A$\\\hline
			$\tau$ &  \cref{degclaim5.3} & $\tau \leq 2\eta$\\
			& \cref{deggclaim5.3} & $\tau \leq 2\eta_2$\\\hline
			$\varepsilon$ & \cref{Lm:4:2} & $n,p,s,\nu,\tau$\\\hline
			$\nu$ & \cref{assump1} &$n,p,s$\\\hline
			$\nu_1$ & \cref{Lm:5:3} &$n,p,s,\varepsilon,\tau$\\\hline
		\end{tabular}
	\end{center}
\end{minipage}
 
\subsubsection{Reduction of oscillation near zero -  collected estimates}\label{nearzerodegen}
 From choosing $c_o$ sufficiently small, we can ensure  all of the corresponding $\tail$ alternatives holds. Thus, we can collect the conclusions as follows:
 \begin{description}
 	\item[Case $-\tau \bsom \leq \bsmu^- \leq \tau \bsom$:] In this case, from \cref{degclaim5.3}, we get
 	\[
 	u \geq \bsmu^- + \bar\eta \bsom \quad\mbox{a.e., in } \quad B_{\frac{c_o\varrho}2} \times \left(t_o -\tfrac14 (c_o\varrho)^{sp},t_o\right].
 	\]
 	\item[Case $\bsmu^-\leq -\tau \bsom$:] In this case, from \cref{Lm:6:1}, we get
 	\[
 	u\ge\bsmu^-+\eta_1\bsom\quad\mbox{a.e., in}\quad  
 	B_{\frac14 c_o\varrho}\times\left(t_o - (\tfrac{c_o\varrho}{2})^{sp},t_o\right].
 	\]
 	\item[Case $-\tau \bsom \leq \bsmu^+ \leq \tau \bsom$:] In this case, from \cref{deggclaim5.3}, we get
 	\begin{equation*}
 		u \leq \bsmu^+ - \bar\eta_2 \bsom \txt{a.e., in} B_{\frac12c_o\varrho} \times (t_o- \tfrac{\nu}{4}(\tfrac{c_o\varrho}{2})^{sp}, t_o].
 	\end{equation*}
 	\item[Case $\tau \bsom \leq \bsmu^+  \leq 2\bsom$:] In this case, from \cref{sub5.2.6.a} and using $2^{j_*(p-2)}\geq 1$, we have
 	\[
 	\bsmu^{+}-u\geq \frac{\varepsilon\bsom}{2^{j_*+1}}\txt{a.e., in }  B_{\frac12c_o\varrho}\times (t_o-(\tfrac{c_o\varrho}{2})^{sp},t_o].
 	\]
 \end{description}

\section{Reduction of oscillation near zero - Singular case}

In this case, we assume $p \leq 2$ and prove reduction of oscillation.

\subsection{Setting up the geometry}
Let us fix some reference cylinder $\mbcq \subset \omt$. Fix $(x_o,t_o)\in \omt$ and let ${\mcq}_o={\mcq}_{\varrho} = B_{\varrho}(x_0) \times (t_0-\varrho^{sp},t_0] \subset \mbcq $  be any cylinder centred at $(x_o,t_o)$. 
Assume, without loss of generality, that $(x_o,t_o) = (0,0)$ and let
\begin{equation*}
	\bsom := 2\esssup_{\mbcq} |u| +\tail(|u|; \mbcq), \quad \bsmu^+:=\esssup_{{\mcq_o}}u,
	\txt{and}
	\bsmu^-:=\essinf_{{\mcq_o}}u,
\end{equation*}
then it is easy to see that 
\[
\essosc_{\mcq_o} u   = \bsmu^+ - \bsmu^-\leq \bsom.
\]

For some positive constant $\tau \in (0,\tfrac14)$ that will be made precise later on, we have two cases
\refstepcounter{equation}
\def\mynames{\theequation}
\begin{align}[left=\empheqlbrace]
	&\mbox{when $u$ is \emph{near} zero:}  \bsmu^-\le\bsom  \,\,\text{and} \,\,
	\bsmu^+\ge-\bsom,\tag*{(\mynames$_{a}$)}\label{Eq:Hp-main1singh}\\
	&\mbox{when $u$ is \emph{away} from zero:} \bsmu^- >\bsom \,\,\text{or}\,\, \bsmu^+<-\bsom.\tag*{(\mynames$_{b}$)}\label{Eq:Hp-main2sing}
\end{align}
Furthermore, we will always assume the following is satisfied:
\begin{equation*}
	\bsmu^+ -\bsmu^- >\tfrac12\bsom.
\end{equation*} since in the other case, we trivially get the reduction of oscillation.


\subsection{Reduction of Oscillation near zero when \texorpdfstring{\tlcref{Eq:Hp-main1singh}}. holds}\label{singreducoscnearzero}

\begin{description}
	\item[Lower bound for $\bsmu^-$:] In this case, we have
	\begin{equation*}
		\bsmu^- = \bsmu^- - \bsmu^+ + \bsmu^+ \geq -\bsom + \bsmu^+ \geq -2\bsom.
	\end{equation*}
	\item[Upper bound for $\bsmu^+$:] In this case, we have
	\begin{equation*}
		\bsmu^+ = \bsmu^+ - \bsmu^- + \bsmu^- \leq \bsom + \bsmu^- \leq 2\bsom.
	\end{equation*}
\end{description}
In particular, \tlcref{Eq:Hp-main1singh} implies
\begin{equation}\label{boundmusing}
	|\bsmu^{\pm}| \leq 2\bsom.
\end{equation}

\begin{assumption}\label{assump1sing}
	Observe that one of the following two cases must be true: 
	\begin{equation}\label{Eq:1st-alt-meassing}
		\begin{array}{l}
		\abs{\left\{u(\cdot, - (c_o\varrho)^{sp})\geq \bsmu^-+\tfrac{\bsom}{4} \right\}
			\cap 
			{B}_{c_o\varrho}}\geq  \tfrac12 |{B}_{c_o\varrho}|,\\
		\abs{\left\{u(\cdot, - (c_o\varrho)^{sp})\leq \bsmu^+-\tfrac{\bsom}{4} \right\}
			\cap 
			{B}_{c_o\varrho}}\geq  \tfrac12 |{B}_{c_o\varrho}|.
		\end{array}
	\end{equation}
  Since both the cases are treated similarly due to the symmetry of the alternatives, without loss of generality, we shall assume the first inequality in \cref{Eq:1st-alt-meassing} holds. 
\end{assumption}

We will now apply \cref{Prop:1:1} with $\alpha = \tfrac12$, $\xi = 8$, $M= \tfrac14 \bsom$ to get constants $\eta = \eta(\datanb{}) \in (0,1)$, $\de = \de(\datanb{})\in(0,1)$ and $c_o = c_o(\datanb{,\eta})$ such that the following holds:
\[
\tail((u - \bsmu^-)_-;c_o\varrho,0,(-(c_o\varrho)^{sp},0)) \leq\eta \bsom \txt{and} u \geq \bsmu^- + \eta  \bsom \quad \text{on} \,\, B_{c_o\varrho} \times (-\tfrac{\de}2(c_o\varrho)^{sp},0).
\]
Note that since $p<2$, we automatically see that $|\bsmu^{\pm}|\leq \xi M = 2 \bsom$ holds due to \cref{boundmusing}.


\section{Reduction of oscillation for a scaled parabolic fractional \texorpdfstring{$p$}.-Laplace type equations - Preliminary lemmas}
\label{section6}
\subsection{Reduction of Oscillation away from  zero}
We are in the case \tlcref{Eq:Hp-main2} holds and without loss of generality, we consider the case $\bsmu^- > \tau \bsom$ noting that the other case follows by replacing $u$ with $-u$. 
\begin{assumption}\label{assump3}
	Without loss of generality, we consider the case $\bsmu^- > \tau \bsom$. Let us further assume that the following bound holds: 
	\[
	\bsmu^- \leq \frac{1+\tau}{1-\mathring{\eta}} \bsom,
	\]
	for some $\mathring{\eta} = \mathring{\eta}(\datanb{}) \in (0,1)$ to be eventually chosen according to \descrefnormal{step8deg}{Step 8} of \cref{section10} in the degenerate case. Note that with the choice, we have $\frac{1+\tau}{1-\mathring{\eta}} \geq 1$.  See \cref{rmkdeg5.15} on how this assumption is always satisfied.
\end{assumption}

\begin{remark}\label{rmkdeg5.15}
	In order to obtain H\"older regularity, we iteratively assume that \tlcref{Eq:Hp-main1} holds and denote $i_o$ to be the first time when \tlcref{Eq:Hp-main2} occurs. If $i_o=0$, then we trivially have $\bsmu^- \leq \bsom$ from the choice of \cref{defbsom}. On the other hand, if $i_o \geq 1$, then \tlcref{Eq:Hp-main2} implies the following holds:
	\[
	\text{either} \quad\bsmu_{i_o}^-> \tau\bsom_{i_o}\;\quad
	\text{or}\quad
	\;\bsmu_{i_o}^+<-\tau\bsom_{i_o}.
	\]
	We work with $\bsmu_{i_o}^->\tau\bsom_{i_o}$ with the other case being analogous.
	Since $i_o$ is the first index for this to happen,
	we have $\bsmu_{i_o-1}^-\le \tau\bsom_{i_o-1}$ (here is where we need to assume $i_o \geq 1$) and
	\[
	\bsmu_{i_o}^-
	\leq \bsmu_{i_o}^+
	\leq
	\bsmu_{i_o-1}^{+}
	=
	\bsmu_{i_o-1}^{-} + \bsmu_{i_o-1}^{+} - \bsmu_{i_o-1}^{-}
	\le
	\bsmu_{i_o-1}^{-} + \bsom_{i_o-1}
	\le
	(1+\tau)\bsom_{i_o-1} = \frac{1+\tau}{1-\mreta}\bsom_{i_o},
	\]
	where we assumed $\bsom_{i_o} = (1-\mreta) \bsom_{i_o-1} =(1-\mreta)^2 \bsom_{i_o-2} \ldots = (1-\mreta)^{i_o} \bsom $. Recall $\mreta = \mreta(\datanb{})$ is chosen according to \descrefnormal{step8deg}{Step 8} of \cref{section10} in the degenerate case (analogous choice is made in the singular case also). 
	As a result, we have
	\begin{equation}\label{Eq:5:4}
		\tau\bsom_{i_o}
		<
		\bsmu_{i_o}^-\le\frac{1+\tau}{1-\mreta}\bsom_{i_o}.
	\end{equation}
	In particular, \cref{assump3} is always satisfied. 
\end{remark}

\begin{remark}
	By an abuse of notation, we shall suppress keeping track of $i_o$ in the next section and instead denote $\bsom_{i_o}$ by $\bsom$. 
\end{remark}

\begin{definition}\label{defvw}
	Let ${\mcq}_o={\mcq}_{\varrho} = B_{\varrho}(x_o) \times (t_o-\bar A\varrho^{sp},t_o] \subset \mbcq $ where $\bar A \geq 1$ will be chosen later in terms of $\data$. Let $w:=\tfrac{u}{\bsmu^-}$ and denote $v=w^{p-1}$ in $\mcq_o$, then making use of \cref{assump3}, we see that $v$ satisfies
	\begin{equation}\label{Eq:6:9}
		1 \leq w \leq \frac{1+\tau}{\tau} \qquad \text{ a.e in }\quad {\mcq_o},
	\end{equation}
	where we made use of the following bound:
	\[
	\bsmu^- \leq u \leq \bsmu^+ - \bsmu^- + \bsmu^- \leq \boldsymbol \om + \bsmu^- \leq (1+\tfrac{1}{\tau})\bsmu^-. 
	\]
	Moreover, from the scale invariance of \cref{maineq}, we see that $w$ also solves \cref{maineq}. 
	
\end{definition}

\begin{definition}\label{vwdef5.2}
	Let us denote $\bsmu^-_w  =  \frac{\essinf_{\mcq_o} u}{\bsmu^-}=	\essinf_{\mcq_o} w$ and  $\bsmu^+_w   = \frac{\esssup_{\mcq_o} u}{\bsmu^-}=	\esssup_{\mcq_o} w$. Analogously, denote $\bsmu^-_v = (\bsmu^-_w)^{p-1}=	(\essinf_{\mcq_o} w)^{p-1} =\essinf_{\mcq_o} v$ and  $\bsmu^+_v = (\bsmu^+_w)^{p-1}= 	(\esssup_{\mcq_o} w)^{p-1} =\esssup_{\mcq_o} v$.
	We shall denote $\bsom_w = \frac{\bsom}{\bsmu^-}$ where $\bsom \geq \essosc_{\mcq_o} u$,  then it is easy to see that 
	\[
	\bsom \geq \bsmu^+ - \bsmu^- \qquad \Longrightarrow \qquad \bsom_w \geq \bsmu^+_w - \bsmu^-_w.
	\]
	
	Since $w$ satisfies \cref{Eq:6:9}, then for any $x,y \in \mbcq$ with $v(x)\geq v(y)$, we have
	\[
	v(x) - v(y) = w^{p-1}(x) - w^{p-1}(y) \overlabel{alg_lem}{\leq} (p-1)\lbr \tfrac{1+\tau}{\tau}\rbr^{p-1} (w(x)- w(y)).
	\]
	In particular, this implies that if we take $\bsom_v := (p-1)\lbr \tfrac{1+\tau}{\tau}\rbr^{p-1}  \bsom_w$, then we have
	\[
	\bsom_v \geq 	\bsmu^+_v - \bsmu^-_v.
	\]
	Moreover, making use of \cref{Eq:5:4} and $\bsom_w = \frac{\bsom}{\bsmu^-}$, we also have
	\begin{equation}\label{boundsonomega}
		\mathfrak{C}_1:= (p-1)\frac{1-\mathring{\eta}}{1+\tau}\left(\frac{1+\tau}{\tau}\right)^{p-1}\leq {\bsom}_v\leq \frac{p-1}{\tau}\left(\frac{1+\tau}{\tau}\right)^{p-1}=:\mathfrak{C}_2.
	\end{equation}
\end{definition}
\subsection{Energy estimates}
\begin{lemma}\label{lemma6.3}
	With $v$ and $w$ be as in \cref{defvw} and $k \in \RR$ be given, then
	there exists a constant $\bsc  >0$ depending only on the data such that
	for all  every non-negative, piecewise smooth cut-off function
	$\zeta(x,t) = \zeta_1(x)\zeta_2(t)$ vanishing on $\partial_p \mcq_o$,  there holds
	\begin{multline*}
		\esssup_{t_o-\bar A \varrho^{sp}<t<t_o}\int_{B_{\varrho}\times\{t\}}	
		\zeta^p(v-k)_{\pm}^2\,dx +
		\iint_{\mcq} (v-k)_{\pm}(x,t)\zeta_1^p(x)\zeta_2^p(t)\lbr \int_{B_{\varrho}}\frac{(v-k)_{\mp}^{p-1}(y,t)}{|x-y|^{n+sp}}\,dy\rbr\,dx\,dt 
		\\
		+\iiint_{\mcq_o} \frac{|(v-k)_{\pm}(x,t)-(v-k)_{\pm}(y,t)|^p}{|x-y|^{n+ps}}\max\{\zeta_1(x),\zeta_1(y)\}^p\zeta_2^p(t)\,dx\,dy\,dt \\ 
		\begin{array}{rcl}
			&\leq&
			\boldsymbol{C} \iiint_{\mcq_o} \frac{\max\{(v-k)_{\pm}(x,t),(v-k)_{\pm}(y,t)\}^{p}|\zeta(x,t)-\zeta(y,t)|^p}{|x-y|^{n+sp}}\,dx\,dy\,dt
			\\
			&&+
			\iint_{\mcq}\mathfrak (v-k)_\pm^2 (u,k)|\partial_t\zeta^p| \,dx\,dt
			+\int_{B_{\varrho}\times \{t=t_o-\bar A\varrho^{sp}\}} \zeta^p (v-k)_\pm \,dx 
			\\
			&&+ \,\boldsymbol{C}\iint_{\mcq_o}\zeta_2^p(t)\zeta_1^p(x)(v-k)_{\pm}(x,t)\lbr \esssup\limits_{\stackrel{t\in (t_o - \bar A \varrho^{sp},t_o)}{x \in \spt \zeta_1}}\left[\int_{{\RR^n\setminus B_{\varrho}}} \frac{(w- k^{\frac{1}{p-1}})_{\pm}(y,t)^{p-1}}{|x-y|^{n+sp}}\,dy\right]\rbr\,dx\,dt.
		\end{array}
	\end{multline*}
\end{lemma}
\begin{proof}
Given $k \in \RR$, we shall use $(v-k)_{\pm}\zeta^p(x,t)$ (recall $v=w^{p-1}$) where $\zeta(x,t) - \zeta_1(x)\zeta_2(t)$ is a piecewise smooth cut-off function vanishing on $\pa_p \mcq_o$ (see \tlcref{not10} for notation) as a test function in 
 \begin{equation*}
 	\partial_t (|w|^{p-2}w) + \text{P.V.}\int_{\RR^n} |w(x,t)-w(y,t)|^{p-2}(w(x,t)-w(y,t))K(x,y,t)\,dy=0.
 \end{equation*}
We then have
\begin{multline*}
	-\left.\int_{B_{\varrho}(x_o)} (w^{p-1})(v-k)_{\pm}\zeta^p\,dx\right|_{\{t=t_o-\bar A\varrho^{sp}\}}^{\{t=t_o\}}  - \iint_{\mcq_o} (w^{p-1})(v-k)_{\pm} (\partial_t\zeta^p)\,dx\,dt\\
	+\int_{t_o-\bar A \varrho^{sp}}^{t_o}\iint_{\RR^n \times \RR^n}\frac{|w(x,t)-w(y,t)|^{p-2}(w(x,t)-w(y,t))((v-k)_{\pm}\zeta^p(x,t)-(v-k)_{\pm}\zeta^p(y,t))}{|x-y|^{n+sp}}\,dz\,dt
	=0.
\end{multline*}
The first two terms can be estimated in the standard way, see for example \cite[Appendix B]{liaoHolderRegularityParabolic2022}. Hence, we shall only estimate the nonlocal term, which we can split into two parts:
\begin{description}[leftmargin=*]
	\item[The estimate of ${\I}_1$:] 
	
	Recall that
	\begin{equation*}
		{\I}_1:=\frac{1}{2}\iiint_{\mcq_o} |w(x,t)-w(y,t)|^{p-2}(w(x,t)-w(y,t))((v-k)_{\pm}(x,t)\zeta_1^p(x)-(v-k)_{\pm}(y,t)\zeta_1^p(y))\zeta_2^p(t)\,d\mu\,dt,
	\end{equation*}
where we have denoted $d\mu := \frac{dz}{|x-y|^{n+sp}}$.
	
	To estimate the nonlocal terms, we consider the following cases pointwise for $(x,t)$ and $(y,t)$ in $B_{\varrho} \times B_{\varrho}\times I$ where $I:=(t_o-\bar A \varrho^{sp},t_o)$. Furthermore, we shall only consider the case $(v-k)_{+}$ noting that the case $(v-k)_{-}$ can be handled analogously. Recalling the notation from \tlcref{not11}, for any fixed $t\in I$, we claim that
	\begin{itemize}
		\item If $x \notin A_+(k,\varrho,t) = \{v \geq k\}\cap \mcq_o$ and $y \notin A_+(k,\varrho,t)$ then
		\begin{equation*}
			|w(x,t)-w(y,t)|^{p-2}(w(x,t)-w(y,t))((v-k)_{+}(x,t)\zeta_1^p(x)-(v-k)_{+}(y,t)\zeta_1^p(y))=0.
		\end{equation*}
		This estimate follows from the fact that $((v-k)_{+}(x,t)\zeta_1^p(x)-(v-k)_{+}(y,t)\zeta_1^p(y))=0$ when $x \notin A_+(k,\varrho,t)$ and $y \notin A_+(k,\varrho,t)$. 
		\item If $x \in A_+(k,\varrho,t)$ and $ y \notin A_+(k,\varrho,t)$ then
		\begin{multline}\label{eq:B}
			|w(x,t)-w(y,t)|^{p-2}(w(x,t)-w(y,t))((v-k)_{+}(x,t)\zeta_1^p(x)-(v-k)_{+}(y,t)\zeta_1^p(y))\\
			\geq\bsc_{p,\tau}\min\{2^{p-2},1\}\left[|(v-k)_+(x,t)-(v-k)_+(y,t)|^p+(v-k)_-(y,t)^{p-1}(v-k)_+(x,t)\right]\zeta_1(x)^p.
		\end{multline}
		
		To obtain estimate \cref{eq:B}, we note that 
		\begin{multline*}
			|w(x,t)-w(y,t)|^{p-2}(w(x,t)-w(y,t))(v-k)_{+}(x,t)\zeta_1^p(x)-(v-k)_{+}(y,t)\zeta_1^p(y))\\
			\begin{array}{rcl}
			  &\overred{mdeg5.4a}{a}{=} &\lbr|v^{\frac{1}{p-1}}(x,t)-v^{\frac{1}{p-1}}(y,t)|^{p-2}(v^{\frac{1}{p-1}}(x,t)-v^{\frac{1}{p-1}}(y,t))\rbr (v-k)_{+}(x,t)\zeta_1^p(x)\\
			  & \overred{mdeg5.4b}{b}{\geq} & \bsc_{p,\tau} \lbr|v(x,t)-v(y,t)|^{p-2}(v(x,t)-v(y,t))\rbr (v-k)_{+}(x,t)\zeta_1^p(x)\\
			   & \overred{mdeg5.4c}{c}{=} & \bsc_{p,\tau} \lbr(v-k)_{+}(x,t)+(v-k)_{-}(y,t)\rbr^{p-1} (v-k)_{+}(x,t)\zeta_1^p(x)
			  \end{array}
		\end{multline*} where to obtain \redref{mdeg5.4a}{a}, we note that $v = w^{p-1}$ on $\mcq_o$ (see \cref{defvw}); to obtain \redref{mdeg5.4b}{b}, we note that when $x \in A_+(k,\varrho,t)$ and $y \notin A_+(k,\varrho,t)$, we have $v(x,t) \geq v(y,t)$ and thus we can apply \cref{alg_lem} with $q=\tfrac{1}{p-1}$, $c_1 = 1$ and $c_2 = \tfrac{1+\tau}{\tau}$; to obtain \redref{mdeg5.4c}{c}, we again used the fact that  $x \in A_+(k,\varrho,t)$ and $y \notin A_+(k,\varrho,t)$ which says $v(x,t) \geq k\geq v(y,t)$. 
		The estimate follows by an application of  Jensen's inequality when $p \leq 2$ and applying \cref{pineq1} with $\theta=0$ in the case $p \geq 2$ .
		\item If $x, y \in A_+(k,\varrho,t)$ then
		\begin{multline}\label{eq:C}
			|w(x,t)-w(y,t)|^{p-2}(w(x,t)-w(y,t))((v-k)_{+}(x,t)\zeta_1^p(x)-(v-k)_{+}(y,t)\zeta_1^p(y))\\
			 \geq \frac{1}{2}|(v-k)_{+}(x,t)-(v-k)_{+}(y,t)|^p\max\{\zeta_1(x),\zeta_1(y)\}^p\\
			 -C(p)\max\{(v-k)_{+}(x,t),(v-k)_{+}(y,t)\}^p|\zeta_1(x)-\zeta(y)|^p.
		\end{multline}
		Since $x,y \in A_+(k,\varrho,t)$, we have 
		\begin{multline*}
			|w(x,t)-w(y,t)|^{p-2}(w(x,t)-w(y,t))((v-k)_{+}(x,t)\zeta_1^p(x)-(v-k)_{+}(y,t)\zeta_1^p(y))\\
			=\left\{ \begin{array}{ll}
				(w(x,t)-w(y,t))^{p-1}(\zeta_1^p(x)(v-k)_{+}(x,t)-\zeta_1^p(y)(v-k)_{+}(y,t)) & \text{when} \, w(x,t) \geq w(y,t),\\
				(w(y,t)-w(x,t))^{p-1}(\zeta_1^p(y)(v-k)_{+}(y,t)-\zeta_1^p(x)(v-k)_{+}(x,t)) & \text{when} \, w(x,t) \leq w(y,t).
			\end{array}\right.
		\end{multline*}
		Thus, without loss of generality, we can assume $w(x,t) \geq w(y,t)$. Noting that $x,y \in A_+(k,\varrho,t)$, in the case $\zeta_1(x)\geq \zeta_1(y)$, we get
		\begin{multline*}	|w(x,t)-w(y,t)|^{p-2}(w(x,t)-w(y,t))((v-k)_{+}(x,t)\zeta_1^p(x)-(v-k)_{+}(y,t)\zeta_1^p(y))\\
			\begin{array}{rcl}
				& = & |v^{\frac{1}{p-1}}(x,t)-v^{\frac{1}{p-1}}(y,t)|^{p-1}((v-k)_{+}(x,t)\zeta_1^p(x)-(v-k)_{+}(y,t)\zeta_1^p(y))\\
				& \overlabel{alg_lem}{\geq} & \bsc_{p,\tau} |v(x,t)-v(y,t)|^{p-1}((v-k)_{+}(x,t)\zeta_1^p(x)-(v-k)_{+}(y,t)\zeta_1^p(y))\\
				& \geq & \bsc_{p,\tau}((v-k)_{+}(x,t)-(v-k)_{+}(y,t))^p \zeta_1^p(x).
			\end{array}
		\end{multline*}
				Analogously, in the case $\zeta_1(x)\leq \zeta_1(y)$, we get
				\begin{multline*}	|w(x,t)-w(y,t)|^{p-2}(w(x,t)-w(y,t))((v-k)_{+}(x,t)\zeta_1^p(x)-(v-k)_{+}(y,t)\zeta_1^p(y))\\
					\begin{array}{rcl}
						& = & |v^{\frac{1}{p-1}}(x,t)-v^{\frac{1}{p-1}}(y,t)|^{p-1}((v-k)_{+}(x,t)\zeta_1^p(x)-(v-k)_{+}(y,t)\zeta_1^p(y))\\
						& \overlabel{alg_lem}{\geq} & \bsc_{p,\tau} |v(x,t)-v(y,t)|^{p-1}((v-k)_{+}(x,t)\zeta_1^p(x)-(v-k)_{+}(y,t)\zeta_1^p(y))\\
						& {\geq} & \bsc_{p,\tau} \lbr[[](v-k)_{+}(x,t)-(v-k)_{+}(y,t)\rbr[]]^{p}(\zeta_1(y))^p\\
						&&-\bsc_{p,\tau}\lbr[[](v-k)_{+}(x,t)-(v-k)_{+}(y,t)\rbr[]]^{p-1}(v-k)_{+}(x,t)(\zeta_1^p(x)-\zeta_1^p(y))\\
						& \overred{mabs}{a}{\geq}& \bsc_{p,\tau} \lbr[[](v-k)_{+}(x,t)-(v-k)_{+}(y,t)\rbr[]]^{p}(\zeta_1(y))^p\\
						&& -\bsc_{p,\tau} (v-k)_{+}^p(x,t)|\zeta_1(y)-\zeta(x)|^p
					\end{array}
				\end{multline*}
			where to obtain \redref{mabs}{a}, we apply \cref{pineq3} with $a=\zeta_1(y)$ and $b=\zeta_1(x)$ and 
			$\ve=\tfrac{1}{2}\tfrac{(v-k)_{+}(x,t)-(v-k)_{+}(y,t)}{(v-k)_{+}(x,t)}$  to obtain
			\begin{multline*}
				((v-k)_{+}(x,t)-(v-k)_{+}(y,t))^{p-1}(v-k)_{+}(x,t)(\zeta_1^p(x)-\zeta_1^p(y))\\
				\leq \frac{1}{2}((v-k)_{+}(x,t)-(v-k)_{+}(y,t))^p\zeta_1^p(y)+[2(p-1)]^{p-1}(v-k)_{+}^p(x,t)(\zeta_1(y)-\zeta(x))^p.
			\end{multline*}
		Combining the previous three estimates,  we obtain \cref{eq:C} when $w(x,t) \geq w(y,t)$.  The case {$w(x,t) \leq w(y,t)$}  can be handled analogously by interchanging the role of $x$ and $y$.
	\end{itemize}

	As a consequence of these cases and \cref{boundsonKernel}, we have 
	\begin{equation*}
		\begin{array}{rcl}
		{\I}_1&\geq& \bsc_{p,\tau}\Biggr[\iiint_{\mcq_o} \frac{|(v-k)_{+}(x,t)-(v-k)_{+}(y,t)|^p}{|x-y|^{n+ps}}\max\{\zeta_1(x),\zeta_1(y)\}^p\zeta_2^p(t)\,dx\,dy\,dt \\
		&&+\iint_{\mcq_o} \zeta_2^p(t)\zeta_1^p(x)(v-k)_{+}(x,t)\left(\int_{B_{\varrho}}\frac{(v-k)_-(y,t)^{p-1}}{|x-y|^{n+sp}}\,dy\right)\,dx\,dt\Biggr]\\
		&&-\bsc_{p,\tau}\iiint_{\mcq_o}\max\{(v-k)_{+}(x,t),(v-k)_{+}(y,t)\}^p\frac{|\zeta_1(x)-\zeta_1(y)|^p}{|x-y|^{N+sp}}\zeta_2^p(t)\,dx\,dy\,dt.
		\end{array}
	\end{equation*}
	
	\item[The estimate of ${\I}_2$:] Recall that
	\begin{equation*}
		\begin{array}{r@{}cl}
		{\I}_2\,\,\,&:=&\int_{I} \iint_{(\RR^n\setminus B_{\varrho})\times B_{\varrho}}\hspace*{-1.2cm}\frac{|w(x,t)-w(y,t)|^{p-2}(w(x,t)-w(y,t))}{|x-y|^{n+sp}}((v-k)_{+}(x,t)\zeta_1^p(x)-(v-k)_{+}(y,t)\zeta_1^p(y))\zeta_2^p(t)\,dz\,dt\\
		&=&\int_{I}\int_{B_{\varrho}}\zeta_2^p(t)\zeta_1^p(x)(v-k)_{+}(x,t)\left[\int_{\RR^n\setminus B_{\varrho}}\frac{|w(x,t)-w(y,t)|^{p-2}(w(x,t)-w(y,t))}{|x-y|^{n+sp}}\,dy\right]\,dx\,dt\\
		&=& \bsc\underbrace{\int_{I} \int_{A_+(k,\varrho,t)}\zeta_2^p(t)\zeta_1^p(x)(v-k)_{+}(x,t)\left[\int_{\RR^n\setminus B_{\varrho} \cap \{w(x,t)\geq w(y,t)\}} \frac{(w(x,t)-w(y,t))^{p-1}}{|x-y|^{n+sp}}\,dy\right]\,dx\,dt}_{T_1}\\
		&&\qquad - \bsc \underbrace{\int_{I} \int_{A_+\left(k,\varrho,t\right)}\zeta_2^p(t)\zeta_1^p(x)(v-k)_{+}(x,t)\left[\int_{\RR^n\setminus B_{\varrho} \cap \{w(y,t)> w(x,t)\}} \frac{(w(y,t)-w(x,t))^{p-1}}{|x-y|^{n+sp}}\,dy\right]\,dx\,dt}_{T_2}.
		\end{array}
	\end{equation*}
	 Ignoring the nonnegative term $T_1$, we estimate $T_2$ as follows:
	\begin{equation*}
		\begin{array}{rcl}
		T_2&=& \int_{I} \int_{B_{\varrho}}\zeta_2^p(t)\zeta_1^p(x)(v-k)_{+}(x,t)\left[\int_{{\RR^n\setminus B_{\varrho}\cap \{w(y,t)> w(x,t)\}}} \hspace*{-0.5cm}\frac{(w(y,t)-w(x,t))^{p-1}}{|x-y|^{n+sp}}\,dy\right]\,dx\,dt\\
		&\overred{t2a}{a}{\leq}& \int_{I} \int_{B_{\varrho}}\zeta_2^p(t)\zeta_1^p(x)(v-k)_{+}(x,t)\left[\int_{{\RR^n\setminus B_{\varrho}}} \frac{(w(y,t)-k^{\frac{1}{p-1}})_+^{p-1}}{|x-y|^{n+sp}}\,dy\right]\,dx\,dt\\
		&\leq & \iint_{\mcq_o}\zeta_2^p(t)\zeta_1^p(x)(v-k)_{+}(x,t)\lbr \esssup\limits_{t\in (t_o - \bar A \varrho^{sp},t_o)}\left[\int_{{\RR^n\setminus B_{\varrho}}} \frac{(w- k^{\frac{1}{p-1}})_+(y,t)^{p-1}}{|x-y|^{n+sp}}\,dy\right]\rbr\,dx\,dt,
	\end{array}
	\end{equation*}
where to obtain \redref{t2a}{a}, we made use of the fact that $v(x,t) \geq k \Longrightarrow w(x,t)^{p-1} \geq k$ on $\mcq_o$.
\end{description}

Combining all the estimates gives the desired proof of the energy inequality.
\end{proof}

\subsection{Some measure and pointwise lemmas}

The proof follows closely the ideas developed in \cite{liaoHolderRegularityParabolic2022} and we will refer to \cite{liaoHolderRegularityParabolic2022} for the details. The first lemma is the De Giorgi type lemma:

\begin{lemma}\label{lemma5.4}
	Let $v$ and $w$ be as in \cref{defvw}. Given $\de_1, \varepsilon_1 \in (0,1)$, set $\theta = \de_1 (\varepsilon_1 \bsom_v)^{2-p}$ and assume $\mcq_{c_o\varrho}^{\theta} = B_{c_o\varrho} \times (-\theta (c_o\varrho)^{sp},0) \subset \mbcq$. Then there exists a constant $\nu_3 = \nu_3(\datanb{,\delta_1}) \in (0,1)$ and $c_o = c_o(\tau,\varepsilon_1,\mreta)$ such that if 
	\begin{equation*}
		\left|\left\{
		\pm\left(\bsmu^{\pm}_v-v\right)\le \varepsilon_1 \bsom_v\right\}\cap \mcq_{c_o\varrho}^\theta\right|
		\le
		\nu_3|\mcq_{c_o\varrho}^{\theta}|,
	\end{equation*}
	holds along with the assumption 
	\begin{equation}\label{vwLm:3:3:hypothesis}
		c_o^{\frac{sp}{p-1}}\tail((u-\bsmu^{\pm})_{\pm};\mcq_o) \leq \varepsilon_1\bsom,
	\end{equation}
	then the following conclusion follows:
	\begin{equation*}
		\pm\left(\bsmu^{\pm}_v-v\right)\ge\tfrac{1}2\varepsilon_1 \bsom_v
		\quad
		\mbox{ on }\quad \mcq_{\frac{c_o\varrho}{2}}^\theta = B_{\frac{c_o\varrho}2} \times \left(t_o-\theta\lbr \tfrac{c_o\varrho}{2}\rbr^{sp}, t_o\right].
	\end{equation*}
We note that $\nu_3 \approx \delta_1^q$ for some $q =q(n,s)>1$. 
\end{lemma}
\begin{proof}
	The proof is very similar to that of \cite[Lemma 3.1]{liaoHolderRegularityParabolic2022}, but we have to keep track of the relationship between $v$ and $w$ in estimating the $\tail$ term. In order to do this, we only highlight the main differences. With notation as in \cite[Lemma 3.1]{liaoHolderRegularityParabolic2022}, we take $k_i = \bsmu^-_v + \frac{\varepsilon_1\bsom_v}{2}+\frac{\varepsilon_1\bsom_v}{2^{i+1}}$ and $\varrho_i = \frac{c_o\varrho}{2} + \frac{c_o\varrho}{2^{i+1}}$ , then we need to impose the following bound for the proof to hold:
	\[
	 \esssup_{t\in (t_o - \theta \varrho_i^{sp},t_o)}\left[\int_{{\RR^n\setminus B_{\varrho_i}}} \frac{(w- k_i^{\frac{1}{p-1}})_-(y,t)^{p-1}}{|y|^{n+sp}}\,dy\right] \leq \bsc b^i \frac{(\varepsilon_1 \bsom_v)^{p-1}}{(c_o\varrho)^{sp}},
	\]for some universal constants $\bsc = \bsc(\datanb{})$ and $b = b(\datanb{})$. In order to see this, we proceed as follows:
	\begin{equation*}
		\begin{array}{rcl}
			\int_{{\RR^n\setminus B_{\varrho_i}}} \frac{(w- k_i^{\frac{1}{p-1}})_-(y,t)^{p-1}}{|y|^{n+sp}}\,dy & \overred{wv5.6a}{a}{=} &	\int_{{\RR^n\setminus B_{\varrho}}} \frac{( k_i^{\frac{1}{p-1}}-w(y,t))_+^{p-1}}{|y|^{n+sp}}\,dy\\
			&& + \int_{{B_{\varrho}\setminus B_{\varrho_i}}} \frac{( k_i^{\frac{1}{p-1}} -w(y,t))_+ ^{p-1}}{|y|^{n+sp}}\,dy\\
			& \overred{wv5.6b}{b}{=} &\int_{{\RR^n\setminus B_{\varrho}}} \frac{( k_i^{\frac{1}{p-1}}-w(y,t))_+^{p-1}}{|y|^{n+sp}}\,dy\\
			&& + \int_{{B_{\varrho}\setminus B_{\varrho_i}}} \frac{( k_i^{\frac{1}{p-1}} -v^{\frac{1}{p-1}}(y,t))_+ ^{p-1}}{|y|^{n+sp}}\,dy\\
			& = : & {\J}_1 + {\J}_2.\\
		\end{array}
	\end{equation*}
where to obtain \redref{wv5.6a}{a}, we split the integral into two parts and to obtain \redref{wv5.6b}{b} we note that $v = w^{p-1}$ on $\mcq_o$ (see \cref{defvw}). Now we estimate each of the terms as follows:
\begin{description}
	\item[Estimate for ${\J}_1$:] We see that $1 \leq k_i \leq p\lbr \tfrac{1+\tau}{\tau}\rbr^{p-1}$ which holds due to 
	\[
	k_i \geq \bsmu^-_v = (\bsmu^-_w)^{p-1} \geq 1 \txt{and} k_i \overred{j1a}{a}{\leq} (\bsmu^-_w)^{p-1} + \ve_1 \bsom_v \overred{j1b}{b}{\leq} \lbr \tfrac{1+\tau}{\tau}\rbr^{p-1} + \ve_1 (p-1)\lbr \tfrac{1+\tau}{\tau}\rbr^{p} \overred{j1c}{c}{\leq} p \lbr \tfrac{1+\tau}{\tau}\rbr^{p},
	\]
	where to obtain \redref{j1a}{a}, we made use of \cref{vwdef5.2}; to obtain \redref{j1b}{b}, we made use of \cref{Eq:6:9} along with \cref{vwdef5.2} and finally to obtain \redref{j1c}{c}, we made use of $\ve \leq 1$.   In particular,the bound on $k_i$  implies $k_i^{\frac{1}{p-1}} \leq \bsc_{\tau} k_i$. Using this, we get
	\begin{equation}\label{wv5.8}
		\begin{array}{rcl}
( k_i^{\frac{1}{p-1}}-w(y,t))_+ &\leq& ( \bsc_{\tau}k_i-w(y,t)) \lsb{\chi}{\{w^{p-1} \leq k_i\}}\\
& \overred{wv5.8a}{a}{\leq} & (\bsc_{\tau}k_i - k_i)  + (k_i-w(y,t)) \lsb{\chi}{\{w \leq k_i\}}\\
& \overred{wv5.8b}{b}{\leq} & (\bsc_{\tau}-1) |\bsmu_v^-| + (\bsc_{\tau}-1)(\varepsilon_1\bsom_v) +   2(\varepsilon_1\bsom_v)+ (\bsmu_v^--w(y,t))_+,
		\end{array}
	\end{equation}
where to obtain \redref{wv5.8a}{a}, we made use of $w \geq 1$ and to obtain \redref{wv5.8b}{b}, we substituted the expression for $k_i$. 
We will further estimate the last term appearing on the right hand side of \cref{wv5.8} as follows:
\begin{equation*}
	\begin{array}{rcl}
	(\bsmu_v^--w(y,t))_+ & = & ((\bsmu_w^-)^{p-1} - \bsmu_w^- + \bsmu_w^--w(y,t))_+	\\
	& \apprle &  \bsmu_w^- ((\bsmu_w^-)^{p-2}-1) + (\bsmu_w^--w(y,t))_+\\
	& \overred{vw5.9a}{a}{\apprle} & \bsc_{\tau} |\bsmu_w^-| + \frac{(\bsmu^--u(y,t))_+}{\bsmu^-}\\
	& \overred{vw5.9b}{b}{\apprle}  & \bsc_{\tau,\mreta} |\bsom_w| + \frac{(\bsmu^--u(y,t))_+}{\bsmu^-}\\
	& \overred{vw5.9c}{c}{\apprle}  & \bsc_{\tau,\mreta} |\bsom_v| + \frac{(\bsmu^--u(y,t))_+}{\bsmu^-},
	\end{array}
\end{equation*}
where to obtain \redref{vw5.9a}{a}, we made use of \cref{Eq:6:9} to bound $|(\bsmu_w^-)^{p-2}-1| \leq \bsc_{\tau}$ and made use of \cref{defvw}; to obtain  \redref{vw5.9b}{b}, we use \cref{assump3} to get $\bsmu_w^- = \frac{\essinf_{\mcq_o}u}{\bsmu^-} \leq \bsc_{\tau,\mreta} \frac{\bsom}{\bsmu^-} =\bsc_{\tau,\mreta}  \bsom_w$ and finally to obtain \redref{vw5.9c}{c}, we make use of \cref{vwdef5.2} to get $\bsom_w \approx_{\{p,\tau\}} \bsom_v$ which gives $\bsmu_v^- \leq_{\tau,\mreta} \bsom_v$. 
Combining everything together, we get
\begin{equation*}
	\begin{array}{rcl}
	{\J}_1 &\leq& \frac{1}{(c_o\varrho)^{sp}} \lbr c_o^{sp}(\bsc_{\tau}-1)^{p-1} |\bsom_v|^{p-1} +c_o^{sp} \bsc_{\tau}^{p-1}(\varepsilon_1\bsom_v)^{p-1} +c_o^{sp} |\bsom_v|^{p-1}\rbr \\
	&& + \frac{1}{(\bsmu^-)^{p-1}}\int_{\RR^n \setminus B_{\varrho}}\frac{(\bsmu^--u(y,t))_+^{p-1}}{|y|^{n+sp}}\,dy \\
	& = & \frac{1}{(c_o\varrho)^{sp}} \lbr c_o^{sp}(\bsc_{\tau}-1)^{p-1} |\bsom_v|^{p-1} +c_o^{sp} \bsc_{\tau}^{p-1}(\varepsilon_1\bsom_v)^{p-1} +c_o^{sp} |\bsom_v|^{p-1}\rbr \\
	&& + \frac{1}{(\bsmu^-)^{p-1}}\frac{1}{(c_o\varrho)^{sp}} c_o^{sp}\tailp((u-\bsmu^-)_-;\mcq_o).
	\end{array}
\end{equation*}
Choosing $c_o$ small such that $c_o^{sp}(\bsc_{\tau}-1)^{p-1} \leq \varepsilon_1^{p-1}$, $c_o^{sp}(\bsc_{\tau})^{p-1} \leq \varepsilon_1^{p-1}$, $c_o^{sp} \leq \varepsilon_1^{p-1}$ and making use of \cref{vwLm:3:3:hypothesis} along with the fact that $\lbr \frac{\bsom}{\bsmu^-}\rbr^{p-1} = \bsom_w^{p-1} \overlabel{vwdef5.2}{\approx} \bsom_v^{p-1}$ on $\mcq_o$, we get
\[
{\J}_1 \apprle_{\tau,\mreta} \frac{(\varepsilon_1\bsom_v)^{p-1}}{(c_o\varrho)^{sp}}.
\]
\item[Estimate for ${\J}_2$:] Since the change of variable from $u$ to $v$ is valid on $B_{\varrho}$, we see that 
\begin{equation*}
	\begin{array}{rcl}
{\J}_2 & \overlabel{alg_lem}{\apprle_{p,\tau}} & \int_{{B_{\varrho}\setminus B_{\varrho_i}}} \frac{( k_i -v(y,t))_+ ^{p-1}}{|y|^{n+sp}}\,dy	\\
& \apprle & \frac{2^{isp}}{(c_o\varrho)^{sp}} (\varepsilon_1\bsom_v)^{p-1}. 
\end{array}
\end{equation*}
\end{description}
Once we have this estimate, the rest of the proof follows verbatim as in \cite[Lemma 3.1]{liaoHolderRegularityParabolic2022}.
\end{proof}
The next lemma is also a De Giorgi type iteration involving quantitative initial data.

\begin{lemma}\label{lemma5.45}
	Let $v$ and $w$ be as in \cref{defvw}. Given $\varepsilon_1 \in (0,1)$, set $\theta = (\varepsilon_1 \bsom_v)^{2-p}$ and assume $\mcq_{c_o\varrho}^{\theta} \subset \mcq_o$ (where $\mcq_o$ is from \cref{defvw}), then there exists constants $\bar\nu_3 = \bar\nu_3(\datanb{}) \in (0,1)$ and $c_o = c_o(\tau,\varepsilon_1,\mreta)$ such that if 
	\begin{equation*}
		\pm(\bsmu^{\pm}_v-v(\cdot,t_o))\geq \ve_1\bsom_v\quad\mbox{ on }\quad B_{c_o\rho}(x_0),
	\end{equation*}
	holds along with the assumption 
	\begin{equation}\label{vwLm:3:3:hypothesis2}
		c_o^{\frac{sp}{p-1}}\tail((u-\bsmu^{\pm})_{\pm};\mcq_o) \leq \varepsilon_1\bsom,
	\end{equation}
	then the following conclusion holds:
	\begin{equation*}
		\pm\left(\bsmu^{\pm}_v-v\right)\ge\tfrac{1}2\varepsilon_1 \bsom_v
		\quad
		\mbox{ on }\quad
		B_{\frac{c_o\varrho}2} \times \left(t_o,t_o+\bar\nu_3\theta\lbr{c_o\varrho}\rbr^{sp}\right],
	\end{equation*} provided the cylinders are included in $\mbcq$.
\end{lemma}
\begin{proof}
	The proof follows verbatim as in \cite[Lemma 3.2]{liaoHolderRegularityParabolic2022} once we make use of the $\tail$ estimates as obtained in the proof of \cref{lemma5.4}.
\end{proof}
The next lemma we will need is the expansion of positivity.
\begin{lemma}\label{vwLemma5.5}
	Let $v$ and $w$ be as in \cref{vwdef5.2} and fix two parameters $\varepsilon_1, \alpha_1 \in (0,1)$. Then there exists $\de_1 = \de_1(\datanb{,\alpha_1})\in(0,1)$, $\bar\varepsilon = \bar\varepsilon(\datanb{,\alpha_1})\in(0,1)$ and $c_o = c_o(\datanb{,\varepsilon_1,\tau,\mreta})$ such that if 
	\[
	|\{\pm (\bsmu_v^{\pm}-v(\cdot,\mft)) \geq \varepsilon_1\bsom_v\} \cap B_{c_o\varrho}| \geq \alpha_1 |B_{c_o\varrho}|, 
	\]
	and 
	\[
	c_o^{\frac{sp}{p-1}} \tail((u-\bsmu^{\pm})_{\pm};\mcq_o) \leq \varepsilon_1\bsom,
	\]
	holds, then the following conclusion follows:
	\[
	|\{\pm (\bsmu_v^{\pm}-v(\cdot,t)) \geq \bar\varepsilon\varepsilon_1\bsom_v\} \cap B_{c_o\varrho}| \geq \tfrac12 \alpha_1 |B_{c_o\varrho}| \txt{for all} t \in  (\mft, \mft + \de_1(\varepsilon_1\bsom_v)^{2-p}(c_o\varrho)^{sp}).
	\]
	Moreover the dependence is of the form $\bar{\varepsilon} \approx \alpha_1$ and $\delta_1 \approx \alpha_1^{p+1}$. 
\end{lemma}
\begin{proof}
	The proof follows verbatim as in \cite[Lemma 3.3]{liaoHolderRegularityParabolic2022} once we make use of the $\tail$ estimates as obtained in the proof of \cref{lemma5.4}.
\end{proof}
We now prove a measure shrinking lemma.
\begin{lemma}\label{vwLemma5.6}
	Let $v$ and $w$ be as in \cref{vwdef5.2} and fix four parameters $\alpha_2, \bar\varepsilon_2, \bar\delta_2, \bar\sigma_2 \in (0,1)$. Then there exists  $c_o = c_o(\datanb{,\bar\varepsilon_2,\bar\sigma_2,\tau,\mreta})$ such that if 
	\[
	|\{\pm (\bsmu_v^{\pm}-v(\cdot,t)) \geq \bar\varepsilon_2\bsom_v\} \cap B_{c_o\varrho}| \geq \alpha_2 |B_{c_o\varrho}| \txt{for all} t \in (\mft - \bar\delta_2(\bar\sigma_2\bar\varepsilon_2\bsom_v)^{2-p}(c_o\varrho)^{sp}, \mft), 
	\]
	and 
	\[
	c_o^{\frac{sp}{p-1}} \tail((u-\bsmu^{\pm})_{\pm};\mcq_o) \leq \bar\sigma_2\bar\varepsilon_2\bsom,
	\]
	holds, then the following conclusion follows:
	\[
	|\{\pm (\bsmu_v^{\pm}-v) \leq \bar\sigma_2\bar\varepsilon_2\bsom_v\} \cap \mcq_{c_o\varrho}^{\theta}| \leq \bsc_{\data{,\tau,\mreta}}\frac{\bar\sigma_2^{p-1}}{\bar\delta_2\alpha_2}  |\mcq_{c_o\varrho}^{\tilde\theta}|,
	\]
	where $\mcq_{c_o\varrho}^{\tilde\theta} := B_{c_o\varrho} \times (\mft - \bar\delta_2(\bar\sigma_2\bar\varepsilon_2\bsom_v)^{2-p}(c_o\varrho)^{sp}, \mft)$.
\end{lemma}
\begin{proof}
	The proof follows verbatim as in \cite[Lemma 3.4]{liaoHolderRegularityParabolic2022} once we make use of the $\tail$ estimates as obtained in the proof of \cref{lemma5.4}.
\end{proof}


\section{Reduction of oscillation for a scaled parabolic fractional \texorpdfstring{$p$}.-Laplace type equations (away from zero case) - Degenerate Case}

\subsection{Defining the alternatives}

%

Without loss of generality, let us reduce the oscillation at the point $(0,0)$. Recalling the reference cylinder $\mcq_o = B_{\varrho}(0) \times (-\bar{A}\varrho^{sp},0)$ from \cref{defvw} with $\bar{A}=A$, let us consider the cylinder 
\begin{equation}\label{def_tht}
\mcq_{c_o\varrho}^{\tht} := B_{c_o\varrho} \times (-\tht(c_o\varrho)^{sp},0) \txt{where} \tht:=\left(\tfrac14\bsom_v\right)^{2-p},
\end{equation}
for some $c_o\in (0,1)$ to be eventually determined. 
Moreover, we have the following set inclusion:
\begin{equation*}
\mcq_{c_o\varrho}^{\tht}\subseteq \mcq_{o}\txt{ provided } c_o \leq 1 \quad \text{and}\quad  c_o^{sp} \leq \lbr \frac{\bsom_v}{4}\rbr^{p-2}\bar{A}
\end{equation*}
Moreover, from \cref{vwdef5.2}, we see that the following is satisfied:
\begin{equation*}
	\essosc_{\mcq_{c_o\varrho}(\tht)} v \leq {\bsom}_v.
\end{equation*}
\begin{remark}\label{rmk6.9}
	We will eventually choose $\bar{A} = \lbr \frac{\bsom_v}{\mathbf{a}}\rbr^{2-p}$, for some $\mathbf{a} \gg 1
	$ to be chosen depending only on data satisfying \cref{defdega1} and \cref{defdega2} (see \cref{degconstantsscaled}). In what follows, we are interested in obtaining reduction of oscillation in the cylinder 
	\begin{equation*}
		\mcq_{c_o\varrho}^{\bar{A}} := B_{c_o\varrho} \times (-\bar{A}(c_o\varrho)^{sp},0).
	\end{equation*}
\end{remark}

\begin{definition}
	Without loss of generality, we assume the following is satisfied:
	\begin{equation}\label{degEq:mu-pm-}
		\bsmu_v^+-\bsmu_v^->\frac12 \bsom_v.
	\end{equation}
For a constant $\nu_4$ to be eventually determined according to \descrefnormal{step1}{Step 1} in the proof of \cref{redoscfirstp>2scaled} (see \cref{degconstantsscaled} for all the constants),  we have one of two alternatives:
\begin{description}
	\descitemnormal{DegAlt-$\I$:}{alt1} There exists a time level $\bar{t} \in (-\bar{A}\varrho^{sp},0)$ such that the following holds:
	\begin{equation*}
		\left|\left\{v\leq \bsmu_v^-+\tfrac14 \bsom_v\right\}\cap
		 \mcq_{c_o\rho}^{\tht}(\bar{t})\right|\leq \nu_4\left|\mcq_{c_o\rho}^{\tht}\right|,
	\end{equation*}
where we have denoted $\mcq_{c_o\rho}^{\tht}(\bar{t}) = B_{c_o\varrho} \times (\bar{t}-\tht(c_o\varrho)^sp,\bar{t})$ with $\tht$ as in \cref{def_tht}.
	\descitemnormal{DegAlt-$\II$:}{alt2} OR for every time level $\bar{t} \in (-\bar{A}\varrho^{sp},0)$, the following holds:
	\begin{equation}
		\label{secondaltdeg2}
		\left|\left\{v\leq \bsmu_v^-+\tfrac14 \bsom_v\right\}\cap \mcq_{c_o\rho}^{\tht}(\bar{t})\right|\geq \nu_4\left|\mcq_{c_o\rho}^{\tht}\right|,
	\end{equation}
where we have denoted $\mcq_{c_o\rho}^{\tht}(\bar{t}) = B_{c_o\varrho} \times (\bar{t}-\tht(c_o\varrho)^sp,\bar{t})$ with $\tht$ as in \cref{def_tht}.
\end{description}
\end{definition}
%

\subsection{Reduction of oscillation when \texorpdfstring{\descrefnormal{alt1}{DegAlt-$\I$}}. holds }\label{redoscfirstp>2scaled}

In this section, we work with $v$  near its infimum and the proof will proceed in several steps:
\begin{description}
	\descitemnormal{Step 1:}{step1} Let us first take $\nu_4= \nu_3$, where $\nu_3$ is from \cref{lemma5.4} applied with $\delta_1 = 1$, $\varepsilon_1 = \tfrac14$ and $c_o$ chosen small enough to satisfy \[
	c_o^{\frac{sp}{p-1}}\tail((u-\bsmu^{\pm})_{\pm};\mcq_o) \leq \tfrac14\bsom.
	\] This gives
	\begin{equation}\label{degeq7.10}
	v \geq \bsmu^-_v + \tfrac18\bsom_v \txt{ a.e in } B_{\frac{c_o\varrho}{2}} \times \left(\bar{t}-\theta\lbr \tfrac{c_o\varrho}{2}\rbr^{sp}, \bar{t}\right],
	\end{equation}
	where $\tht$ as defined in \cref{def_tht}. 
	\descitemnormal{Step 2:}{step2} For some $\varepsilon_1 \in (0,\tfrac18)$ to be chosen according to \cref{eq6.21}, we  apply \cref{lemma5.45} with $t_o = \bar{t}$ noting that \cref{degeq7.10} satisfies the first hypothesis of \cref{lemma5.45}. Additionally, we apply \cref{lemma5.45} with $c_o\varrho$ replaced by $\tfrac12c_o\varrho$ to get
	\begin{equation}\label{eq6.20}
	v \geq \bsmu_v^-+\tfrac12 \varepsilon_1 \bsom_v \quad
	\mbox{ on }\quad
	B_{\frac{c_o\varrho}4} \times \left(\bar{t},\bar{t}+\bar\nu_3(\varepsilon_1\bsom_v)^{2-p}\lbr{\tfrac14c_o\varrho}\rbr^{sp}\right],
	\end{equation}
	provided $c_o = c_o(\datanb{,\varepsilon_1,\tau,\mreta})$ is chosen small enough to ensure the tail alternative in  \cref{vwLm:3:3:hypothesis2} is satisfied.
	\item[Step 3:] From \descrefnormal{step2}{Step 2}, we choose $\varepsilon_1$ small such that $\bar\nu_3 (\varepsilon_1\bsom_v)^{2-p}\lbr{\tfrac14c_o\varrho}\rbr^{sp} \geq \lbr \tfrac{\bsom_v}{\mathbf{a}}\rbr^{2-p}(c_o\varrho)^{sp}$, which gives
	\begin{equation}\label{eq6.21}
		\varepsilon_1\leq \lbr\frac{\bar\nu_3}{4\mathbf{a}^{p-2}}\rbr^{\frac{1}{p-2}}.
	\end{equation}
In particular, with this choice of $\varepsilon_1$, we have \cref{eq6.20} holds till $t=0$.
	\item[Step 4:] Combining the estimates, we have 
	\begin{equation*}
		\essosc_{\mcq} v \leq \lbr 1-\tfrac12 \varepsilon_1\rbr \bsom_v \quad \text{where} \quad \mcq:= B_{\frac{c_o\varrho}4} \times \left(-\bar\nu_3(\varepsilon_1\bsom_v)^{2-p}\lbr{\tfrac14c_o\varrho}\rbr^{sp},0\right].
	\end{equation*}
\end{description}

\begin{remark}
	 Note that the constant $\mathbf{a}$ in \cref{rmk6.9} is yet to be determined and $c_o$ is to be further restricted in the subsequent steps of the proof. In application of \cref{vwLemma5.5}, we see that $c_o$ depends on $\varepsilon_1$ which in turn depends on $\mathbf{a}$ according to \cref{eq6.21}. We will keep track of the constants to ensure suitable choice of $\mathbf{a}$ and $c_o$, see \cref{degconstantsscaled}. 
\end{remark}
This completes the proof of the reduction of oscillation when \descrefnormal{alt1}{DegAlt-$\I$} holds. 
\subsection{Reduction of oscillation when \texorpdfstring{\descrefnormal{alt2}{DegAlt-$\II$}}. holds}\label{redoscsecondp>2scaled}
In this section, we work with $v$  near its supremum and the proof will proceed in several steps:

\begin{description}
	\item[Step 1:] Using \cref{degEq:mu-pm-},  we see that \cref{secondaltdeg2} can be rewritten as 
	\begin{equation}
		\label{secondaltdeg2rewrite}
		\left|\left\{v\leq \bsmu_v^+-\tfrac14 \bsom_v\right\}\cap \mcq_{c_o\rho}^{\tht}(\bar{t})\right|\geq \nu_4\left|\mcq_{c_o\rho}(\tht)\right|,
	\end{equation}
	\item[Step 2:]  The first step involves finding a good time slice, see \cref{figdeg1} for the geometry.  In particular, we have the following result:
	\begin{claim}\label{claim6.12}
		Let  $\bar{t} \in (-\bar{A}\varrho^{sp},0)$ be any time level and suppose \cref{secondaltdeg2rewrite} holds for this time level $\bar{t}$, 
		where we have denoted $\mcq_{c_o\rho}^{\tht}(\bar{t}) = B_{c_o\varrho} \times (\bar{t}-\tht(c_o\varrho)^{sp},\bar{t})$. Then there exists $t^{\ast} \in (\bar{t} - \tht (c_o\varrho)^{sp}, \bar{t} - \tfrac12 \nu_4\tht (c_o\varrho)^{sp} )$ such that the following is satisfied:
		\begin{equation*}
			\left|\left\{v(\cdot,t^{\ast})\leq \bsmu_v^+-\tfrac14 \bsom_v\right\}\cap B_{c_o\rho}\right|\geq \tfrac12\nu_4\left|B_{c_o\rho}\right|.
		\end{equation*}
	\end{claim}
\begin{proof}
	The proof is by contradiction, suppose the claim is false, then we have for all $t^{\ast} \in (\bar{t} - \tht (c_o\varrho)^{sp}, \bar{t} - \tfrac12 \nu_4\tht (c_o\varrho)^{sp} )$, we would have \begin{equation}\label{eq6.26}
		\left|\left\{v(\cdot,t^{\ast})\leq \bsmu_v^+-\tfrac14 \bsom_v\right\}\cap B_{c_o\rho}\right|< \tfrac12\nu_4\left|B_{c_o\rho}\right|.
	\end{equation}
Denoting the time interval $T:= (\bar{t} - \tht (c_o\varrho)^{sp}, \bar{t} - \tfrac12 \nu_4\tht (c_o\varrho)^{sp} )$,  we have the following sequence of estimates:
\begin{equation*}
	\begin{array}{rcl}
	\left|\left\{v\leq \bsmu_v^+-\tfrac14 \bsom_v\right\}\cap \mcq_{c_o\rho}^{\tht}(\bar{t})\right|	& = & \int_{\bar{t} - \tht (c_o\varrho)^sp}^{\bar{t} - \tfrac12 \nu_4\tht (c_o\varrho)^sp} \left|\left\{v(\cdot,t^{\ast})\leq \bsmu_v^+-\tfrac14 \bsom_v\right\}\cap B_{c_o\rho}\right| \,ds\\
	&& + \int_{\bar{t} - \tfrac12 \nu_4\tht (c_o\varrho)^sp}^{\bar{t}} \left|\left\{v(\cdot,t^{\ast})\leq \bsmu_v^+-\tfrac14 \bsom_v\right\}\cap B_{c_o\rho}\right| \,ds\\
	& \overset{\cref{eq6.26}}{<}& \tfrac12 \nu_4|B_{c_o\varrho}| \tht (c_o\varrho)^{sp} (1-\tfrac12 \nu_4) + \tfrac12 \nu_4|B_{c_o\varrho}| \tht (c_o\varrho)^{sp}\\
	& < & \tfrac12 \nu_4|\mcq_{c_o\varrho}^{\tht}(\bar{t})|,
	\end{array}
\end{equation*}
which is a contradiction to \cref{secondaltdeg2rewrite}. 

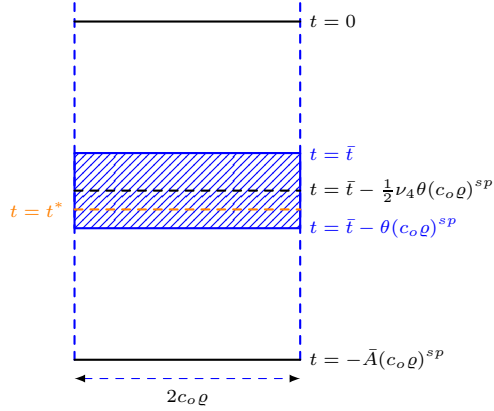
\begin{figure}[ht]
	\begin{center}
		\begin{tikzpicture}[line cap=round,line join=round,>=latex,scale=0.5]
			\coordinate  (O) at (0,0);
			\draw[thick, draw=blue, pattern color=blue,pattern=north east lines, opacity=0.3] (-3,0) rectangle (3,2);
			
			\node  [anchor=west] at (3,2) {\scriptsize \textcolor{blue}{$t=\bar{t}$}};
			\node  [anchor=west] at (3,0) {\scriptsize \textcolor{blue}{$t=\bar{t}-\tht(c_o\varrho)^{sp}$}};
			\draw[thick, draw=blue, dashed] (-3,-3.5) -- (-3,6);
			\draw[thick, draw=blue, dashed] (3,-3.5) -- (3,6);
			
			\draw[thick, draw=black, dashed] (-3,1) -- (3,1);
			\node  [anchor=west] at (3,1) {\scriptsize \textcolor{black}{$t=\bar{t}-\tfrac12\nu_4\tht(c_o\varrho)^{sp}$}};
			
			\draw[thick, draw=orange, dashed] (-3,0.5) -- (3,0.5);
			\node  [anchor=east] at (-3,0.5) {\scriptsize \textcolor{orange}{$t=t^{\ast}$}};
			\draw[draw=black,thick] (-3,-3.5) -- (3,-3.5);
			
			\draw[draw=black,thick] (-3,5.5) -- (3,5.5);
			\node  [anchor=west] at (3,5.5) {\scriptsize \textcolor{black}{$t=0$}};
			\draw[draw=blue, dashed, <->] (-3,-4) -- (3,-4);
			\node  at (0,-4.5) {\scriptsize $2c_o\varrho$};
			\node [anchor=west] at (3,-3.5) {\scriptsize $t=- \bar{A}(c_o\varrho)^{sp}$};
		\end{tikzpicture}
	\end{center}
		\caption{ Finding a time slice}
	\label{figdeg1}
\end{figure}
\end{proof}
	\descitemnormal{Step 3:}{step3} In this step, we expand measure information forward in time on one of the intrinsic cylinder, see \cref{figdeg2} for the geometry. 
		We  apply \cref{vwLemma5.5} with the choice $\alpha_1 = \tfrac12 \nu_4$ and  $\varepsilon_1=\bar\varepsilon_1  \in (0,\tfrac14)$ (with $\bar\varepsilon_1$ to be chosen in  \cref{degdefep1}) and choose $c_o = c_o(\datanb{,\bar\varepsilon_1, \nu_4,\tau,\mreta})$ small enough such that the tail alternative is satisfied, noting that the measure hypothesis of \cref{vwLemma5.5} holds from \cref{claim6.12}. From this, we  get the existence of two constants $\varepsilon_4 = \varepsilon_4(\datanb{,\nu_4,\mreta,\tau})$ and $\delta_4 = \delta_4(\datanb{,\nu_4,\mreta,\tau})$ such that the following conclusion follows:
		\begin{equation}\label{defdeg6.28}
		|v(\cdot,t)) \leq \bsmu_v^{+} -\varepsilon_4\bar\varepsilon_1\bsom_v\} \cap B_{c_o\varrho}| \geq \tfrac12 \alpha_1 |B_{c_o\varrho}| \txt{for all} t \in  (t^{\ast},t^{\ast} + \de_4(\bar\varepsilon_1\bsom_v)^{2-p}(c_o\varrho)^{sp}),
		\end{equation}
		where $t^{\ast}$ is as obtained in \cref{claim6.12}. 
		
		Now we choose $\bar\varepsilon_1$ to satisfy:
		\begin{equation}\label{degdefep1}
		\de_4(\bar\varepsilon_1\bsom_v)^{2-p}(c_o\varrho)^{sp} \geq  \tht (c_o \varrho)^{sp} \overset{\cref{def_tht}}{=}  (\tfrac14 \bsom_v)^{2-p} (c_o \varrho)^{sp} \quad  \Longrightarrow \quad \bar\varepsilon_1  \leq \tfrac14 \delta_4^{\frac{1}{p-2}}.
		\end{equation}
	In this way, the measure information can be propagated to the top of the intrinsic cylinder. In particular, \cref{defdeg6.28} holds in the interval $(t^{\ast},\bar{t})$ where $\bar{t}$ and $t^{\ast}$ are as obtained in \cref{claim6.12}.
	
	Since \descrefnormal{alt2}{DegAlt-$\II$} holds, we see that \cref{defdeg6.28} holds in the range
	\begin{equation*}
		|v(\cdot,t)) \leq \bsmu_v^{+} -\varepsilon_4\bar\varepsilon_1\bsom_v\} \cap B_{c_o\varrho}| \geq \tfrac12 \alpha_1 |B_{c_o\varrho}| \quad \text{holds} \,\forall \, t \in  (-\bar{A}(c_o\varrho)^{sp} +  \tht(c_o\varrho)^{sp},0),
	\end{equation*}
	where $\tht$ is defined  from \cref{def_tht}.
\begin{figure}[ht]
	\begin{center}
		\begin{tikzpicture}[line cap=round,line join=round,>=latex,scale=0.5]
			\coordinate  (O) at (0,0);
			\draw[thick, draw=blue, pattern color=blue,pattern=north east lines, opacity=0.3] (-3,0) rectangle (3,2);
			
			\draw[ draw=orange,pattern=north west lines, pattern color=orange, opacity=1] (-3,0.5) rectangle (3,2);
			\node  [anchor=west] at (3,2) {\scriptsize \textcolor{blue}{$t=\bar{t}$}};
			\node  [anchor=west] at (3,0) {\scriptsize \textcolor{blue}{$t=\bar{t}-\tht(c_o\varrho)^{sp}$}};
			\draw[thick, draw=blue, dashed] (-3,-3.5) -- (-3,6);
			\draw[thick, draw=blue, dashed] (3,-3.5) -- (3,6);
			
			\draw[thick, draw=black, dashed] (-3,1) -- (3,1);
			\node  [anchor=west] at (3,1) {\scriptsize \textcolor{black}{$t=\bar{t}-\tfrac12\nu_4\tht(c_o\varrho)^{sp}$}};
			
			\draw[draw=orange, dashed] (-3,0.5) -- (3,0.5);
			\node  [anchor=east] at (-3,0.5) {\scriptsize \textcolor{orange}{$t=t^{\ast}$}};
			\draw[draw=black,thick] (-3,-3.5) -- (3,-3.5);
			
			\draw[draw=black,thick] (-3,5.5) -- (3,5.5);
			\node  [anchor=west] at (3,5.5) {\scriptsize \textcolor{black}{$t=0$}};
			\draw[draw=blue, dashed, <->] (-3,-4) -- (3,-4);
			\node  at (0,-4.5) {\scriptsize $2c_o\varrho$};
			\node [anchor=west] at (3,-3.5) {\scriptsize $t=- \bar{A}(c_o\varrho)^{sp}$};
		\end{tikzpicture}
	\qquad 
	\begin{tikzpicture}[line cap=round,line join=round,>=latex,scale=0.5]
		\coordinate  (O) at (0,0);
		\draw[thick, draw=orange, pattern color=orange,pattern=north east lines, opacity=0.3] (-3,-2.5) rectangle (3,5.5);
		
		\node  [anchor=west] at (3,-2.5) {\scriptsize \textcolor{blue}{$t=-\bar{A} c_o\varrho)^{sp}\bar{t}+\tht(c_o\varrho)^{sp}$}};
		\draw[thick, draw=blue, dashed] (-3,-3.5) -- (-3,6);
		\draw[thick, draw=blue, dashed] (3,-3.5) -- (3,6);
		
%
		\draw[draw=black,thick] (-3,-3.5) -- (3,-3.5);
		
		\draw[draw=black,thick] (-3,5.5) -- (3,5.5);
		\node  [anchor=west] at (3,5.5) {\scriptsize \textcolor{black}{$t=0$}};
		\draw[draw=blue, dashed, <->] (-3,-4) -- (3,-4);
		\node  at (0,-4.5) {\scriptsize $2c_o\varrho$};
		\node [anchor=west] at (3,-3.5) {\scriptsize $t=- \bar{A}(c_o\varrho)^{sp}$};
	\end{tikzpicture}
	\end{center}
		\caption{Expansion of measure information when \descrefnormal{alt2}{DegAlt-$\II$} holds}
	\label{figdeg2}
\end{figure}
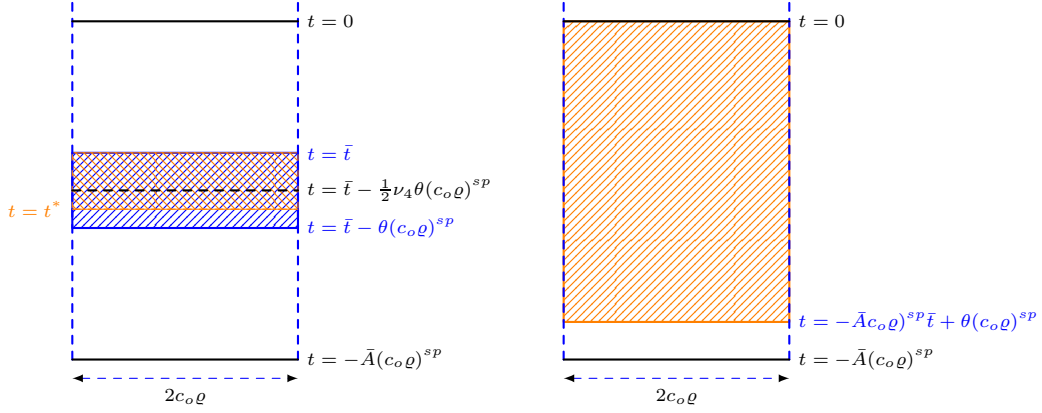
	\descitemnormal{Step 4:}{step4de} With $\delta_1 = 1$ and $\varepsilon_1 = \varepsilon_4\bar\varepsilon_1$, we will apply \cref{lemma5.4} to obtain two constants $\nu_5= \nu_3(\datanb{})$ and $c_o = c_o(\tau,\varepsilon_4,\bar\varepsilon_1,\mreta)$. In \cref{lemma5.4}, we will have to ensure that
	\begin{equation}\label{defdega1}
	\lbr \frac{\bsom_v}{\mathbf{a}}\rbr^{2-p} (c_o\varrho)^{sp} \geq (\varepsilon_4\bar\varepsilon_1\bsom_v)^{2-p} (\tfrac{c_o\varrho}{2})^{sp} \quad \Longrightarrow \quad \mathbf{a} \geq \frac{1}{2^{\frac{sp}{p-2}}\varepsilon_4\bar\varepsilon_1}.
	\end{equation}  Note that $c_o = c_o(\varepsilon_4,\bar\varepsilon_1)$ will be further restricted  to satisfy the tail alternative holds.
	\descitemnormal{Step 5:}{step5de} With $\nu_5$ obtained in \descrefnormal{step4de}{Step 4}, we will apply \cref{vwLemma5.6} with $\bar{\delta_2} = 1$, $\bar{\varepsilon_2} = \bar\varepsilon_1\varepsilon_4$ and $\alpha_2 = \tfrac18 \nu_4$ followed by choosing $\bar{\sigma_2}$ to satisfy
	\[
	\bsc \frac{\bar{\sigma_2}^{p-1}}{\alpha_2} \leq \nu_5,
	\]
	where $\bsc$ is the constant in the conclusion of \cref{vwLemma5.6}. Again for this step to hold, we will need 
	\begin{equation}\label{defdega2}
	\lbr \frac{\bsom_v}{\mathbf{a}}\rbr^{2-p} (c_o\varrho)^{sp}  \geq (\bar{\sigma_2}\bar\varepsilon_1\varepsilon_4\bsom_v)^{2-p} (c_o\varrho)^{sp} \quad \Longrightarrow\quad \mathbf{a} \geq \frac{1}{\bar{\sigma_2}\bar\varepsilon_1\varepsilon_4}.
		\end{equation}
	Again we further restrict $c_o = c_o(\bar\varepsilon_1,\varepsilon_4,\bar\sigma_2)$ such that the tail condition in \cref{vwLemma5.6} is satisfied. 
	\item[Step 6:] From \descrefnormal{step4de}{Step 4} and \descrefnormal{step5de}{Step 5}, we obtain the following conclusion:
	\[
	v \leq \bsmu^+_v - \tfrac12 \bar{\sigma_2}\bar\varepsilon_1\varepsilon_4 \bsom_v\qquad \text{in} \quad B_{\frac14c_o\varrho}\times (-(\bar{\sigma_2}\bar\varepsilon_1\varepsilon_4 \bsom_v)^{p-2}(\tfrac14c_o\varrho)^{sp},0).
	\]
\end{description}
This completes the proof of the reduction of oscillation when \descrefnormal{alt2}{DegAlt-$\II$} holds.

\subsection{Collection of dependence of constants}\label{degconstantsscaled}
Combining \cref{redoscfirstp>2scaled} and \cref{redoscsecondp>2scaled}, we have the following proposition:
\begin{proposition}\label{propdegaway}
	Let $\mcq_{c_o\varrho}^{\bar{A}} := B_{c_o\varrho} \times (-\lbr\tfrac{\bsom_v}{\mathbf{a}} \rbr^{2-p}(c_o\varrho)^{sp},0)$ be a given cylinder,  then there exists a constant $c_o = c_o(\datanb{,\tau,\mreta,\tfrac{1}{\mathbf{a}}})$ such that one of the two conclusions hold:
	\begin{description}
		\descitemnormal{Conclusion 1:}{conc1deg} We have
		\begin{equation*}
			\essosc_{\mcq} v \leq \lbr 1-\tfrac12 \varepsilon_1\rbr \bsom_v \quad \text{where} \quad \mcq:= B_{\frac{c_o\varrho}4} \times \left(-\bar\nu_3(\varepsilon_1\bsom_v)^{2-p}\lbr{\tfrac14c_o\varrho}\rbr^{sp},0\right],
		\end{equation*}
	and 
	\[
	c_o^{\frac{sp}{p-1}}\tail((u-\bsmu^{-})_{-};0,c_o\varrho,(-\lbr\tfrac{\bsom_v}{\mathbf{a}} \rbr^{2-p}(c_o\varrho)^{sp},0)) \leq \min\{\tfrac14,\varepsilon_1,\bar\sigma_2\varepsilon_2\} \bsom.
	\]
	 \descitemnormal{Conclusion 2:}{conc2deg} Or
	 \[
	 v \leq \bsmu^+_v - \tfrac12 \bar{\sigma_2}\bar\varepsilon_1\varepsilon_4 \bsom_v\qquad \text{in} \quad B_{\frac14c_o\varrho}\times (-(\bar{\sigma_2}\bar\varepsilon_1\varepsilon_4 \bsom_v)^{p-2}(\tfrac14c_o\varrho)^{sp},0).
	 \]
	 and 
	 \[
	 c_o^{\frac{sp}{p-1}}\tail((u-\bsmu^{+})_{+};0,c_o\varrho,(-\lbr\tfrac{\bsom_v}{\mathbf{a}} \rbr^{2-p}(c_o\varrho)^{sp},0)) \leq \min\{\tfrac14,\bar\varepsilon_1,\bar\varepsilon_1\varepsilon_4, \bar\sigma_2\bar\varepsilon_2\} \bsom.
	 \]
	\end{description}
We can combine both conclusions to have
\[
\essosc_{\mcq}v \leq \lbr1-\eta_d \rbr \bsom_v\qquad \text{where} \quad \mcq:= B_{\frac14c_o\varrho}\times (-\mathbf{d}\bsom_v^{2-p}(\tfrac14c_o\varrho)^{sp},0),
\]
where $\mathbf{d} = \min\{\bar\nu_3(\varepsilon_1)^{2-p}, (\bar\sigma_2\bar\varepsilon_1\varepsilon_4)^{2-p}\}$, $\eta_d= \min\{\tfrac12 \bar{\sigma_2}\bar\varepsilon_1\varepsilon_4, \tfrac12 \varepsilon_1\}$
and 
\[
c_o^{\frac{sp}{p-1}}\tail((u-\bsmu^{\pm})_{\pm};0,c_o\varrho,(-\lbr\tfrac{\bsom_v}{\mathbf{a}} \rbr^{2-p}(c_o\varrho)^{sp},0)) \leq \mathbf{d}_t \bsom,
\]
where $\mathbf{d}_t:=\min\{\tfrac14,\bar\varepsilon_1,\bar\varepsilon_1\varepsilon_4, \bar\sigma_2\bar\varepsilon_2\}$ and $\bsom,\bsom_v$ are related by $\bsom_v = (p-1)\lbr \tfrac{1+\tau}{\tau}\rbr^{p-1}\tfrac{\bsom}{\bsmu^-}$. 
\end{proposition}

\noindent\begin{minipage}{0.49\textwidth}
	\begin{center}
		\begin{tabular}{|c|c|c|}
			\hline
			Symbol & location & dependence\\
			\hline\hline
			$c_o$ & \descrefnormal{step1}{Step 1} of \cref{redoscfirstp>2scaled} & $n,p,s,\tau,\mreta$\\
			& \descrefnormal{step2}{Step 2} of \cref{redoscfirstp>2scaled} & $\tau,\varepsilon_1,\mreta$\\
			 & \descrefnormal{step3}{Step 3} of \cref{redoscfirstp>2scaled} & $\bar\varepsilon_1,\nu_4,\mreta,\tau$\\
			 & \descrefnormal{step4de}{Step 4} of \cref{redoscfirstp>2scaled} & $\varepsilon_4, \bar\varepsilon_1$ \\
			 &\descrefnormal{step5de}{Step 5} of \cref{redoscfirstp>2scaled}&$\bar\varepsilon_1,\varepsilon_4,\bar\sigma_2$\\\hline
			 $\mathbf{a} $& \cref{defdega1}& $\tfrac{1}{\bar\varepsilon_1},\tfrac{1}{\varepsilon_4}$ \\
			 &\cref{defdega2}& $\tfrac{1}{\bar\sigma_2}, \tfrac{1}{\bar\varepsilon_1},\tfrac{1}{\varepsilon_4}$  \\\hline
			 $\alpha_1$ &\descrefnormal{step3}{Step 3} of \cref{redoscfirstp>2scaled}  & $:=\tfrac12 \nu_4$\\\hline
		\end{tabular}
	\end{center}
\end{minipage}
\begin{minipage}{0.49\textwidth}
	\begin{center}
		\begin{tabular}{|c|c|c|}
			\hline
			Symbol & location & dependence\\
			\hline\hline
			$ \nu_4$ & \descrefnormal{step1}{Step 1} of \cref{redoscfirstp>2scaled} & $n,p,s,\La$ \\\hline
			$\varepsilon_1$ & \cref{eq6.21} & $n,p,s,\La, \tfrac{1}{\mathbf{a}},\bar\nu_3$\\\hline
			$\bar\nu_3$ &\descrefnormal{step2}{Step 2} of \cref{redoscfirstp>2scaled} & $n,p,s,\La$\\\hline
			$\bar\varepsilon_1$ & \cref{degdefep1} & $n,p,s,\La,\de_4$ \\\hline
			$\de_4$ & \cref{defdeg6.28} & $n,p,s,\La,\nu_4,\mreta,\tau$\\\hline
			$\varepsilon_4$ & \descrefnormal{step3}{Step 3} of \cref{redoscfirstp>2scaled} & $\nu_4,\mreta,\tau$\\\hline
			$\nu_5$ & \descrefnormal{step4de}{Step 4} of \cref{redoscfirstp>2scaled} & $\datanb{}$\\\hline
			$\bar\sigma_2$ & \descrefnormal{step5de}{Step 5} of \cref{redoscfirstp>2scaled} & $\nu_5,\nu_4$\\\hline
		\end{tabular}
	\end{center}
\end{minipage}


\section{Reduction of oscillation for a scaled parabolic fractional \texorpdfstring{$p$}.-Laplace type equations (away from zero case) - singular case}

In this case, we assume $p \leq 2$ and obtain reduction of oscillation. The proof follows the calculations of \cite[Section 4]{liaoHolderRegularityParabolic2022}. 

\begin{remark}
	We note that all the calculations from \cref{section6} holds with $\tau = 1$ in the case $p<2$ with obvious modifications. We will not write these standard modifications for the sake of brevity. 
\end{remark}
\begin{proposition}
	Let $v$ be as in \cref{vwdef5.2} and  \cref{lemma6.3}. Suppose for two  constants $\alpha_5 \in (0,1)$ and $\varepsilon_5 \in (0,1)$, the following assumption holds:
	\[
	|\{\pm (\bsmu^{\pm}_v - v(\cdot,t_o))\geq \varepsilon_5 \bsom_v\}| \geq \alpha_5 |B_{c_o\varrho}|.
	\]
	Then there exists constants $\delta_5 = \delta_5(\datanb{,\alpha_5})$, $\eta_5 = \eta_5(\datanb{,\alpha_5})$ and $c_o = c_o(\eta_5,\varepsilon_5)$ such that the following conclusion follows:
	\[
	c_o^{\frac{sp}{p-1}} \tailp((u - \bsmu^{\pm})_{\pm}; \mcq) \leq \eta_5\varepsilon_5\bsom,
	\]
	and 
	\[
	\pm(\bsmu^{\pm}_v - v) \geq \eta_5\varepsilon_5\bsom_v \quad \text{on} \quad B_{2c_o\varrho} \times (t_o + \tfrac12 \delta_5 (\varepsilon_5\bsom_v)^{2-p}(c_o\varrho)^{sp}, t_o + \delta_5 (\varepsilon_5\bsom_v)^{2-p}(c_o\varrho)^{sp})
		\]
		provided $B_{2c_o\varrho} \times (t_o + \tfrac12 \delta_5 (\varepsilon_5\bsom_v)^{2-p}(c_o\varrho)^{sp}, t_o + \delta_5 (\varepsilon_5\bsom_v)^{2-p}(c_o\varrho)^{sp}) \subset \mcq$. 
		Moreover, we have $\delta_5 \approx \alpha_5^{p-1}$ and $\eta_5 \approx \alpha_5^q$ for some $q>1$ depending only on data.
\end{proposition}
\begin{proof}
	The tail estimate follows analogous to those in the proof of \cref{lemma5.4} by choosing $c_o$ sufficiently small. So, we only need to show the pointwise local estimate for $v$ which we do as follows:
	\begin{description}
		\descitemnormal{Step 1:}{step1sing} We apply \cref{vwLemma5.5} with $\alpha_1 = \alpha_5$ and $\varepsilon_1 = \varepsilon_5$ to get 
		\[
		|\{\pm (\bsmu^{\pm}_v - v(\cdot,t))\geq \bar\varepsilon\varepsilon_5 \bsom_v\}| \geq \tfrac12\alpha_5 |B_{c_o\varrho}| \txt{for a.e} t \in (t_o,t_o+\delta_1(\varepsilon_5\bsom_v)^{2-p}(c_o\varrho)^{sp}), 
		\]
		where $\bar\varepsilon$ and $\delta_1$ are as obtained in \cref{vwLemma5.5} in this situation.
		\descitemnormal{Step 2:}{step2sing}In this step, we want to apply \cref{vwLemma5.6} which requires verifying the measure hypothesis is satisfied. With $\delta_1$ as obtained in \descrefnormal{step1sing}{Step 1}, we need to ensure $\de_1 (\varepsilon_5\bsom_v)^{2-p}(c_o\varrho)^{sp} \geq \bar\delta_2(\bar\sigma_2\bar\varepsilon\varepsilon_5\bsom_v)^{2-p}(c_o\varrho)^{sp}$ which holds provided we choose $\bar\delta_2=\delta_1$ and  $\bar\sigma_2\bar\varepsilon \leq 1$. 
		 
		 In particular, with this choice, \descrefnormal{step1sing}{Step 1} gives
		 \[
		 |\{\pm (\bsmu^{\pm}_v - v(\cdot,t))\geq \bar\varepsilon\varepsilon_5 \bsom_v\}| \geq \tfrac12\alpha_5 |B_{c_o\varrho}| \txt{for a.e} t \in (t_o,t_o+\delta_1(\bar\sigma_2\bar\varepsilon\varepsilon_5\bsom_v)^{2-p}(c_o\varrho)^{sp}). 
		 \]
		 
		 Now apply \cref{vwLemma5.6}, we get the following conclusion for any $\bar\sigma_2 \in (0,1)$:
		 \[
		 |\{\pm (\bsmu_v^{\pm}-v) \leq \bar\sigma_2\bar\varepsilon\varepsilon_5\bsom_v\} \cap \mcq| \leq \bsc_{\data{,\tau,\mreta}}\frac{\bar\sigma_2^{p-1}}{\delta_1\alpha_5}  |\mcq|,
		 \]
		 where $\mcq = B_{c_o\varrho} \times (t_o,t_o+\delta_1(\bar\sigma_2\bar\varepsilon\varepsilon_5\bsom_v)^{2-p}(c_o\varrho)^{sp})$.
		 \item[Step 3:] Let us take $\varepsilon_1 = \bar\sigma_2\bar\varepsilon\varepsilon_5$ and $\de_1$ from \descrefnormal{step2sing}{Step 2} in \cref{lemma5.4} to obtain $\nu_3 = \nu_3(\datanb{,\delta_1})$. Now we make the choice $\bar\sigma_2 \in (0,1)$ such that $\bsc_{\data{,\tau,\mreta}}\frac{\bar\sigma_2^{p-1}}{\delta_1\alpha_5} \leq \nu_3$, noting that $\nu_3$ is independent of $\bar\sigma_2$. Now \cref{lemma5.4} becomes applicable which leads to the following conclusion:
		 \[
		 \pm(\bsmu^{\pm}_v-v) \geq \tfrac12 \bar\sigma_2\bar\varepsilon\varepsilon_5 \txt{on} B_{\frac12c_o\varrho} \times (t_o + \tfrac12 \de_1(\bar\sigma_2\bar\varepsilon\varepsilon_5\bsom_v)^{2-p}(\tfrac12c_o\varrho)^{sp}), t_o + \de_1(\bar\sigma_2\bar\varepsilon\varepsilon_5\bsom_v)^{2-p}(\tfrac12c_o\varrho)^{sp}).
		 \]
		 Let us define $\eta_5 := \tfrac12 \bar\sigma_2\bar\varepsilon$ and $\de_5:= \tfrac{\delta_1 (\bar\sigma_2\bar\varepsilon)^{2-p}}{2^{sp}}$ which completes the proof. 
	\end{description}
\end{proof}

\section{H\"older regularity for nonlocal doubly nonlinear equations - degenerate case}\label{section10}
For some fixed $R \leq 1$, let us consider the parabolic cylinder 
\begin{equation*}
	\mathbf{Q}_o := B_{8R}(x_o)\times\left(t_o-(8R)^{sp},t_o+(8R)^{sp}\right].
\end{equation*}
With this choice, let us  introduce numbers $\bsmu^{\pm}$ and $\bsom$ satisfying
\begin{equation*}
	\bsmu^+\geq  \esssup_{\mathbf{Q}_o}u, \qquad \bsmu^-\leq \essinf_{\mathbf{Q}_o} u \qquad \text{ and } 
	\qquad \bsom \geq \bsmu^+ - \bsmu^-.	
\end{equation*}
We also assume $(x_o,t_o) = (0,0)$ without loss of generality. 

\begin{definition}
	There exists $\mathbf{C}_o \gg 2$ large depending only on data and we define $R_i := \mathbf{C}_o^{-i}R$ for $i = 1,2,\ldots$. With this, we denote $\mcq_i := B_{R_i} \times (-R_i^{sp},0) =: B_i \times I_i$. It is easy to see that $\mcq_o = B_{R} \times (-R^{sp},0) \subset \mathbf{Q}_o$. 

	Furthermore, for some fixed $j_o \geq 1$,  let us define the constant 
	\[
	L:= 2 \mathbf{C}_o^{\frac{sp}{p-1}j_o} \|u\|_{L^{\infty}(B_{8R}\times (-(8R)^{sp},0))} + \tail(|u|,8R,0,(-(8R)^{sp},0)).
	\]
%
%
	Let us also take $m_o := -\tfrac12 L$ and $M_o := \tfrac12 L$. 
\end{definition}

The main proposition we will prove is the following:
\begin{proposition}\label{prop10.1}
	There exists two sequence of numbers nondecreasing $\{m_i\}_{i=0}^{\infty}$  and nonincreasing $\{M_i\}_{i=0}^{\infty}$ such that for any $i \in \{0,1,2,3\ldots\}$, there holds
	\[
	m_i \leq u \leq M_i \txt{on} \mcq_i = B_i \times I_i,
	\]
	where $m_i$ and $M_i$ further satisfy
	\begin{equation*}\label{eq10.2deg}
		\def\arraystretch{1.2}
		\begin{array}{rcl}
			M_i - m_i & = & \mathbf{C}_o^{-\beta i} L,\\
			\beta & \leq  & \max\left\{ \log_{\mathbf{C}_o} \tfrac{1}{1-\bar\eta}, \log_{\mathbf{C}_o} \tfrac{1}{1-\eta_1}, \log_{\mathbf{C}_o} \tfrac{1}{1-\bar\eta_2}, \log_{\mathbf{C}_o} \tfrac{1}{1-\frac{\varepsilon}{2^{j_*+1}}}\right\} \txt{from \descrefnormal{step4covdeg}{Step 4}}. \\
			\beta & \leq & \min\{\log_{\mathbf{C}_o}\frac{1}{1-\frac12 \varepsilon_1}, \log_{\mathbf{C}_o}\frac{1}{1-\frac12\bar\sigma_2\bar\varepsilon_1\varepsilon_4}\} \txt{from \descrefnormal{step11deg}{Step 11}}\\
			\kappa & \leq & \leq \min\{\bar\eta,\eta_1,\bar\eta_2,\tfrac{\varepsilon}{2^{j_*+1}}\} \txt{from \descrefnormal{step6deg}{Step 6}}\\
			C_1\int_{\mathbf{C}_o}^1  \frac{\lbr b^{\beta}-1\rbr^{p-1}}{b^{1+sp}}\,db   & \leq &  \tfrac12 \kappa^{p-1}\\
			j_o &=& \lceil\frac{1}{sp-\beta(p-1)}\log_{\mathbf{C}_o}\frac{2C_3}{\kappa^{p-1}}\rceil \txt{from \descrefnormal{step5deg}{Step 5}}.
		\end{array}
	\end{equation*}
In particular, this implies $\essosc_{\mcq_i} u \leq \mathbf{C}_o^{\beta i}L$. 
\end{proposition}

\subsection{Proof of \texorpdfstring{\cref{prop10.1}}.}

\begin{description}[leftmargin=*]
	\item[Step 1:] For $i = 1,2,\ldots, j_o$, let us define $m_i = -\tfrac12 \mathbf{C}_o^{-\beta i} L$ and $M_i = \tfrac12 \mathbf{C}_o^{-\beta i} L$. Then we have the following estimate:
	\[
	\|u\|_{L^{\infty}(\mcq_i)}\leq \|u\|_{L^{\infty}(\mcq_o)} \leq \|u\|_{L^{\infty}(\mathbf{Q}_o)} = \tfrac12 \mathbf{C}_o^{-\beta i} \lbr 2 \mathbf{C}_o^{\frac{sp}{p-1}}\|u\|_{L^{\infty}(\mathbf{Q}_o)} \rbr \lbr \frac{\mathbf{C}_o^{\beta i}}{\mathbf{C}_o^{\frac{sp}{p-1}j_o}}\rbr \leq M_i.
	\]
	In particular, this says 
	\[
	m_i \leq u \leq M_i \txt{on} \mcq_i. 
	\]
	\item[Step 2:] We inductively define $M_i$ and $m_i$ and in this regard, assume that we have defined $m_i$ and $M_i$ for $i = j_o,j_o+1, \ldots , j$, then we will define it for $i=j+1$.  With the notation as used in \cref{section5}, let us denote $\bsmu^- = m_j$, $\bsmu^+ = M_j$ and $\bsom = M_j - m_j =\mathbf{C}_o^{-\beta j}L$.  Based on this, either \tlcref{Eq:Hp-main1} or \tlcref{Eq:Hp-main2} holds, with $\tau$ fixed as \cref{subsection5.2.6}. 
	
	\item[Step 3:] Let us consider the cylinder $\mcq_{c_o\varrho} := B_{c_o\varrho} \times (-(A-1)(c_o\varrho)^{sp},0) \subset \mcq_j = B_{R_j} \times (-R_j^{sp},0)$. This inclusion is possible provided we choose $ \varrho = \mathbf{C}_o^{-j}R = R_j$ and $c_o \leq \min\{A^{-\frac{1}{sp}},1\}$ (note that $\mathbf{C}_o$ is yet to be chosen).  
	
	\descitemnormal{Step 4:}{step4covdeg} Let us reduce the oscillation near zero, i.e., we assume \tlcref{Eq:Hp-main1} holds with $\tau$ fixed as \cref{subsection5.2.6}. From \cref{nearzerodegen}, we obtain the reduction of oscillation, provided the required tail alternatives are satisfied. Assuming that the tail alternatives are satisfied, we take
	\begin{description}
		\item[Case $-\tau \bsom \leq \bsmu^- \leq \tau \bsom$:] In this case, from \cref{degclaim5.3}, we get
		\[
		u \geq \bsmu^- + \bar\eta \bsom \quad\mbox{a.e., in } \quad B_{\frac{c_o\varrho}2} \times \left( -\tfrac14 (c_o\varrho)^{sp},0\right].
		\]
		Let us take $m_{j+1} = M_j - \mathbf{C}_o^{-\beta (j+1)}L$ and $M_{j+1} = M_j$. Then, it is easy to see that $M_{j+1} - m_{j+1} = \mathbf{C}_o^{-\beta (j+1)}L$. Furthermore, $u \leq M_{j+1}$ trivially and to see $u \geq m_{j+1}$ on $B_{\frac{c_o\varrho}2} \times \left( -\tfrac14 (c_o\varrho)^{sp},0\right]$, we proceed as follows:
		\begin{equation*}
		u  \geq  m_j + \bar\eta\mathbf{C}_o^{-\beta j} L = m_j - M_j + M_j +	\bar\eta\mathbf{C}_o^{-\beta j} L
		 =  M_j - \mathbf{C}_o^{-\beta j} L (1-\bar\eta)
		 \geq    M_j - \mathbf{C}_o^{-\beta (j+1)} L,
	\end{equation*}
	provided we choose $\beta \leq \log_{\mathbf{C}_o} \tfrac{1}{1-\bar\eta}$ which is equivalent to $\mathbf{C}_o^{-\beta}\geq 1-\bar\eta$. 
		\item[Case $\bsmu^-\leq -\tau \bsom$:] In this case, from \cref{Lm:6:1}, we get
		\[
		u\ge\bsmu^-+\eta_1\bsom\quad\mbox{a.e., in}\quad  
		B_{\frac14 c_o\varrho}\times\left( - (\tfrac{c_o\varrho}{2})^{sp},0\right].
		\]
		Let us take $m_{j+1} = M_j - \mathbf{C}_o^{-\beta (j+1)}L$ and $M_{j+1} = M_j$. Then, it is easy to see the same calculations as in the case $-\tau \bsom \leq \bsmu^- \leq \tau \bsom$ follows and we need to further restrict $\beta \leq \log_{\mathbf{C}_o} \tfrac{1}{1-\eta_1}$ which is equivalent to $\mathbf{C}_o^{-\beta}\geq 1-\eta_1$.
		\item[Case $-\tau \bsom \leq \bsmu^+ \leq \tau \bsom$:] In this case, from \cref{deggclaim5.3}, we get
		\begin{equation*}
			u \leq \bsmu^+ - \bar\eta_2 \bsom \txt{a.e., in} B_{\frac12c_o\varrho} \times (- \tfrac{\nu}{4}(\tfrac{c_o\varrho}{2})^{sp}, 0].
		\end{equation*}
		Let us take $m_{j+1} = m_j$   and $M_{j+1} = m_j + \mathbf{C}_o^{-\beta (j+1)}L$. Then, it is easy to see that $M_{j+1} - m_{j+1} = \mathbf{C}_o^{-\beta (j+1)}L$. Furthermore, $u \geq m_{j+1}$ trivially and to see $u \leq M_{j+1}$ on $B_{\frac{c_o\varrho}2} \times \left( -\tfrac{\nu}4 (c_o\varrho)^{sp},0\right]$, we proceed as follows:
		\begin{equation*}
			u  \leq  M_j - \bar\eta_2\mathbf{C}_o^{-\beta j} L = M_j - m_j + m_j -	\bar\eta_2\mathbf{C}_o^{-\beta j} L
			=  m_j + \mathbf{C}_o^{-\beta j} L (1-\bar\eta_2)
			\leq    m_j + \mathbf{C}_o^{-\beta (j+1)} L,
		\end{equation*}
		provided we choose $\beta \leq \log_{\mathbf{C}_o} \tfrac{1}{1-\bar\eta_2}$ which is equivalent to $\mathbf{C}_o^{-\beta}\geq 1-\bar\eta_2$. 
		\item[Case $\tau \bsom \leq \bsmu^+  \leq 2\bsom$:] In this case, from \cref{sub5.2.6.a} and using $2^{j_*(p-2)}\geq 1$, we have
		\[
		\bsmu^{+}-u\geq \frac{\varepsilon\bsom}{2^{j_*+1}}\txt{a.e., in }  B_{\frac12c_o\varrho}\times (t_o-(\tfrac{c_o\varrho}{2})^{sp},t_o].
		\]
		Analogous to calculations in the case $-\tau \bsom \leq \bsmu^+ \leq \tau \bsom$, we  choose $m_{j+1} = m_i$   and $M_{j+1} = m_j+ \mathbf{C}_o^{-\beta (j+1)}L$ with $\beta \leq \log_{\mathbf{C}_o} \tfrac{1}{1-\frac{\varepsilon}{2^{j_*+1}}}$.
	\end{description}
\descitemnormal{Step 5:}{step5deg} In this step, we will prove the validity of the tail estimates needed to ensure \descrefnormal{step4covdeg}{Step 4} is applicable. Recalling $m_j = \bsmu^-$, $M_j = \bsmu^+$ and $m_i \leq u \leq M_i$ on $B_i \times I_i$ for $i \in \{0,1,2,\ldots, j\}$,  we have the following sequence of estimates:
\begin{equation}\label{deg10.3}
	\begin{array}{rcl}
		\tail((u-\bsmu^{\pm})_{\pm}, R_j, 0, (-R_j^{sp},0))^{p-1} & = &  R_j^{sp} \esssup_{I_j} \int_{\RR^n\setminus B_0} \frac{(u-\bsmu^{\pm})_{\pm}^{p-1}}{|x|^{n+sp}}\,dx \\
		&& + R_j^{sp} \esssup_{I_j}\int_{B_0\setminus B_j} \frac{(u-\bsmu^{\pm})_{\pm}^{p-1}}{|x|^{n+sp}}\,dx\\
		& \leq & R_j^{sp} \esssup_{I_j}  2^{p-1}\int_{\RR^n\setminus B_0} \frac{|u|^{p-1} + |\bsmu^{\pm|^{p-1}}}{|x|^{n+sp}}\,dx \\
		&& + R_j^{sp} \esssup_{I_j} \int_{B_0\setminus B_j} \frac{(u-\bsmu^{\pm})_{\pm}^{p-1}}{|x|^{n+sp}}\,dx\\
		& \leq & R_j^{sp} C_3 \frac{L^{p-1}}{R^{sp}} + R_j^{sp} \esssup_{I_j} \int_{B_0\setminus B_j} \frac{(u-\bsmu^{\pm})_{\pm}^{p-1}}{|x|^{n+sp}}\,dx,
	\end{array}
\end{equation}
where we used $|\bsmu^{\pm}| \leq 2\|u\|_{L^{\infty}(\mathbf{Q}_o)} \leq L$ and $\int_{\RR^n\setminus B_0} \frac{|u|^{p-1}}{|x|^{n+sp}}\,dx = \tfrac{\tail(|u|, R,0,(-R^{sp},0))^{p-1}}{R^{sp}}\leq \frac{L^{p-1}}{R^{sp}}$. 

We further estimate the last term appearing on the right hand side of \cref{deg10.3} as follows: For each  $x \in B_0 \setminus B_j$, then there exists $l \in 0,1,\ldots, j-1$ such that $\mathbf{C}_o^{-l}R \geq |x| \geq \mathbf{C}_o^{-{l+1}}R$ and $m_l \leq u \leq M_l$ which gives
\begin{equation}\label{deg10.4}\def\arraystretch{1.5}
	\begin{array}{rcl}
		(u - m_j)_- &\leq&  m_j - m_l = m_j - M_j + M_j - m_l = -\mathbf{C}_o^{-\beta j}L + M_j - m_l \leq -\mathbf{C}_o^{-\beta j}L + M_l - m_l \\
		&= & -\mathbf{C}_o^{-\beta j}L + \mathbf{C}_o^{-\beta l}L
		 \leq L \lbr \frac{|x|}{\mathbf{C}_oR}\rbr^{\beta}-\mathbf{C}_o^{-\beta j}L \\
		(u - M_j)_+ & \leq & M_l - M_j = M_l - m_j + m_j - M_j = M_l - m_j-\mathbf{C}_o^{-\beta j}L \leq M_l - m_l-\mathbf{C}_o^{-\beta j}L\\
		& = & \mathbf{C}_o^{-\beta l}L - \mathbf{C}_o^{-\beta j}L
		\leq L \lbr \frac{|x|}{\mathbf{C}_oR}\rbr^{\beta}-\mathbf{C}_o^{-\beta j}L \\
	\end{array}
\end{equation}
Thus we have the following sequence of estimates: 
\begin{equation}\label{deg10.5}
	\begin{array}{rcl}
	\esssup_{I_j} \int_{B_0\setminus B_j} \frac{(u-\bsmu^{\pm})_{\pm}^{p-1}}{|x|^{n+sp}}\,dx & \overset{\cref{deg10.4}}{\leq} & L^{p-1} \mathbf{C}_o^{-\beta j(p-1)}\esssup_{I_j} \int_{B_0\setminus B_j} \frac{\lbr\lbr \frac{|x|}{\mathbf{C}_o^{1-j}R}\rbr^{\beta}-1\rbr^{p-1}}{|x|^{n+sp}}\,dx\\
	& \leq & 	L^{p-1} \mathbf{C}_o^{-\beta j(p-1)}\esssup_{I_j} \int_{\RR^n\setminus B_{\mathbf{C}_o^{-j}R}} \frac{\lbr\lbr \frac{|x|}{\mathbf{C}_o^{1-j}R}\rbr^{\beta}-1\rbr^{p-1}}{|x|^{n+sp}}\,dx\\
	& = & 	L^{p-1} \mathbf{C}_o^{-\beta j(p-1)}\frac{1}{(\mathbf{C}_o^{1-j}R)^{sp}}\esssup_{I_j} \int_{\RR^n\setminus B_{\mathbf{C}_o}} \frac{\lbr |y|^{\beta}-1\rbr^{p-1}}{|y|^{n+sp}}\,dy\\
	& = & C_1 L^{p-1} \mathbf{C}_o^{-\beta j(p-1)}\frac{1}{(\mathbf{C}_o^{1-j}R)^{sp}} \int_{\mathbf{C}_o}^{\infty}  \frac{\lbr b^{\beta}-1\rbr^{p-1}}{b^{1+sp}}\,db
	\end{array}
\end{equation}
Combining \cref{deg10.5} and \cref{deg10.3}, for any $\kappa \in (0,1)$ we get
\begin{equation*}
	\begin{array}{rcl}
	\tail((u-\bsmu^{\pm})_{\pm}, R_j, 0, (-R_j^{sp},0))^{p-1} & \leq & R_j^{sp} C_3 \frac{L^{p-1}}{R^{sp}} + \frac{C_1 R_j^{sp} L^{p-1} \mathbf{C}_o^{-\beta j(p-1)}}{(\mathbf{C}_o^{1-j}R)^{sp}} \int_{\mathbf{C}_o}^{\infty}  \frac{\lbr b^{\beta}-1\rbr^{p-1}}{b^{1+sp}}\,db\\
	& \leq & \mathbf{C}_o^{-\beta j(p-1)} L^{p-1} \lbr[[] C_3 \mathbf{C}_o^{-j(sp-\beta(p-1))}+ \frac{C_1}{\mathbf{C}_o^{sp}}\int_{\mathbf{C}_o}^{\infty}  \frac{\lbr b^{\beta}-1\rbr^{p-1}}{b^{1+sp}}\,db \rbr[]]\\
	& \leq & \mathbf{C}_o^{-\beta j(p-1)} L^{p-1} \kappa^{p-1},
	\end{array}
\end{equation*}
provided we choose $\beta \in (0,1)$ small (independent of $\mathbf{C}_o \geq 1$) such that $$\frac{C_1}{\mathbf{C}_o^{sp}}\int_{\mathbf{C}_o}^{\infty}  \frac{\lbr b^{\beta}-1\rbr^{p-1}}{b^{1+sp}}\,db  \leq {C_1}\int_{1}^{\infty}  \frac{\lbr b^{\beta}-1\rbr^{p-1}}{b^{1+sp}}\,db \leq \tfrac12 \kappa^{p-1},$$ and $j_o =\lceil \frac{1}{sp-\beta(p-1)}\log_{\mathbf{C}_o}\frac{2C_3}{\kappa^{p-1}}\rceil$ and noting that $j \geq j_o$. 
\descitemnormal{Step 6:}{step6deg} In this step, we will choose $\kappa$ to ensure \descrefnormal{step4covdeg}{Step 4} is applicable to obtain the required tail estimates.  We will take $\kappa \leq \min\{\bar\eta,\eta_1,\bar\eta_2,\tfrac{\varepsilon}{2^{j_*+1}}\}$, where each of these numbers are as obtained in \descrefnormal{step4covdeg}{Step 4}. This ensures all the tail alternatives are satisfied. 

\descitemnormal{Step 7:}{step7deg} From the previous steps, we see that the reduction of oscillatin is obtained in $B_{R_1} \times (-R_1^{sp},0)$, where $R_1 \leq \tfrac12 c_oR_j$ and $R_1^{sp} \leq \min\{\tfrac14 (c_oR_j)^{sp}, \tfrac{\nu}{4}(\tfrac14c_oR_j)^{sp}\}$. We need $\mcq_{j+1} \subset B_{R_1} \times (-R_1^{sp},0)$ which holds provided we take $\mathbf{C}_o^{-1} \leq \min\{\tfrac12 c_o, 4^{-\frac{1}{sp}}c_o, (\nu/4)^{\frac{1}{sp}}c_o\}$. 
 Then the previous conclusions shows that $\essosc_{\mcq_{j+1}} u \leq (1-\sigma)\essosc_{\mcq_{j}} u$ where $\sigma  = \min\{\bar\eta,\eta_1,\bar\eta_2,\tfrac{\varepsilon}{2^{j_*+1}}\}$. Further choosing $\beta \leq \log_{\mathbf{C}_o}(1-\sigma)$ implies $(1-\sigma) \leq \mathbf{C}_o^{-\beta}$ which gives
 \[
 \essosc_{\mcq_{j+1}} u \leq (1-\sigma)\essosc_{\mcq_{j}} u = (1-\sigma) \mathbf{C}_o^{-\beta j}L \leq \mathbf{C}_o^{-\beta (j+1)}L.
 \]
 
 \descitemnormal{Step 8:}{step8deg} Let us assume there is a number $i_o \geq j_o$ such that if we denote $\bsmu^-=m_{i_o}$, $\bsmu^+ = M_{i_o}$ and $\bsom = \bsmu^+ - \bsmu^-$, then \tlcref{Eq:Hp-main2} holds. This will correspond to the reduction of oscillation away from zero.   In \cref{rmkdeg5.15}, we will take $\mreta = \sigma$ where $\sigma$ is from \descrefnormal{step7deg}{Step 7}.  Unlike in \cref{section6}, in subsequent calculations, we will keep track of $i_o$ and not suppress writing the subscript. 
 
 With this notation, we have $m_{i_o} \leq u \leq M_{i_o}$ and without loss of generality, we consider the case $m_{i_o} > \tau \mathbf{C}_o^{-\beta i_o}L$ holds which is the first case of \tlcref{Eq:Hp-main2} (note that $M_{i_o} < -\tau\mathbf{C}_o^{-\beta i_o}L$ can be handled analogously). Then we take $v := \lbr \tfrac{u}{m_{i_o}}\rbr^{p-1}$ and we have the bounds $m_{i_o} \leq u \leq \lbr \tfrac{1+\tau}{\tau}\rbr m_{i_o}$. We need to obtain reduction of oscillation in the cylinder $\mcq_{i_o} = B_{R_{i_o}} \times (-R_{i_o}^{sp},0)$.
 
 \item[Step 9:] In order to apply the conclusion of \cref{degconstantsscaled}, we need to ensure that the cylinder $\mcq_{c_o\varrho}^{\bar{A}} \subset \mcq_{i_o}$ for some $c_o, \varrho$ and $\bar{A} = \lbr \tfrac{\bsom_v}{\mathbf{a}}\rbr^{2-p}$. Let us take $\varrho = R_{i_o}$ and $c_o \leq \lbr \tfrac{\mathfrak{C}_1}{\mathbf{a}}\rbr^{p-2}$, where $\mathbf{a}$ is from \cref{degconstantsscaled} and $\mathfrak{C}_1$ is from \cref{boundsonomega}, then we have $\mcq_{c_o\varrho}^{\bar{A}} \subset \mcq_{i_o}$. 
 \begin{figure}[ht]
 	\begin{center}
 		\begin{tikzpicture}[line cap=round,line join=round,>=latex,scale=0.5]
 			\coordinate  (O) at (0,0);
 			
 						\draw[thick, draw=teal,pattern=north east lines, pattern color=teal, opacity=0.1] (-1.5,5.5) rectangle (1.5,-2);
 			\draw[thick, draw=blue, dashed] (-3,-3.5) -- (-3,5.5);
 			\draw[thick, draw=blue, dashed] (3,-3.5) -- (3,5.5);
 			\draw[draw=teal, dotted, ->] (1.5,-2) -- (3,-2);
 			\draw[thick, draw=black, dotted, ->] (3,-3) -- (-3,-3);
 			
 			\draw[draw=teal, dashed,<->] (-1.5,6) -- (1.5,6);
 			\draw[draw=blue,thick] (-3,-3.5) -- (3,-3.5);
 			
 			\draw[draw=black,thick] (-3,5.5) -- (3,5.5);
 			\node  [anchor=west] at (3,5.5) {\scriptsize \textcolor{black}{$t=0$}};
 			\draw[draw=blue, dashed, <->] (-3,-4) -- (3,-4);
 			\node  at (0,6.5) {\color{teal}{\scriptsize $2c_oR_{i_o}$}};
 			\node  at (0,2) {\scriptsize \color{teal}{\bf $\mcq_{c_oR_{i_o}}^{\bar{A}}$}};
 			\node  at (0,-4.5) {\scriptsize {\color{blue}$2R_{i_o}$}};
 			\node [anchor=west] at (3,-2) {\color{teal}{\scriptsize $t=-\lbr \frac{\bsom_v}{\mathbf{a}}\rbr^{2-p} (c_oR_{i_o})^{sp}$}};
 				\node [anchor=west] at (-10.1,-3) {\color{black}{\scriptsize $t=-\lbr \frac{\mathfrak{C}_1}{\mathbf{a}}\rbr^{2-p} (c_oR_{i_o})^{sp}$}};
 			\node [anchor=west] at (3,2) {\color{blue}{\scriptsize $\mcq_{i_o}$}};
 			\node [anchor=west] at (3,-3.5) {\scriptsize ${\color{blue}t=- R_{i_o}^{sp}}$};
 		\end{tikzpicture}
 	\end{center}
 	\caption{Reduction of Oscillation away from zero}
 	\label{figdegcov1}
 \end{figure}
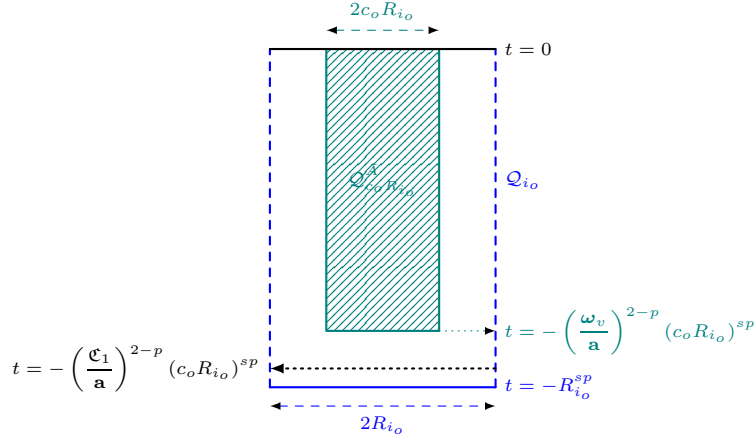

\item[Step 10:] In order to apply \cref{propdegaway}, we need to obtain  a suitable tail decay estimate, which proceeds as follows: 
\begin{equation*}
	\begin{array}{rcl}
		\tail((u-\bsmu^{\pm})_{\pm}, c_oR_{i_o}, 0 (-\bar{A} (c_oR_{i_o})^{sp},0))& = & (c_oR_{i_o})^{sp} \esssup_{(-\bar{A} (c_oR_{i_o})^{sp},0)} \int_{\RR^n\setminus B_{i_o}} \frac{(u-\bsmu^{\pm})_{\pm}^{p-1}}{|x|^{n+sp}}\,dx \\
		&& + (c_oR_{i_o})^{sp} \esssup_{(-\bar{A} (c_oR_{i_o})^{sp},0)} \int_{B_{i_o}\setminus B_{c_oR_{i_o}}} \frac{(u-\bsmu^{\pm})_{\pm}^{p-1}}{|x|^{n+sp}}\,dx\\
		& \leq & c_o^{sp} \tail((u-\bsmu^{\pm})_{\pm}, R_{i_o}, 0, (-R_{i_o}^{sp},0)) \\
		&& +  (c_oR_{i_o})^{sp} \esssup_{(-\bar{A} (c_oR_{i_o})^{sp},0)} \int_{B_{R_{i_o}}\setminus B_{c_oR_{i_o}}} \frac{(u-\bsmu^{\pm})_{\pm}^{p-1}}{|x|^{n+sp}}\,dx\\
		& \leq & c_o^{sp}\mathbf{C}_o^{-\beta i_o(p-1)} L^{p-1}\kappa^{p-1},
	\end{array}
\end{equation*}
since $\bsmu^- = m_{i_o} \leq u \leq M_{i_o} = \bsmu^+$ on $\mcq_{i_o}$ and to estimate $\tail((u-\bsmu^{\pm})_{\pm}, R_{i_o}, 0, (-R_{i_o}^{sp},0))$, we applied the calculations from \descrefnormal{step5deg}{Step 5}. We choose $c_o$ small such that $c_o^{sp} \kappa^{p-1} \leq \mathbf{d}_t$ where $\mathbf{d}_t$ is from \cref{propdegaway}.  

\descitemnormal{Step 11:}{step11deg} Hence we can apply \cref{propdegaway}, in the case \descrefnormal{conc1deg}{Conclusion 1}, we take $M_{i_o+1} = M_{i_o}$ and $m_{i_o+1} = M_{i_o} - \mathbf{C}_o^{-\beta(j_o+1)}L$ with $\beta \leq \log_{\mathbf{C}_o}\frac{1}{1-\frac12 \varepsilon_1}$ and in the case \descrefnormal{conc2deg}{Conclusion 2}, we take $m_{i_o+1}=  m_{i_o}$ and $M_{i_o+1} = m_{i_o} + \mathbf{C}_o^{-\beta(j_o+1)}L$ with $\beta \leq \log_{\mathbf{C}_o}\frac{1}{1-\frac12\bar\sigma_2\bar\varepsilon_1\varepsilon_4}$.  With notation as in \cref{propdegaway}, we further restrict $\mathbf{C}_o$ to satisfy 
\begin{equation}\label{controlsizedeg}
\mathbf{C}_o^{-1} \leq \tfrac14 c_o \txt{and} \mathbf{C}_o^{-sp} \leq \mathbf{d} \mathfrak{C}_2^{2-p} (\tfrac14 c_o)^{sp} \leq \mathbf{d} \bsom_v^{2-p} (\tfrac14 c_o)^{sp},
\end{equation}
where $\mathfrak{C}_2$ is from \cref{boundsonomega} and $\mathbf{d}$ is from \cref{propdegaway}. With this restriction, we see that $\mcq_{i_o+1} \subset B_{c_oR_{i_o}} \times (-\mathbf{d} \bsmu_v^{2-p} (\tfrac14c_oR_{i_o})^{sp},0)$ and hence the required reduction of oscillation follows. 

\descitemnormal{Step 12:}{step12deg} From $i_o$ onwards, we are always in the away from zero case and  there exists $\alpha \in (0,1)$ depending only on data, such that  for any $\sigma \in (0,1)$, there holds:
\[
\essosc_{\sigma\mcq_{i_o}} v \leq \mathbf{C}\sigma^{\alpha} \essosc_{\mcq_{i_o}}v.
\]
In particular, making use of \cref{vwdef5.2} and \cref{alg_lem}, we get
\[
\essosc_{\sigma \mcq_{i_o}} u \leq \mathbf{C}_{\data{}} \sigma^{\alpha}\lbr \frac{1+\tau}{\tau}\rbr^p \mathbf{C}_o^{-\beta i_o}L.
\]

We note that the $\mathfrak{C}_2$ from \cref{boundsonomega} is used for controlling the size of the cylinders in \cref{controlsizedeg} and not $\mathfrak{C}_1$. As a consequence, we can iterate this reduction of oscillation. 
\item[Step 13:] Combining the calculations from \descrefnormal{step12deg}{Step 12}, for any $\varrho < R$ and possibly smaller $\alpha$ (again depending on data), we get
\[
\essosc_{B_{\varrho}\times (-\varrho^{sp},0)}u \leq \mathbf{C}_{\data{}} \lbr \frac{\varrho}{R}\rbr^{\alpha} L. 
\]

\end{description}

\section{H\"older regularity for nonlocal doubly nonlinear equations - singular case}\label{section11}

The reduction of oscillation in the case $p<2$ follows analogously as the degenerate case and we leave the details to the interested readers.

\section*{References} 

\end{document}